\documentclass[11pt]{article}
 \usepackage[utf8]{inputenc}
 \usepackage[T1]{fontenc}
\usepackage[english]{babel} 

\usepackage{amsmath, amsfonts, amssymb, amstext, amsthm} 
 \usepackage{stmaryrd}
 \usepackage{graphicx, subfigure} 
 \usepackage{color} 
 \usepackage{enumitem}

 \usepackage[colorlinks=true,bookmarks=true,citecolor=blue]{hyperref}

\def\R{\mathbb{R}}
\def\N{\mathbb{N}}
\def\e{\varepsilon}
\def\la{\lambda}
\def\vp{\varphi}

\numberwithin{equation}{section}

\newtheorem{Theorem}{Theorem}[section]
\newtheorem{Lemma}[Theorem]{Lemma}

\newtheorem{theorem}[Theorem]{Theorem}
\newtheorem{lemma}[Theorem]{Lemma}
\newtheorem{proposition}[Theorem]{Proposition}
\newtheorem{corollary}[Theorem]{Corollary}
\newtheorem{remark}[Theorem]{Remark}

\newtheorem{prop}[Theorem]{Proposition}

\newtheorem{cor}[Theorem]{Corollary}

\DeclareMathOperator{\dist}{dist}

\newcommand{\cI}{{\mathcal I}}

\newcommand{\cM}{{\mathcal M}}

\newcommand{\cO}{{\mathcal O}}
\newcommand{\cP}{{\mathcal P}}

 \newcommand{\cV}{{\mathcal V}}

\newcommand{\al}{\alpha}
\newcommand{\be}{\beta}
\newcommand{\ga}{\gamma}

\newcommand{\de}{\delta}
\newcommand{\ep}{\epsilon}

\newcommand{\om}{\omega}

\newcommand{\Om}{\Omega}
\newcommand{\La}{\Lambda}

\newcommand{\codt}{\cdot} 

\newcommand{\ol}{\overline}
\newcommand{\ul}{\underline}

\newcommand{\upto}{\nearrow}

\newcommand{\Lao}{\Lambda_{out}}
\newcommand{\Sii}{\Sigma_{in}}
\newcommand{\hp}{\hat p}
\newcommand{\hq}{\hat q}

\title{Large-time behavior of solutions of parabolic 
  equations on the real line with convergent initial data II:
 equal limits at infinity
} 
\author{Antoine Pauthier\thanks{corresponding author (\texttt{apauthie@uni-bremen.de}). Present address: 
	University of Bremen, Bibliothekstr. 5, MZH4130, 28359 Bremen, Germany. Telephone +49 421 218 63741} 
 ~and
  Peter Pol\'{a}\v{c}ik\footnote{Supported in part by the NSF
    Grant DMS-1856491}  \\
{\small School of Mathematics, University of Minnesota}\\
{\small Minneapolis, MN 55455}
}
 \date{\today}
\begin{document}
\maketitle
\begin{abstract}
  We continue our study of bounded solutions of the semilinear
  parabolic equation $u_t=u_{xx}+f(u)$ on the real line, where $f$ is
  a locally Lipschitz function on $\R.$ Assuming that the initial
  value $u_0=u(\cdot,0)$ of the solution has finite limits
  $\theta^\pm$ as $x\to\pm\infty$, our goal is to describe the
  asymptotic behavior of $u(x,t)$ as $t\to\infty$.  In a prior work,
  we showed that if the two limits are distinct, then the solution is
  quasiconvergent, that is, all its locally uniform limit profiles as
  $t\to\infty$ are steady states. It is known that this result is not
  valid in general if the limits are equal: $\theta^\pm=\theta_0$. In
  the present paper, we have a closer look at the equal-limits case.
  Under minor non-degeneracy assumptions on the nonlinearity, we show
  that the solution is quasiconvergent if either $f(\theta_0)\ne0$, or
  $f(\theta_0)=0$ and $\theta_0$ is a stable equilibrium of the
  equation $\dot \xi=f(\xi)$. If $f(\theta_0)=0$ and $\theta_0$ is an
  unstable equilibrium of the equation $\dot \xi=f(\xi)$, we also
  prove some quasiconvergence theorem making (necessarily) additional
  assumptions on $u_0$.  A major ingredient of our proofs of the
  quasiconvergence theorems---and a result of independent interest---is
  the classification of entire solutions of a certain type
  as steady states and
  heteroclinic connections between two disjoint sets of steady states.
\end{abstract}


\footnotesize
\tableofcontents
\normalsize
\section{Introduction and main results}
\subsection{Background}
Consider the Cauchy problem
\begin{align}
  u_t=u_{xx}+f(u), & \qquad x\in\R,\ t>0, \label{eq:1}\\
  u(x,0)=u_0(x), & \qquad x\in\R, \label{ic1}
\end{align}
where $f$ is a locally Lipschitz function on $\R$ and
$u_0\in C_b(\R):=C(\R)\cap L^\infty(\R).$ We denote by
$u(\cdot,t,u_0),$ or simply $u(\cdot,t)$ if there is no danger of
confusion, the unique classical solution of (\ref{eq:1}), (\ref{ic1})
and by $T(u_0)\in(0,+\infty]$ its maximal existence time. If $u$ is
bounded on $\R\times[0,T(u_0))$, then necessarily $T(u_0)=+\infty,$
that is, the solution is global. In this paper, we are concerned with
the behavior of bounded solutions as $t\to\infty$. A basic question we
specifically want to address is whether, or to what extent, the
large-time behavior of bounded solutions is governed by steady states
of \eqref{eq:1}.

If the initial data $u_0$ admits limits as $x\to\pm\infty,$ then for
all time $t>0,$ the solution $u(\cdot,t)$ of \eqref{eq:1}, \eqref{ic1}
admits limits as $x\to\pm\infty.$ In other words, the function space
\begin{equation}\label{spacelimit}
  \mathcal{V}:=\left\{ v\in C_b(\R):\textrm{ the limits }
    v(-\infty),\,v(+\infty)\in \R \textrm{ exist}\right\}
\end{equation}
is invariant for (\ref{eq:1}). Continuing our study initiated in
\cite{Pauthier_Polacik1}, we examine the large time behavior of
bounded solutions in $\cV$. More specifically, we are interested in
the behavior of the solutions in bounded---albeit arbitrarily
large---spatial intervals, as $t\to\infty$. For that purpose, we
introduce the $\omega-$limit set of a bounded solution $u$, denoted by
$\omega(u)$ or $\omega(u_0)$ with $u_0=u(\codt,0)$, as follows:
\begin{equation}\label{defomega}
  \omega(u):=\left\{ \vp\in L^\infty(\R),\ u(\cdot,t_n)\to\vp \textrm{
      for some sequence }t_n\to\infty\right\}. 
\end{equation}
Here the convergence is in the topology of $L^\infty_{loc}(\R)$, that
is, the locally uniform convergence. By standard parabolic estimates,
the trajectory $\{ u(\cdot,t),\ t\geq1\}$ of a bounded solution is
relatively compact in $L^\infty_{loc}(\R).$ This implies that
$\omega(u)$ is nonempty, compact, and connected $L^\infty_{loc}(\R)$,
and it attracts the solution in (the metric space)
$L^\infty_{loc}(\R)$:
$$
\textrm{dist}_{L^\infty_{loc}(\R)}\left(
  u(\cdot,t),\omega(u)\right)\underset{t\to\infty}{\longrightarrow}0.
$$
If the $\omega-$limit set reduces to a single element $\varphi$, then
$u$ is \textit{convergent:} $u(\cdot,t)\to\vp$ in $L^\infty_{loc}(\R)$
as $t\to\infty$.  Necessarily, $\vp$ is a steady state of \eqref{eq:1}.
If all functions $\vp\in\omega(u)$ are steady states of (\ref{eq:1}),
the solution $u$ is said to be \textit{quasiconvergent}.

Convergence and quasiconvergence both express a relatively tame
character of the solution in question, entailing in particular the
property that $u_t(\cdot,t)$ approaches zero locally uniformly on $\R$
as $t\to \infty$. In some cases, quasiconvergence can be established
by means of energy estimates when bounded solutions in suitable energy
spaces are considered (see \cite{Feireisl_NoDEA97}, for example).
However, when no particular rate of approach of $u_0(x)$ to its limits
at $x=\pm\infty$ is assumed, energy techniques typically do not
apply. Nonetheless, several quasiconvergence results are available for
solutions in $\cV$ (see \cite{P:quasiconv-overview} for a general
overview).  These include quasiconvergence theorems of
\cite{Matano_Polacik_CPDE16} for nonnegative functions $u_0$ with
$u_0(\pm\infty)=0$ when $f(0)=0$---convergence theorems are available
under additional conditions on $u_0\ge 0$, see
\cite{Du_Matano,p-Du,Matano_Polacik_CPDE16}; or for generic $f$, see
\cite{p-Ma:1d-p2}---and a theorem of \cite{Polacik_terrasse} for
functions $u_0\in \cV$ satisfying $u_0(-\infty)>u_0> u_0(\infty)$ or
$u_0(-\infty)<u_0< u_0(\infty)$. An improvement over the latter
quasiconvergence result was achieved in \cite{Pauthier_Polacik1},
where we proved that the condition $u_0(-\infty)\neq u_0(+\infty)$
alone, with no relations involving $u_0(x)$ for $x\in\R$, is already
sufficient for the quasiconvergence of the solution if it is bounded.

It is also known that the $\om$-limit set of a bounded solution always
contains at least one equilibrium \cite{Gallay-S,Gallay-S2}.  However,
bounded solutions, even those in $\cV$, are not always quasiconvergent
(see \cite{P:examples,P:unbal}).  Moreover, as shown in
\cite{P:unbal}, non-quasiconvergent solutions occur in \eqref{eq:1} in
a persistent manner: they exist whenever $f$ is a $C^1$ nonlinearity
satisfying certain robust conditions (cp. \eqref{eq:47} below).  In
view of these results, the  
following question arises naturally. Given $f$, can one characterize
in some way the initial data $u_0\in\cV$ which yield quasiconvergence
solutions? Our previous work \cite{Pauthier_Polacik1} was our first
step in addressing this question: we proved the quasiconvergence in
the distinct-limits case: $u_0(-\infty)\neq u_0(+\infty)$. In the
present paper, we
consider the case when the limits are equal:
$u_0(-\infty)= u_0(+\infty):=\theta_0$. We assume the nonlinearity $f$
to be fixed and satisfy minor nondegeneracy conditions (see the next
section).

In our first main theorem, Theorem \ref{thm:1}, we show that if
$f(\theta_0)\ne 0$, or if $f(\theta_0)=0$ and $\theta_0$ is a stable
equilibrium of the equation $\dot \xi=f(\xi)$, then the solution $u$
of (\ref{eq:1}), (\ref{ic1}) is quasiconvergent if bounded. In the
examples of non-quasiconvergent solutions with
$u_0(\pm\infty)=\theta_0$, as given in \cite{P:examples,P:unbal},
$\theta_0$ is an unstable equilibrium of $\dot \xi=f(\xi)$. Thus, our
theorem shows that this is in fact necessary.  Other two results,
Theorems \ref{thm:2} and \ref{thm:3}, 
give sufficient conditions for the quasiconvergence of the solution in
the case that $\theta_0$ is an unstable equilibrium of the equation
$\dot \xi=f(\xi)$. A special case of the sufficient condition of
Theorem \ref{thm:3} is the condition 
that $u_0-\theta_0$ has compact support and only finitely many sign
changes. Theorem \ref{thm:2} has a somewhat surprising result
saying that any element 
 $\varphi$ of $\om(u)$ whose range is not included in the minimal
 bistable interval containing $\theta_0$ is necessarily a steady
 state.  

 We give formal statements of
these results in Subsection \ref{qcthms}, after first formulating our
hypotheses in Subsection \ref{standing-hyp}. In Subsection
\ref{outline}, we give an outline of our strategy of proving the
quasiconvergence theorems.

A quasiconvergence result closely related to our
Theorem \ref{thm:1} has   recently been proved by Risler. In
\cite{Risler_1}, he considers the Cauchy problem for
gradient  reaction-diffusion systems on $\R$,
where the initial data are assumed to converge at $\pm\infty$
to stable homogeneous steady states contained in the same
level of the potential function. Under 
certain generic conditions on the corresponding
stationary system,  he proves that bounded solutions of such Cauchy
problems 
are quasiconvergent in a localized
topology (in the companion paper \cite{Risler:terrace-1d}, the global
shape of such solutions at large times is investigated).
His approach is variational, which has an advantage that it applies to
gradient systems, as opposed to our techniques based on the zero
number, which only apply to scalar equations. In the scalar case,
our method seems to have some advantages. For example, it allows us 
to treat to some extent the case when the limit of the initial data at
$\pm\infty$ is unstable. Also, in principle, the method can be used
under much less stringent nondegeneracy conditions 
(cp. Subsection \ref{standing-hyp}) and, we believe, will eventually
allow us to  dispose of the nondegeneracy conditions altogether.

A key ingredient of our method of proof of the quasiconvergence
theorems is a classification result for a certain type of entire
solutions of \eqref{eq:1}, that is, solutions defined for all
$t\in\R$. Roughly speaking, the result shows that the entire
solutions are either steady states or connections between two disjoint
sets of steady states of \eqref{eq:1} (see Sections \ref{outline} and
\ref{sec:entire} for details).  Entire solutions play an important
role in qualitative analysis of solutions of parabolic equations, as
it can usually be proved that the large-time behavior
of bounded solutions is governed by entire solutions. In our setting,
for example, the $\om$-limit sets---or their generalized versions, as
defined in Section \ref{sec:inv}---of bounded solutions of \eqref{eq:1}
consist of orbits of entire solutions.  Entire solutions of
\eqref{eq:1} have been extensively studied and many different types of
such solutions have been found. These include, in addition to steady
states, spatially periodic heteroclinic orbits between steady states
(see \cite{Fiedler-B:neumann, Fiedler-R:conn, P:entire} for
example), traveling waves and many types of ``nonlinear
superpositions'' of traveling waves and other entire solutions (see
\cite{Chen-G:ex, Chen-G-N, Chen-G-N-Y, Guo-M-ent, Hamel-N-ent,
  Morita-N:entire, Morita-N:exposition} and references therein), as
well entire solutions involving colliding pulses \cite{p-Ma:entire}.
Unlike for equations on bounded intervals where rather general
classification results for entire solutions are available (see
\cite{Brunovsky-F:conn, Fiedler-R:conn, Wolfrum:jde} and references
therein), no such general classification is currently in sight for the vast
variety of entire solutions of \eqref{eq:1}. Our result classifying
certain entire solutions as connections between two sets of steady
state is a modest contribution in this area, exploring the asymptotic
behavior of entire solutions as $t\to\pm\infty$ in the topology of
$L^\infty_{loc}(\R)$.

\subsection{Standing hypotheses} \label{standing-hyp} As above, $f$ is
a locally Lipschitz function.  We also assume the following
nondegeneracy condition:
\begin{description}
\item[\bf(ND)] For each $\gamma\in f^{-1}\{0\}$, $f$ is of class $C^1$
 in a neighborhood of $\gamma$ and $f'(\gamma)\neq0.$
\end{description}

Our theorems can be proved under weaker conditions. To give an example
of how (ND) can be relaxed, set
\begin{equation}
  \label{eq:3}
  \displaystyle F(v):=\int_0^v f(s)ds,
\end{equation}
so  zeros of $f$ are critical points of $F$. The following
nondegeneracy conditions can be considered in place of (ND).
\begin{description}
\item[\rm (ND1)] Each $\gamma\in f^{-1}\{0\}$ is locally a point of
  strict maximum or strict minimum for $F.$
\item[\rm (ND2)] If $\gamma_1<\gamma_2$ are two consecutive
  local-maximum points of $F$ and $F(\gamma_1)=F(\gamma_2),$ then
  $\gamma_1$, $\gamma_2$ are nondegenerate critical points of $F$: $f$
  is of class $C^1$ in a neighborhood of $\gamma_{1,2}$ and
  $f'(\gamma_{1,2})<0.$
\end{description}

Relaxing (ND) to (ND1), (ND2) does not pose  major problems in the
proof of our results, but it would obscure the exposition a bit and
would require modification of some standard results we refer to.
Thus we decided to just work with (ND). All these nondegeneracy
conditions are just technical and we believe our theorems can be
proved by the same general method without them.  Clearly, condition
(ND) is generic with respect to ``reasonable'' topologies.  Note,
however, that we allow some nongeneric situations, for example, the
existence of two consecutive local-maximum points of $F$ at which $F$ takes
the same value.  The nondegeneracy condition constrains considerably
the complexity of possible phase portraits associated with equation
for the steady states of \eqref{eq:1}:
\begin{equation}\label{eq:steady}
  u_{xx}+f(u)=0,\qquad x\in\R.
\end{equation}
This is mainly how the nondegeneracy condition is useful in this
paper.

We will make another assumption on the nonlinearity.  It concerns the
behavior of $f(u)$ for large values of $|u|$ and it can be assumed
with no loss of generality. Indeed, our main quasiconvergence theorems
deal with individual bounded solutions only.  Thus we can modify $f$
freely outside the range of the given solution with no effect on the
validity of the theorems.  It will be convenient to assume that
\begin{description}
\item[\bf(MF)]$f$ is globally Lipschitz and there is $\kappa>0$ such
  that for all $s$ with $|s|>\kappa$ one has $ f(s)={s}/{2}.$
\end{description}

Hypotheses (ND) and (MF) are our \emph{standing hypotheses on $f$}.

Each zero of $f$ is of course an equilibrium of the equation
\begin{equation}\label{eq:ODE}
  \dot{\xi}(t)=f(\xi). 
\end{equation}
Hypothesis (ND) implies in particular that any such equilibrium is
either unstable from above and below, or asymptotically stable (this
property would also be implied by (ND1)).
 
As mentioned above, in this paper we take $u_0\in \cV$, assuming that
its limits at $\pm\infty$ are equal. Without loss of generality, we
assume the limits to be equal to zero:
\begin{equation}\label{localizedu0}
  u_0\in\mathcal{V},\qquad u_0(-\infty)=u_0(+\infty)=0.
\end{equation}
We distinguish the following two cases:
\begin{itemize}
\item[\bf(S)] Either $f(0)\neq 0$, or $f(0)=0$ and 0 is a stable
  equilibrium for (\ref{eq:ODE}).
\item[\bf(U)] $f(0)=0$ and 0 is an unstable equilibrium for
  (\ref{eq:ODE}).
\end{itemize}

\subsection{Quasiconvergence theorems} \label{qcthms}

If (S) holds, we have a general quasiconvergence theorem:
\begin{theorem}\label{thm:1}
  Assume that {\rm (S)} holds, and let $u_0$ be as in
  \eqref{localizedu0}. Then if the solution $u$ of \eqref{eq:1},
  \eqref{ic1} is bounded, it is quasiconvergent: $\omega(u)$ consists
  entirely of steady states of \eqref{eq:1}.
\end{theorem}

\begin{remark}
  \label{rm:noper}
  {\rm We will show, more precisely, that any element $\varphi$ of
    $\om(u)$ is a constant steady state, or a ground state at some
    level $\xi\in f^{-1}\{0\}$, or a standing wave of \eqref{eq:1}. See
    Section \ref{sec:22} for a description of the structure of steady
    states of \eqref{eq:1} and the meaning of the terminology used
    here. We will also show that there is a single chain in the phase
    plane of $\varphi_{xx}+f(\varphi)=0$ containing the trajectories
    of all steady states $\varphi\in\om(u)$ (the definition of a chain
    is also given in Section \ref{sec:22}). The same remarks apply to
    Theorem \ref{thm:3} below. }
\end{remark}

If (U) holds, then, as already noted in the introduction, a similar
quasiconvergence does not hold in general: the references
\cite{P:examples,P:unbal} provide examples of bounded solutions of
\eqref{eq:1}, \eqref{ic1} with $u_0(\pm \infty)=0$ which are not
quasiconvergent. More specifically, such solutions exist whenever
$f\in C^1$ and $0$ belongs to a \emph{bistable} interval of $f$: there are
$\ga_1, \ga_2\in \R$ such that
\begin{equation}
  \label{eq:47}
  \ga_1<0<\ga_2,\quad f(\ga_1)=f(\ga_2)=0, \quad
  f'(\ga_1),f'(\ga_2)<0,\quad \text{and }f\ne 0\text{ in
  }(\ga_1,0)\cup(0,\ga_2). 
\end{equation}
Whether the bistable nonlinearity $f$ is\emph{ balanced in $(\ga_1,\ga_2)$:}
$F(\ga_1)=F(\ga_2)$, or \emph{unbalanced:}
$F(\ga_1)\ne F(\ga_2)$, there always exists a continuous function
$u_0$ such that $u_0(\pm \infty)=0$, $\ga_1\le u_0\le \ga_2$ and the
solution $u$ of \eqref{eq:1}, \eqref{ic1} is not quasiconvergent.
Obviously, all limit profiles, stationary or not,
of the solution $u$ take also values between $\ga_1$, $\ga_2$.
One could naturally speculate that when the initial data are not
constrained by the assumption $\ga_1\le u_0\le \ga_2$, the behavior of the
corresponding solutions
can only get more complicated, with some nonstationary
limit profiles possibly occurring outside the interval
$[\ga_1,\ga_2]$. Surprisingly perhaps, this
turns out not to be the case.  In our next theorem, we show that 
any limit profile whose range is not contained in
$(\ga_1,\ga_2)$ is a steady state. Thus it is really the bistable
interval $[\ga_1,\ga_2]$ which is ``responsible'' for the
nonquasiconvergence of the solutions with $u_0(\pm\infty)=0$,
regardless of whether the range of $u_0$ is contained in
$[\ga_1,\ga_2]$ or not. 

Note that if (U) holds and $\ga_1$, $\ga_2$ are the zeros of $f$
immediately preceding and immediately succeeding 0, respectively,
assuming they exist, then the relations in \eqref{eq:47} are
satisfied.
\begin{theorem}
  \label{thm:2}
  Assume that {\rm (U)} and 
  \eqref{eq:47} hold.  Assume further that $u_0$ is as in
  \eqref{localizedu0} and the solution $u$ of \eqref{eq:1}, \eqref{ic1}
  is bounded. Then any function
  $\varphi\in \om(u)$ whose range is not contained in
  the interval $(\ga_1,\ga_2)$ is a  steady state of \eqref{eq:1}.
\end{theorem}
A stronger version of this result will be given  in Theorem
\ref{thm:4} after some  needed terminology has been introduced.
Obviously, Theorem \ref{thm:2} implies that the solution $u$ is
quasiconvergent if no function $\varphi\in\om(u_0)$ has its range in
$(\ga_1,\ga_2)$.

Another aspect of the examples of non-quasiconvergent solutions
given in \cite{P:examples,P:unbal} is that the
solutions $u$ found there are always
highly oscillatory in space: for all $t>0$ the function
$u(\cdot,t)$ has infinitely many
critical points and infinitely many zeros. This raises another natural
question whether, in the case (U), spatially nonoscillatory solutions are
always quasiconvergent. Here, by a \emph{spatially nonoscillatory
  solution} we mean a 
solution satisfying the following condition
\begin{description}
\item[\bf (NC)] There is $t>0$ such that $u(\cdot,t)$ has only
  finitely many critical points.
\end{description}
A sufficient condition for (NC) in terms of $u_0$ is that there exist
$a<b$ such that the function $u_0$ is monotone and nonconstant on each
of the intervals $(-\infty,a)$, $(b,\infty)$.  For if this holds, then
one shows easily, using the comparison principle, that $u_x(x,t)\ne 0$
for all $x\in \R$ with $|x|\approx \infty$ and all sufficiently small
positive times $t$. Consequently, by standard zero number results
(cp. Section \ref{sub:zero}), $u_x(\cdot,t)$ has only a finite number
of zeros for all $t>0$.

Presently, we are able to prove the quasiconvergence  assuming
that (NC) holds together with the following technical condition:
\begin{description}
\item{\bf(R)} There are  sequences $a_n\to-\infty$,
  $b_n\to\infty$ such that the following holds.  For every
  $\la\in \{a_1,a_2,\dots\}\cup\{b_1,b_2,\dots\}$ there is $t\ge 0$
  such that the function
  $V_\la u(\cdot,t):=u(2\la-\codt,t)-u(\cdot,t)$ has only finitely
  many zeros.
\end{description}
Remark \ref{rm:suff-cond}(ii) below gives sufficient conditions for
(R) in terms of $u_0$.

\begin{theorem}\label{thm:3}
  Assume that {\rm (U)} holds together with {\rm (NC)} and {\rm (R)}, and let $u_0$ be
  as in \eqref{localizedu0}.  If the solution $u$ of \eqref{eq:1},
  \eqref{ic1} is bounded, it is quasiconvergent.
\end{theorem}

\begin{remark}\label{rm:suff-cond}
  {\rm (i) If (R) is strengthened so as to say that
    $V_\la u(\cdot,t):=u(2\la-\codt,t)-u(\cdot,t)$ has only finitely
    many zeros for \emph{every} $\la\in \R$, then the quasiconvergence
    theorem holds---without any extra condition like {\rm (NC)} on $u$ and
    without the nondegeneracy condition (ND) on $f$---due to a result
    of \cite{Pauthier_Polacik1} which we recall in Theorem
    \ref{thmPP1} below. This in particular applies when for some $t>0$
    the function $u(\codt,t)$ has an odd (finite) number of zeros, all
    of them simple.  Indeed, in this case $u(x,t)$ has opposite signs
    for $x\approx-\infty$ and $x\approx\infty$ and, consequently, for
    every $\la\in \R$ one has $V_\la u(x,t):=u(2\la-x,t)-u(x,t)\ne 0$
    if $|x|$ is large enough.  Since the zeros of $V_\la u(\codt,t)$
    are isolated (cp. Section \ref{sub:zero}), there are only finitely
    many of them.  Of course, it may easily happen for a function
    $\psi\in\cV$ with $\psi(\pm\infty)=0$ that $V_\la\psi$ has only
    finitely many zeros if $|\la|$ is sufficiently large, but has
    infinitely many of them if $|\la|$ is sufficiently small. An
    example is any continuous function such that
    \begin{equation*}
      \psi(x)=\left\{
        \begin{aligned}
          &Me^x\quad&&\text{for $x<-k$,}\\
          &e^{-x}(M+\sin x)\quad&&\text{for $x>k$,}
        \end{aligned}
      \right.
    \end{equation*}
    where $k>0$ and $M>2$.
    
    (ii) Conditions (NC) and (R) are both satisfied if $u_0$ has
    compact support $[c,d]$ and only finitely many zeros in $(c,d)$.
    More generally, they are satisfied if $u_0\equiv 0$ on
    $\R\setminus(c,d)$ and there is $\ep>0$ such that $u_0$ is
    monotone and nonconstant on each of the intervals $(c,c+\ep]$,
    $[d-\ep,d)$. Indeed, the validity of (NC) is verified in the
    remark following (NC).  To show the validity of (R), take any
    $\la>d$. Then, the assumption on $u_0$ implies that
    $V_\la u_0(2\la-c-\ep)\ne 0$ and in the whole interval
    $J:=[2\la-c-\ep,\infty)$ one has either $V_\la u_0\ge 0$ or
    $V_\la u_0\le 0$. The comparison principle applied to the function
    $V_\la u(x,t)$ (cp. Section \ref{sec:24}) shows that
    $V_\la u(x,t)\ne 0$ for all $x\in J$ and $t>0$.  This also implies
    that 
    $V_\la u(x,t)\ne 0$ for all $x\approx -\infty$, as the function
    $V_\la u(\codt,t)$ is odd about $x=\la.$ Consequently, as in the
    previous remark (i),  $V_\la u(\codt,t)$ has only finitely
    many zeros for all $t>0$.  Similarly one shows that
    $V_\la u(\codt,t)$ has only finitely many zeros if $\la<d$ (in
    fact, with a little more effort one can show this for any
    $\la\ne (c+d)/2)$). Variations of these arguments show that (NC)
    and (R) are satisfied if $u_0\equiv 0$ on an interval
    $(-\infty, c)$ and on an interval $(d,\infty)$ one has $u_0\ge 0$,
    $u_0\not\equiv 0$ (or $u_0\le 0$, $u_0\not\equiv 0$).  }
\end{remark}

\subsection{Entire solutions and chains}\label{outline}
Our strategy for proving the quasiconvergence theorems consists in
careful analysis of a certain type of entire solutions of
\eqref{eq:1}. By an entire solution we mean a solution $U(x,t)$ of
\eqref{eq:1} defined for all $t\in\R$ (and $x\in\R$). If is well known
(see Section \ref{sec:inv} for more details) that for any $\vp\in\omega(u)$
there exists an {entire solution} $U(x,t)$ of \eqref{eq:1} such that
$ U(\cdot,0)=\vp$ and $U(\cdot,t)\in\omega(u)$ for all $t\in\R$.

In analysis of such entire solutions we employ, as in several earlier
papers starting with \cite{P:ctw}, a geometric technique  involving
spatial trajectories of solutions of \eqref{eq:1}. This technique,
powered by properties of the zero number functional, often allows one
to gain insights into the behavior of solutions of equation
\eqref{eq:1} (whose trajectories are in an infinite dimensional space)
by examining their spatial trajectories, which are curves in
$\R^2$. Specifically for any $\varphi \in C^1(\R)$, we define
\begin{equation}
  \label{eq:straj}
  \tau(\vp):=\left\{ \left( \vp(x),\vp_x(x)\right):x\in\R\right\} 
\end{equation}
and refer to this set as the \textit{spatial trajectory (or orbit)} of
$\vp.$ If $Y\subset C^1(\R)$, $\tau(Y)\subset \R^2$ is the union of the spatial
trajectories of the functions in $Y$:
\begin{equation}
  \label{eq:2}
  \tau(Y):=\left\{ \left( \vp(x),\vp_x(x)\right):x\in\R,\,\varphi\in Y\right\}. 
\end{equation}
Note that if $\vp$ is a steady state of \eqref{eq:1}, then $\tau(\vp)$
is the usual trajectory of the solution $(\varphi,\varphi_x)$ of the
planar system
\begin{equation}\label{sysI}
  u_x=v,\qquad v_x=-f(u),
\end{equation}
associated with equation \eqref{eq:steady}.

When considering an entire solution $U$ with $U(\cdot,0)\in \om(u)$,
we want to constrain the locations that the spatial trajectories
$\tau(U(\cdot,t))$, $t\in\R$, can occupy in $\R^2$ relative to the
locations of the spatial trajectories of steady states of
\eqref{eq:1}. In this, the concept of a chain is crucial. By
definition, a \emph{chain} is any connected component of the set
$\R^2\setminus \cP_0$, where $\cP_0$ is the union of trajectories of
all nonstationary periodic solutions of \eqref{sysI}.  Any chain
consists of equilibria, homoclinic orbits, and, possibly, heteroclinic orbits of
\eqref{sysI} (see Section \ref{sec:22}). Our ultimate goal is to prove
that the spatial trajectories $\tau(U(\cdot,t))$, $t\in\R$, are all
contained in one chain. From this it follows, via a
unique-continuation type result (cp. Lemma \ref{le:2.7} below), that
$U$ is a steady state of \eqref{eq:1}, which proves that $\om(u)$
consists of steady states, as desired.

To achieve our goal, we first show that if the spatial trajectories
$\tau(U(\cdot,t))$, $t\in\R$, are not contained in one chain, then
none of them can intersect any chain; and there exist two distinct
chains $\Sigma_1$, $\Sigma_2$ such that
\begin{equation}
  \label{eq:chains}
  \tau\left(\alpha(U)\right)\subset\Sigma_1,\qquad
  \tau\left(\omega(U)\right)\subset\Sigma_2. 
\end{equation}
Here $\om(U)$, $\al(U)$ stand for the $\om$ and $\al$-limit sets of
$U$; $\om(U)$ is defined as in \eqref{defomega} and the definition of
$\al(U)$ is analogous, with $t_n\to\infty$ replaced by
$t_n\to-\infty$. We will also show that the set of all relevant
chains, namely, the chains  that can
possibly intersect $\tau(\om(u))$, is finite and ordered by a suitable
order relation, and that the chains in \eqref{eq:chains} always satisfy
$\Sigma_1<\Sigma_2$ in that relation. As a consequence, we obtain that
the sets
\begin{equation*}
  K:=\{\varphi\in\om(u):\tau(\varphi)\subset \Sigma\}, \quad
\end{equation*}
corresponding to the chains $\Sigma$ with
$\Sigma\cap \tau(\om(u))\ne \emptyset$ constitute a Morse decomposition
for the flow of \eqref{eq:1} in $\omega(u)$
(see \cite{Conley_78}). However, the existence of such a
Morse decomposition (with at least two Morse sets) contradicts
well-known recurrence properties of the flow in $\om(u)$ (cp.
\cite{Conley_78,Chen_Polacik1995,p-Fo:conv-asympt}).  This
contradiction shows that the spatial trajectories $\tau(U(\cdot,t))$,
$t\in\R$, must in fact be contained in one chain, as desired.

The detailed proof following the above scenario, which we give below,
is rather involved mainly because there are several different
possibilities as to how the chains $\Sigma_1$, $\Sigma_2$ in
\eqref{eq:chains} can look like and how the spatial trajectories
$\tau(U(\cdot,t))$, $t\in\R$, can fit into the structure of
$\Sigma_1$, $\Sigma_2$. Even
though the number of these possibilities is reduced considerably by
the nondegeneracy condition (ND), the  possibilities that still remain
require special considerations and arguments.

We wish to emphasize that we prove \eqref{eq:chains} for a large class of
entire solutions of \eqref{eq:1}, regardless of their containment in
the $\omega$-limit set of the solution $u$ of \eqref{eq:1},
\eqref{ic1}, for any $u_0$.  Accordingly, we have striven to make
Section \ref{sec:entire}, where the entire solutions are studied in
detail, completely independent from the other parts of the paper. In
particular, no reference is made in that section to the solution $u$
or its limit set $\om(u)$. Thus the results there can be viewed as a
contribution to the general understanding of entire solutions of
\eqref{eq:1}. Relations \eqref{eq:chains} can be interpreted as a
classification result, characterizing nonstationary entire solutions
as connections between two different sets of steady states. We refer
the reader to Section \ref{sec:entire} for more details.

The rest of the paper is organized as follows.  In Section
\ref{prelims}, we collect preliminary results on the zero number,
steady states of \eqref{eq:1}, entire solutions of \eqref{eq:1}
and their $\al$ and
$\om$-limit sets. We also recall there some results from
earlier paper that are repeatedly used in the proofs of our main theorems.
The proofs themselves comprise Sections \ref{sec:3} and \ref
{morse-dec}.  Section \ref{sec:entire} is devoted to the classification
of entire solutions, as mentioned above. 

\section{Preliminaries}\label{prelims}
In this section, we collect preliminary results and basic tools of our
analysis. We first recall some well known properties of the
zero-number functional and then examine trajectories of steady states
of \eqref{eq:1} in the phase plane, taking our standing hypotheses
(ND), (MF) into account.  Next we recall invariance properties of
various limit sets of bounded solutions of \eqref{eq:1}. Finally, in
Subsection \ref{sec:24} we state several important technical results
concerning bounded solutions of \eqref{eq:1}. 

\subsection{Zero number for linear parabolic equations}\label{sub:zero}
In this subsection, we consider solutions of a linear parabolic
equation
\begin{equation}\label{eqlin}
  v_t=v_{xx}+c(x,t)v,\qquad x\in\R,\ t\in\left( s,T\right),
\end{equation}
where $-\infty\leq s<T\leq \infty$ and $c$ is a bounded measurable
function.  Note that if $u$, $\bar u$ are bounded solutions of
\eqref{eq:1}, then their difference $v=u-\bar u$ satisfies
\eqref{eqlin} with a suitable function $c$.  Similarly, $v=u_x$ and
$v=u_t$ are solutions of such a linear equation. These facts are
frequently used below, often without notice.

For an interval $I=(a,b),$ with $-\infty\leq a < b\leq \infty,$ we
denote by $z_I(v(\cdot,t))$ the number, possibly infinite, of zeros
$x\in I$ of the function $x\mapsto v(x,t).$ If $I=\R$ we usually omit
the subscript $\R$:
$$
z(v(\cdot,t)):=z_\R(v(\cdot,t)).
$$
The following intersection-comparison principle holds (see
\cite{Angenent_Crelle88,Chen_MathAnn98}).
\begin{lemma}\label{lemzero}
  Let $v$ be a nontrivial solution of \eqref{eqlin} and
  $I=(a,b),$ with $-\infty\leq a < b\leq \infty.$ Assume that the
  following conditions are satisfied:
  \begin{itemize}
  \item if $b<\infty,$ then $v(b,t)\neq0$ for all
    $t\in\left( s,T\right),$
  \item if $a>-\infty,$ then $v(a,t)\neq0$ for all
    $t\in\left( s,T\right).$
  \end{itemize}
  Then the following statements hold true.
  \begin{enumerate}
  \item[(i)] For each $t\in\left( s,T\right),$ all zeros of
    $v(\cdot,t)$ are isolated. In particular, if $I$ is bounded, then
    $z_I(v(\cdot,t))<\infty$ for all $t\in\left( s,T\right).$
  \item[(ii)] The function $t\mapsto z_I(v(\cdot,t))$ is monotone
    non-increasing on $(s,T)$ with values in
    $\N\cup\{0\}\cup\{\infty\}.$
  \item[(iii)] If for some $t_0\in(s,T)$ the function $v(\cdot,t_0)$
    has a multiple zero in $I$ and $z_I(v(\cdot,t_0))<\infty,$ then
    for any $t_1,t_2\in(s,T)$ with $t_1<t_0<t_2,$ one has
    \begin{equation}\label{zerodrop}
      z_I(v(\cdot,t_1))>z_I(v(\cdot,t_0))\ge z_I(v(\cdot,t_2)).
    \end{equation}
  \end{enumerate}
\end{lemma}
If \eqref{zerodrop} holds, we say that $z_I(v(\cdot,t))$ drops in the
interval $(t_1,t_2).$

\begin{remark}\label{convzero}{\rm It is clear that if the assumptions
    of Lemma \ref{lemzero} are satisfied and for some $t_0\in(s,T)$
    one has $z_I(v(\cdot,t_0))<\infty,$ then $z_I(v(\cdot,t))$ can
    drop at most finitely many times in $(t_0,T)$; and if it is
    constant on $(t_0,T),$ then $v(\cdot,t)$ has only simple zeros in
    $I$ for all $t\in (t_0,T).$ In particular, if $T=\infty,$ there
    exists $t_1<\infty$ such that $t\mapsto z_I(v(\cdot,t))$ is
    constant on $(t_1,\infty)$ and all zeros of $v(\codt,t)$ are
    simple.  }
\end{remark}

Using the previous remark and the implicit function theorem, we obtain
the following corollary.
\begin{cor}\label{zeroIFT}
  Assume that the assumptions of Lemma \ref{lemzero} are satisfied and
  that the function $t\mapsto z_I(v(\cdot,t))$ is constant on $(s,T).$
  If for some $(x_0,t_0)\in I\times(s,T)$ one has $v(x_0,t_0)=0,$ then
  there exists a $C^1$- function $t\mapsto\eta(t)$ defined for
  $t\in(s,T)$ such that $\eta(t_0)=x_0$ and $v(\eta(t),t)=0$ for all
  $t\in(s,T).$
\end{cor}

The following result, which is a version of Lemma \ref{lemzero} for
time-dependent intervals, is derived easily from Lemma \ref{lemzero}
(cp.  \cite[Section~2]{p-B-Q}).

\begin{lemma}
  \label{le:zerot}
  Let $v$ be a nontrivial solution of \eqref{eqlin} and
  $I(t)=(a(t),b(t))$, where $-\infty \le a(t)<b(t)\le \infty$ for
  $t\in(s,T)$.  Assume that the following conditions are satisfied:
  \begin{itemize}
  \item[ {\rm (c1)}] Either $b\equiv \infty$ or $b$ is a (finite)
    continuous function on (s,T).  In the latter case,
    $v(b(t),t)\ne 0$ for all $t\in(s,T)$.
  \item[\rm (c2)] Either $a\equiv -\infty$ or $a$ is a continuous
    function on (s,T).  In the latter case, $v(a(t),t)\ne 0$ for all
    $t\in(s,T)$.
  \end{itemize}
  Then statements (i), (ii) of Lemma \ref{lemzero} are valid with $I$,
  $a$, $b$ replaced by $I(t)$, $a(t)$, $b(t)$, respectively; and
  statement (iii) of Lemma \ref{lemzero} is valid with all occurrences
  of $z_I(v(\codt,t_j))$, $j=0,1,2$, replaced by
  $z_{I(t_j)}(v(\codt,t_j))$, $j=0,1,2$, respectively.
\end{lemma}

We will also need the following robustness lemma (see \cite[Lemma
2.6]{Du_Matano}).
\begin{lemma}\label{robustnesszero}
  Let $w_n(x,t)$ be a sequence of functions converging to $w(x,t)$ in
  $\displaystyle C^1\left( I\times(s,T)\right)$ where $I$ is an open
  interval. Assume that $w(x,t)$ solves a linear equation
  \eqref{eqlin}, $w\not\equiv0$, and $w(\cdot,t)$ has a multiple zero
  $x_0\in I$ for some $t_0\in(s,T)$.  Then there exist sequences
  $x_n\to x_0$, $t_n\to t_0$ such that for all sufficiently large $n$
  the function $w_n(\cdot,t_n)$ has a multiple zero at $x_n$.
\end{lemma}

\subsection{Phase plane of the stationary problem}\label{sec:22}
In this subsection, we examine the trajectories of the solutions of
equation \eqref{eq:steady}. The first-order system
\begin{equation}\label{eq:sys}
  u_x=v,\qquad v_x=-f(u),
\end{equation}
associated with \eqref{eq:steady} is Hamiltonian with respect to the
energy
\begin{equation}\label{energy}
  H(u,v)=\frac{v^2}{2}+F(u)
\end{equation}
(with $F$ as in \eqref{eq:3}).  Thus, each orbit of \eqref{eq:sys} is
contained in a level set of $H.$ The level sets are symmetric with
respect to the $v-$axis, and our extra hypothesis (MF) implies that
they are all bounded. Therefore, all orbits of \eqref{eq:sys} are bounded
and there are only four types of them: equilibria (all of which are on
the $u-$axis), non-stationary periodic orbits (by which we mean orbits
of nonstationary periodic solutions), homoclinic orbits, and
heteroclinic orbits. Following a common terminology, we say that a
solution $\vp$ of \eqref{eq:steady} is a \emph{ground state at level
  $\ga$} if corresponding solution $(\vp,\vp_x)$ of \eqref{eq:sys} is
homoclinic to the equilibrium $(\ga,0)$; we say that $\vp$ is \emph{a
  standing wave of \eqref{eq:1} connecting $\ga_-$ and $\ga_+$} if
$(\vp,\vp_x)$ is a heteroclinic solution of \eqref{eq:sys} with limit
equilibria $(\ga_-,0)$ and $(\ga_+,0)$.

Each non-stationary periodic orbit $\mathcal{O}$ is symmetric about
the $u-$axis and for some $p<q$ one has
\begin{align}
  \mathcal{O}\cap\{ (u,0):u\in\R\} & = \left\{(p,0),(q,0)\right\} \nonumber \\
  \mathcal{O}\cap \left\{(u,v):v>0\right\} & = \left\{\left( u,\sqrt{2(F(p)-F(u))}\right):u\in(p,q)\right\}. \label{periodicorbits}
\end{align}

Let
\begin{align}
  \mathcal{E} & := \{ (a,0):f(a)=0\}\  \textrm{ (the set of all equilibria of \eqref{eq:sys})}, \nonumber \\
  \mathcal{P}_0 & :=\{(a,b)\in\R^2: (a,b)\textrm{ lies on a non-stationary periodic orbit of \eqref{eq:sys}}\}, \nonumber \\
  \mathcal{P} & := \mathcal{P}_0\cup\mathcal{E}\  \textrm{ (the set of
                all periodic orbits of \eqref{eq:sys}, including the equilibria)}. \nonumber
\end{align}
The next lemma gives a description of the phase plane portrait of
\eqref{eq:sys} with all the non-stationary periodic orbits removed.  The
following observations will be useful in its proof and at other places
below. Let $(p,0)$ be an equilibrium of \eqref{eq:sys}. Then $f(p)=0$
and, by (ND), $f'(p)\ne 0$. Elementary considerations using the
Hamiltonian $H$ show that if $f'(p)>0$, then $(p,0)$ is not contained
in the closure of any homoclinic or heteroclinic orbit of
\eqref{eq:sys}. On the other hand, if $f'(p)<0$, then (MF) implies that
$(p,0)$ is contained in the closure of a homoclinic or heteroclinic
orbit contained in the halfplane $\{(u,v):u>p\}$ as well as of another
one contained in the halfplane $\{(u,v):u<p\}$.

\begin{lemma}\label{MatPolLemma}
  The following two statements are valid.
  \begin{enumerate}
  \item[(i)] Let $\Sigma$ be a connected component of
    $\R^2\setminus\mathcal{P}_0.$ Then $\Sigma$ is a compact set
    contained in a level set of the Hamiltonian $H$ and one has
    \begin{equation*}
      \Sigma = \left\{(u,v)\in\R^2:u\in J,\ v=\pm\sqrt{2(c-F(u))}\right\}
    \end{equation*}
    where $c$ is the value of $H$ on $\Sigma$ and $J=[p,q]$ for some
    $p,q\in\R$ with $p\leq q.$ Moreover, if $(u,0)\in\Sigma$ and
    $p<u<q,$ then $(u,0)$ is an equilibrium.  If $p<q,$ the points
    $(p,0)$ and $(q,0)$ lie on homoclinic
    orbits. 
    If $p=q,$ then $\Sigma=\{(p,0)\},$ and $p$ is an unstable
    equilibrium of \eqref{eq:ODE}.
  \item[(ii)] Each connected component of the set
    $\R^2\setminus\mathcal{P}$ consists of a single orbit of
    \emph{\eqref{eq:sys}}, either a homoclinic orbit or a heteroclinic
    orbit.
  \end{enumerate}
\end{lemma}
\begin{proof}
  These results, except for the last two statements in (i) are proved
  in \cite[Lemma 3.1]{Matano_Polacik_CPDE16} (and they are valid
  without the nondegeneracy condition (ND)). It is also proved there
  that the point $(p,0)$ is an equilibrium or it lies on a homoclinic
  orbit. We show that $(p,0)$ is not an equilibrium if $p<q$.  Indeed,
  assume it is. Then, in view of (ND) and the relation $p<q$, there is
  a homoclinic or heteroclinic orbit of \eqref{eq:sys} (contained in
  $\Sigma$) having $(p,0)$ in its closure. Hence, necessarily,
  $f'(p)<0$, and then it follows that $(p,0)$ is in the closure of
  another homoclinic or heteroclinic orbit contained in
  $\{(u,v):u<p\}$ (see the remarks preceding the lemma).  This
  contradicts the fact $\Sigma$ is a connected component of
  $\R^2\setminus\mathcal{P}_0$.  Analogous arguments show that $(q,0)$
  lies on a homoclinic orbit. For similar reasons, if
  $\Sigma=\{(p,0)\}$, so $(p,0)$ is clearly an equilibrium, the
  relation $f'(p)<0$ would imply that $\Sigma$ is not a connected
  component of $\R^2\setminus\mathcal{P}_0$. Thus $f'(p)>0$.
\end{proof}
 
The above Lemma motivates the following definitions.
  A \textit{chain} is any connected component $\Sigma$ of
    $\R^2\setminus\mathcal{P}_0.$ We say that a chain is \emph{trivial} if
    it consists of a single point.

  If $\mathcal{H}$ is a connected component of
    $\R^2\setminus\mathcal{P}$, let $\Lambda(\mathcal{H})$
    the set consisting of the closure of $\mathcal H$ and the
    reflection of $\mathcal H$ with 
    respect to the $u-$axis. So $\Lambda(\mathcal{H})$ is either the
    union of a homoclinic orbit and its limit equilibrium, or the
    union of two heteroclinic orbits and their common limit
    equilibria. We refer to $\Lambda(\mathcal{H})$ as the
    \textit{loop} associated with $\mathcal{H}.$

Hypotheses (ND) and (MF) imply that $f$ has only finitely many
zeros. Since any chain or loop contains an equilibrium, there is only
a finite number of chains and loops.  In particular, any chain is the
union of finitely many loops. Also, any chain is a compact subset of
$\R^2$ and so their (finite) union, that is, the set
$\R^2\setminus\mathcal{P}_0$, is compact. This  implies that
$\mathcal{P}_0$ admits a  unique unbounded connected component
and all connected components of
  $\mathcal{P}_0$ (as well $\mathcal{P}_0$ itself) are open sets.

If $\Sigma$ is a chain, we denote by $\mathcal{I}(\Sigma)$ the union
of all bounded connected components of $\R^2\setminus\Sigma.$ Thus,
$\mathcal{I}(\Sigma)$ is the union of the interiors of the loops,
viewed as Jordan curves, contained in $\Sigma$; if $\Sigma$ consists
of a single equilibrium (necessarily a center for \eqref{eq:sys}), 
$\mathcal{I}(\Sigma)=\emptyset$. 
Since $\Sigma$ is
clearly compact in $\R^2,$ the set $\mathcal{I}(\Sigma)$ is open.  We
also define
${\overline{\mathcal{I}}(\Sigma)}=\mathcal{I}(\Sigma)\cup\Sigma.$ The
set ${\overline{\mathcal{I}}(\Sigma)}$ is closed and equal to the
closure of $\mathcal{I}(\Sigma)$, except when $\Sigma$ consists of
a single point, in which case  $\mathcal{I}(\Sigma)=\Sigma$.
In a similar way we define the sets
$\mathcal{I}(\La)$, ${\overline{\mathcal{I}}(\La)}$,
$\mathcal{I}(\cO)$, ${\overline{\mathcal{I}}(\cO)}$, when $\Lambda$ is
a loop and $\mathcal{O}$ is a non-stationary periodic orbit.
  
The following lemma introduces two key concepts: the inner chain and
the outer loop associated with a connected component of
$\mathcal{P}_0$ (see also Figure \ref{inandout}).
 
 \begin{lemma}\label{le:p0}
   Let $\Pi$ be any connected component of $\mathcal{P}_0.$ The
   following statements hold true.
   \begin{enumerate}
   \item[(i)] The set $\Pi$ is open.
   \item[(ii)] There exists a unique chain $\Sigma_{in}$ such that for
     all periodic orbits $\mathcal{O}\subset\Pi$ one has
     \begin{equation*}
       {\overline{\mathcal{I}}\left(\Sigma_{in}\right)}\subset\mathcal{I}(\mathcal{O})\textrm{ and }\mathcal{I}(\mathcal{O})\setminus {\overline{\mathcal{I}}(\Sigma_{in})}\subset\Pi.
     \end{equation*}
   \item[(iii)] If $\Pi$ is bounded, there exists a unique loop
     $\Lambda_{out}$ such that for all periodic orbits
     $\mathcal{O}\subset\Pi$ one has
     \begin{equation*}
       \ol{\mathcal{I}}(\mathcal{O})\subset
       \mathcal{I}(\Lambda_{out}),\textrm{ and }
       \mathcal{I}(\Lambda_{out})\setminus
       {\overline{\mathcal{I}}(\mathcal{O})}\subset\Pi.    
     \end{equation*}
   \item[(iv)] There is a zero $\be$ of $f$ such that $f'(\be)>0$ and
     $(\beta,0)\in\mathcal{I}(\mathcal{O}),$ for all periodic orbits
     $\mathcal{O}\subset\Pi.$
   \item[(v)] If $\mathcal{O}_1,\mathcal{O}_2$ are two distinct
     periodic orbits contained in $\Pi,$ then either
     $\mathcal{O}_1\subset\mathcal{I}\left(\mathcal{O}_2\right)$ or
     $\mathcal{O}_2\subset\mathcal{I}\left(\mathcal{O}_1\right)$
     (thus, $\Pi$ is totally ordered by this relation).
   \end{enumerate}
 \end{lemma}
 We refer to $\Sigma_{in}$ and $\La_{out}$ as the \emph{inner chain
   and outer loop} associated with $\Pi$; we denote them by
 $\Sigma_{in}(\Pi)$ and $\La_{out}(\Pi)$ if the correspondence to
 $\Pi$ is to be explicitly indicated.

\begin{figure}[h]
  \vspace{-3.5cm}
 
  \addtolength{\belowcaptionskip}{10pt}
  \addtolength{\abovecaptionskip}{-3cm}
  \hspace{.1in} \includegraphics[scale=.6]{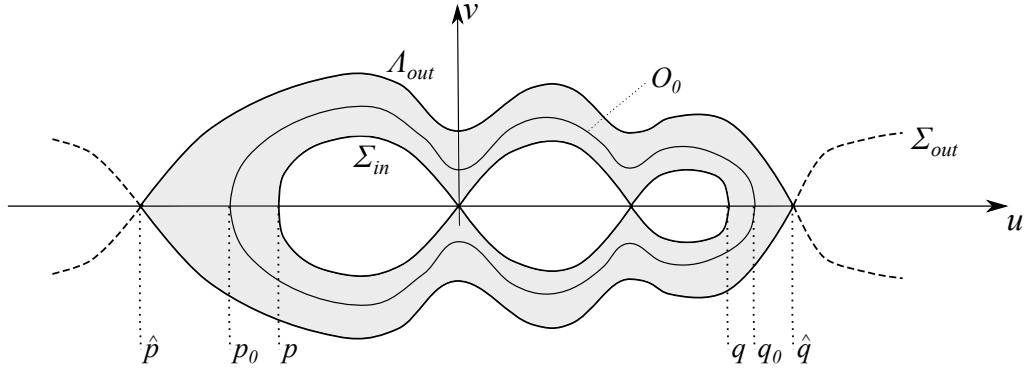}
 
  \caption[Inner chain and outer loop]{The inner chain and outer loop
    associated with a connected component $\Pi$ of $\cP_0$: $\Pi$ is
    indicated by the shaded region, $\La_{out}$ and $\Sigma_{in}$ form
    the boundary of $\Pi$.  The outer loop can be a heteroclinic loop
    (as in this figure) or a homoclinic loop. \label{inandout}}
\end{figure}

\begin{proof}[Proof of Lemma \ref{le:p0}] The openness of $\Pi$
  follows from the compactness of $\R^2\setminus\mathcal{P}_0$, as
  already mentioned above. This takes care of statement (i).

  In the rest of the proof, we assume for definiteness
  that  $\Pi$ is bounded and prove statements (ii)-(v).
  The proof of statements (ii), (iv),
  (v) in the case that $\Pi$ is the unique
  unbounded connected component of $\mathcal{P}_0$
  is similar and is omitted.
 
  Fix any periodic orbit $\mathcal{O}_0\subset\Pi.$ By
  \eqref{periodicorbits}, there are $p_0<q_0$ such that
  $\mathcal{O}_0\cap\{(u,0):u\in\R\}=\{(p_0,0),(q_0,0)\},$ with
  $f(p_0)=F'(p_0)<0$ and $f(q_0)=F'(q_0)>0.$ Define
  \begin{equation*}
    q:=\sup\{q<q_0:(q,0)\not\in\Pi\},\ \textrm{ and }\
    \hat q:=\inf\{q>q_0:(q,0)\not\in\Pi\}. 
  \end{equation*}
  In other words, $(q,\hat q)$ is the maximal open interval containing
  $q_0$ such that $(q,\hat q)\times \{0\}\subset \Pi$.  The existence
  of such an interval is guaranteed by the openness of $\Pi$. Note
  also that none of the points $(q,0)$, $(\hat q,0)$ is contained in
  $ \cP_0$.  Indeed, if, say, $(q,0)\in \cP_0$, then a neighborhood of
  $(q,0)$ is contained in $\cP_0$. By the definition of $q$, this whole
  neighborhood would necessarily be contained in the connected
  component $\Pi$, from which we immediately get a contraction to the
  definition of $q$. So, indeed, $(q,0), (\hat q,0)\not\in\cP_0$, in
  particular they are not equilibria of \eqref{eq:sys}.
  Since there is no element of $\mathcal{E}$ in
  $(q,\hat q)\times\{0\}\subset\Pi$, we have $F'=f\ne 0$ on $(q,\hat q)$ and
  $F'(q_0)>0$ implies that $F'> 0$ on $(q,\hat q)$.  In an analogous
  way, one finds a maximal interval $(\hat p,p)$ containing $p_0$ such
  that $(\hat p,p)\times \{0\}\subset \Pi$, and proves that
  $(p,0), (\hp,0)\not \in \cP_0$ and $F'< 0$ on $(\hat p,p)$.  A
  continuity argument shows that the union of all (periodic) orbits of
  \eqref{eq:sys} intersecting the segment $(q,\hat q)\times \{0\}$ is
  equal to the union of all orbits of \eqref{eq:sys} intersecting
  $(\hat p,p)\times \{0\}$. We denote this union by $\tilde \Pi$.  As
  distinct orbits of \eqref{eq:sys} do not intersect, it is clear that
  the periodic orbits contained in $\tilde \Pi$ are nested in the
  sense that (v) holds with $\Pi$ replaced by $\tilde \Pi$. (We will
  prove below that in fact $\Pi=\tilde \Pi$, thereby proving statement
  (v).)  Observe also that the points $(p,0)$, $(q,0)$ can be
  approximated arbitrarily closely by one orbit contained in
  $\tilde\Pi$. This implies that they are in the same level set of the
  Hamiltonian, that is, $F(p)=F(q)$, and also that $F\le F(q)$ in
  $(p,q)$. One easily proves from this that the points $(p,0)$,
  $(q,0)$ lie on the same chain which we denote by $\Sigma_{in}$.
  Using Lemma \ref{MatPolLemma} (and the fact that $(p,0)$,
  $(q,0)$ are not equilibria), we can write:
  \begin{equation}
    \label{eq:8}
    \Sigma_{in} = \left\{(u,v)\in\R^2:u\in [p,q],\
      v=\pm\sqrt{2(F(q)-F(u))}\right\}. 
  \end{equation}
  
  Similarly one shows that $F(\hat p)=F(\hat q)$, $F\le F(\hat q)$ in
  $(\hat p,\hat q)$, and the points $(\hat p,0)$, $(\hat q,0)$ lie on
  the same chain, which we denote by $\Sigma_{out}$.  By Lemma
  \ref{MatPolLemma},
  \begin{equation*}
    \Sigma_{out} = \left\{(u,v)\in\R^2:u\in [\bar p,\bar q],\ v=\pm\sqrt{2(F(\hat q)-F(u))}\right\}
  \end{equation*}
  for some $\bar p\le \hat p$, $\bar q\ge\hat q.$ The inequality for
  $F$ actually holds in the strict sense: $F< F(\hat q)$ in
  $(\hat p,\hat q)$, due to the previously established relation
  $F\le F(q)$ in $(p,q)$ and the strict monotonicity properties of $F$
  in the intervals $(\hat p,p)$, $(q,\hat q)$. It follows that the
  set
  \begin{equation}
    \label{eq:11}
    \Lambda_{out}:=\left\{(u,v)\in\R^2:u\in [\hat p,\hat q],\
      v=\pm\sqrt{2(F(\hat q)-F(u))}\right\} 
  \end{equation}
  is a loop contained in $\Sigma_{out}.$ Clearly, $\Sigma_{in}$,
  $\Sigma_{out}$ are distinct (hence disjoint) chains; in fact, they
  lie on two different level sets of the Hamiltonian $H$.

  It is obvious from the above constructions that for any periodic
  orbit $\cO\subset \tilde\Pi$ we have
  \begin{equation}
    \label{eq:10}
    {\overline{\mathcal{I}}\left(\Sigma_{in}\right)}\subset\mathcal{I}(\mathcal{O})\subset\overline{\mathcal{I}}(\mathcal{O})\subset  \mathcal{I}(\Lambda_{out}).
  \end{equation}
  We next claim that
  \begin{equation}
    \label{eq:9}
    \mathcal{I}(\Lambda_{out})\setminus
    \overline{\mathcal{I}}\left(\Sigma_{in}\right)=\tilde \Pi.
  \end{equation}
  That $\tilde \Pi$ is included in the set on the left has already
  been proved (cp.  \eqref{eq:10}); we prove the opposite
  inclusion. Take any
  $(\xi,\eta)\in \mathcal{I}(\Lambda_{out})\setminus
  \overline{\mathcal{I}}\left(\Sigma_{in}\right)$. If $(\xi,\eta)$
  lies on a periodic orbit, then that orbit intersects the $u$-axis in
  the set $((q,\hat q)\cup (\hat p,p))\times \{0\}$ (otherwise, in
  view of \eqref{eq:8}, \eqref{eq:11} it would have to intersect one
  of the chain $\Sigma_{in}$, $\Sigma_{out}$, which is impossible),
  and hence $(\xi,\eta)\in \tilde \Pi$. If $(\xi,\eta)$ does not lie
  on a periodic orbit, then it is contained in a chain disjoint from
  $\Sigma_{in}\cup\Sigma_{out}$, and such a chain would also have to
  intersect the set $((q,\hat q)\cup (\hat p,p))\times \{0\}$. This is
  impossible, as this set is included in $\tilde \Pi\subset
  \cP_0$. Thus \eqref{eq:9} is true.

  From \eqref{eq:9} it follows that $\tilde \Pi$ is a connected
  component of $\cP_0$, hence $\tilde \Pi=\Pi$. As already noted above,
  this proves statement (v). Statements (iii) and (iv) follow from
  \eqref{eq:10}, \eqref{eq:9}. To prove statement (iv), take the
  minimum point $\be$ of $F$ in $[p,q]$.  Recalling that $F'(p)<0$,
  $F'(q)>0$, we see that $\be\in (p,q)$ and it is a local minimum point
  of $F$, hence $f(\be)=0$ and $f'(\be)>0$, due to (ND).  Statement
  (v) clearly holds for this $\be$.
\end{proof}

The following lemma shows a relation between any two distinct chains.
\begin{lemma}
  \label{le:insert}
  \begin{itemize}
  \item[(i)] If $\Sigma$ is any chain, then there is a connected
    component $\Pi$ of $\cP_0$ such that $\Sigma$ is the inner chain
    associated with $\Pi${\rm :}
    $\Sigma=\Sigma_{in}(\Pi)$.
  \item[(ii)] If $\Sigma_1$,
    $\Sigma_2$ are any  two  distinct
chains, then  either
$\Sigma_1\subset \mathcal{I}(\Sigma_2)$, or
$\Sigma_2\subset \mathcal{I}( \Sigma_1)$, or else
there are periodic orbits $\cO_1$, $\cO_2$ such that 
$\ol{\cI}(\cO_1)\cap\ol{\cI}( \cO_2)=\emptyset$ and 
\begin{equation}
  \label{eq:53}
  \Sigma_1\subset \mathcal{I}(\cO_1),\qquad \Sigma_2\subset \mathcal{I}(\cO_2). 
\end{equation}
  \end{itemize}
\end{lemma}
\begin{proof}
  Since there are only finitely many chains, for any given chain
  $\Sigma$ there is a connected
  component $\Pi$ of $\cP_0$ such that
  $\Pi\subset \R^2\setminus\ol\cI(\Sigma)$ and the boundary of $\Pi$ contains points
  of $\Sigma$. It then follows from
  Lemma \ref{le:p0} that $\Sigma=\Sigma_{in}(\Pi)$. This proves
  statement (i).

  Let now $\Sigma_1$, $\Sigma_2$ be any  two  distinct
chains, and let $\Pi_1$, $\Pi_2$ be the connected
    components  of $\cP_0$ such that
    $\Sigma_j=\Sigma_{in}(\Pi_i)$, $j=1,2$. Pick periodic orbits
    $\cO_1\subset \Pi_1$, $\cO_2\subset \Pi_2$. By  Lemma \ref{le:p0},
    inclusions \eqref{eq:53} hold. Clearly, exactly one of the
    following possibilities occurs:
    \begin{equation*}
      {\rm (a)} \quad \cO_1\subset \mathcal{I}(\cO_2),\qquad
      {\rm (b)} \quad \cO_2\subset \mathcal{I}(\cO_1)\qquad
      {\rm (c)} \quad \ol{\cI}(\cO_1)\cap\ol{\cI}( \cO_2)=\emptyset.
    \end{equation*}
    For the proof of statement (ii), it is now sufficient to prove
    that (a) implies that $\Sigma_1\subset \mathcal{I}(\Sigma_2)$, and
    (b) implies $\Sigma_2\subset \mathcal{I}( \Sigma_1)$. These being
    symmetrical cases, we only prove the former. Trivially,
    $\Sigma_1\cap(\Pi_2\cup\Sigma_2)=\emptyset$; and  Lemma \ref{le:p0}(ii)
    gives $\mathcal{I}(\mathcal{O}_2)\setminus
    {\overline{\mathcal{I}}(\Sigma_{2})}\subset\Pi_2$.
    Thus, if (a) holds, which entails $\Sigma_1\subset
    \mathcal{I}(\cO_2)$, then necessarily $\Sigma_1\subset
    \mathcal{I}(\Sigma_2)$. 
\end{proof}


\subsection{Limit sets and entire solutions}\label{sec:inv}
Recall that the $\omega-$limit set of a bounded solution $u$ of
\eqref{eq:1}, denoted by $\omega(u)$, or $\omega(u_0)$ if the initial
value of $u$ is given, is defined as in \eqref{defomega}, with the
convergence in $L^\infty_{loc}(\R)$.  By standard parabolic estimates
the trajectory $\{ u(\cdot,t),\ t\geq1\}$ of $u$ is relatively compact
in $L^\infty_{loc}(\R).$ This implies that $\omega(u)$ is nonempty,
compact, and connected in (the metric space) $L^\infty_{loc}(\R)$ and
it attracts the solution in the following sense:
\begin{equation}
  \label{eq:2.6}
  \textrm{dist}_{L^\infty_{loc}(\R)}\left( u(\cdot,t),\omega(u)\right)\underset{t\to\infty}{\longrightarrow}0.
\end{equation}
It is also a standard observation that if $\vp\in\omega(u),$ there
exists an \emph{entire solution} $U(x,t)$ of \eqref{eq:1}, that is, a
solution defined for all $t\in\R$, such that
\begin{equation}\label{entiresol}
  U(\cdot,0)=\vp,\qquad U(\cdot,t)\in\omega(u)\quad (t\in\R).
\end{equation}
We recall briefly how such an entire solution $U$ is found.  By
parabolic regularity estimates, $u_t,u_x,u_{xx}$ are bounded on
$\R\times[1,\infty)$ and are globally $\alpha-$H\"older for any
$\alpha\in(0,1).$ If
$u(\cdot,t_n)\underset{n\to\infty}{\longrightarrow}\vp$ in
$L^\infty_{loc}(\R)$ for some $t_n\to\infty,$ we consider the sequence
$u_n(x,t):=u(x,t+t_n)$, $n=1,2\dots$.  Passing to a subsequence if
necessary, we have $u_n\to U$ in $C^1_{loc}(\R^2)$ for some function
$U;$ this function $U$ is then easily shown to be an entire solution
of \eqref{eq:1}.  By definition, $U$ satisfies \eqref{entiresol}. Note
that the entire solution $U$ is determined uniquely by $\varphi$; this
follows from the uniqueness and backward uniqueness for the Cauchy
problem \eqref{eq:1}, \eqref{ic1}.

Using similar compactness arguments, one shows easily that $\omega(u)$
is connected in $C_{loc}^1(\R).$ Hence, the set
\[
\tau\left(\omega(u)\right)=\left\{
  (\vp(x),\vp_x(x)):\vp\in\omega(u),x\in\R\right\} =
\underset{\vp\in\omega(u)}{\textstyle \bigcup}\tau(\vp)
\]
is connected in $\R^2$. (Here, $\tau(\varphi)$ is as in
\eqref{eq:straj}.)  Also, obviously, $\tau(\vp)$ is connected in
$\R^2$ for all $\vp\in\omega(u).$

If $U$ is a bounded entire solution of \eqref{eq:1}, we define its
$\alpha-$limit set by
\begin{equation}\label{defalpha}
  \alpha(U):=\left\{ \vp\in C_b(\R):U(\cdot,t_n)\to\vp\textrm{ for some sequence }t_n\to-\infty\right\}.
\end{equation}
Here, again, the convergence is in $L_{loc}^\infty(\R).$ The
$\alpha$-limit set has similar properties as the $\omega-$limit set:
it is nonempty, compact and connected in $L^\infty_{loc}(\R)$ as well
as in $C^1_{loc}(\R)$, and for any $\varphi\in \al(U)$ there is an
entire solution $\tilde U$ such that $\tilde U(\cdot,0)=\vp$ and
$\tilde  U(\cdot,t)\in\al(U)$ for all $t\in\R$.  The connectivity property of
$\al(U)$ implies that the set
$$
\tau\left(\alpha(U)\right)=\left\{
  (\vp(x),\vp_x(x)):\vp\in\alpha(U),x\in\R\right\} =
\underset{\vp\in\alpha(U)}{\textstyle\bigcup}\tau(\vp)
$$
is connected in $\R^2.$

We will also employ a generalized notion of $\al$ and $\om$-limit
sets. Namely, if $U$ is a bounded entire solution, we define
\begin{align}
  \Omega(U) & := \left\{ \vp\in C_b(\R):U(\cdot+x_n,t_n)\to\vp\textrm{
              for some sequences }x_n\in\R, \ t_n\to\infty\right\} \label{defOmega}, \\
  A(U) & := \left\{ \vp\in C_b(\R):U(\cdot+x_n,t_n)\to\vp\textrm{ for some sequences }x_n\in\R, \ t_n\to-\infty\right\}. \label{defAlpha}
\end{align}
The convergence is in $L_{loc}^\infty(\R)$, but again one can take the
convergence in $C^1_{loc}(\R)$ without altering the sets $\Om(U)$,
$A(U)$.  These sets are nonempty, compact and connected in
$C^1_{loc}(\R)$, and they have a similar invariance property as
$\om(u)$ (cp. \eqref{entiresol}).  Also, by their definitions, the
sets $\Om(U)$, $A(U)$ are translation invariant as well. Further, the
definitions and parabolic regularity imply that the sets
\begin{equation*}
  \tau\left( A(U)\right)= \underset{\vp\in A(U)}{\textstyle\bigcup}\tau(\vp),
  \quad\tau\left(\Omega(U)\right) =\underset{\vp\in\Omega(U)}{\textstyle\bigcup}\tau(\vp)
\end{equation*}
are connected and compact in $\R^2.$ We remark that the sets
$\tau(\om(u))$, $\tau(\al(u))$ are both connected (as noted above),
but they are not necessarily compact in $\R^2$. 

\subsection{Some  results from earlier papers}\label{sec:24}
Several earlier results are used repeatedly in the forthcoming
sections.  We state them here for reference.

Throughout this subsection, we assume that $u_0\in C_b(\R)$ (not
necessarily in $\cV$), $u$ is the solution of \eqref{eq:1}, \eqref{ic1}
and it is bounded.

In view of the invariance property of $\om(u)$ (see
\eqref{entiresol}), the following lemma gives a criterion for an
element $\varphi\in \om(u)$ to be a steady state. This
unique-continuation
type result is proved in a more general form in
\cite[Lemma 6.10]{Polacik_terrasse}.
\begin{lemma}\label{le:2.7}
  Let $\varphi:=U(\codt,0)$, where $U$ is a solution of \eqref{eq:1}
  defined on a time interval $(-\de,\de)$ with $\de>0$ (this holds in
  particular if $\varphi\in \om(u)$).  If $\tau(\varphi)\subset\Sigma$
  for some chain $\Sigma,$ then $\varphi$ is a steady state of
  \eqref{eq:1}.
\end{lemma}

As already noted above, it is proved in \cite{Gallay-S} (see also
\cite{Gallay-S2}) that the $\om$-limit set of any bounded solution of
\eqref{eq:1} contains a steady state. For bounded entire solutions $U$, the
same is true for the $\al$-limit set due to its compactness and
invariance properties (just apply the previous result to any entire
solution $\tilde U$ with $\tilde U(\cdot,t)\in\al(U)$).
We state this in the following theorem.

\begin{theorem}
  \label{thm:GS} If $U$ is a bounded entire solution of \eqref{eq:1},
  then each of the sets $\om(U)$ and $\al(U)$ contains a steady state
  of \eqref{eq:1}.
\end{theorem}

In the next two results, we make use of the invariance of equation
\eqref{eq:1} under spatial reflections.  For any $\lambda\in\R$
consider the function $V_\la u$ defined by
\begin{equation}\label{reflexion}
  V_\la u(x,t)=u(2\la-x,t)-u(x,t),\quad x\in\R,\,t\ge 0.
\end{equation}
Being the difference of two solutions of \eqref{eq:1}, $V_\la u$ is a
solution of the linear equation \eqref{eqlin} for some bounded
function $c.$

The following lemma is an adaptation of an argument from \cite[Proof
of Proposition 2.1]{p-B-Q}.

\begin{lemma}\label{lemmaBPQ}
  Let $U$ be a solution of \eqref{eq:1} on $\R\times J$, where
  $J\subset \R$ is an open time interval, and let 
  $\theta\in \R$. Assume that for each $t\in J$
  the function $U(\cdot,t)-\theta$ has at least
  one zero and
 $$
 \xi(t):=\sup\{ x:U(x,t)=\theta\}
 $$
 is finite and depends continuously on $t\in J$.  Then, for any
 $t_0, t_1\in J$ satisfying the relations $t_1>t_0$ and
 $\xi(t_1)<\xi(t_0)$, the function $U_x(\cdot,t_1)$ is of constant
 sign on the interval $(\xi(t_1),\xi(t_0)]$.  If $J=(-\infty,b)$ for
 some $-\infty<b\le \infty$ and $\limsup_{t\to-\infty}\xi(t)=\infty$,
 then $U_x$ is of constant sign on $(\xi(t),\infty),$ for all
 $t\in J.$

 Analogous statements hold for $\xi(t)=\inf\{x:U(x,t)=\theta\}.$
\end{lemma}
\begin{proof}
  Pick any $\la\in (\xi(t_1),\xi(t_0)]$ and set
  $\displaystyle \bar t:=\max\left\{ t\in
    [t_0,t_1):\xi(t)=\la\right\}$.  Consider the function $V_\la U$ on
  the domain
$$
Q_\la:=\left\{ (x,t):\ x\in(\xi(t),\la),\ t\in(\bar t,t_1)\right\}.
$$
Clearly, $V_\la U(\la,t)=0$ for all $t$ and, as $\xi(t)$ is the last
zero of $U(\codt,t)$, $V_\la U(\xi(t),t)$ is of constant sign on
$(t_0,t_1)$. Since $V_\la U$ solves a linear parabolic equation
\eqref{eqlin}, the maximum principle implies that $V_\la U$ is of
constant sign on the whole domain $Q_\la$, and the Hopf lemma yields
$-2\partial_xU(\la,t_1)=\partial_xV_\la u(\la,t_1)\neq0$.  Since
$\la\in (\xi(t_1),\xi(t_0)]$ was arbitrary, $U_x(\cdot,t_1)$ is of
constant sign on $(\xi(t_1),\xi(t_0)]$.

To prove the second statement, fix any $t'\in J$ and let
$\la>\xi(t')$.  By the unboundedness assumption on $\xi(t)$,
$t_0:=\sup\{t<t':\xi(t)=\la\}$ is a number in $(-\infty,t')$. Applying
the result just proved, we obtain $U_x(\la,t')\ne0$. Since
$\la>\xi(t')$ was arbitrary, we obtain the desired conclusion.
\end{proof}

We next state a quasiconvergence result from our previous paper
\cite{Pauthier_Polacik1}.
\begin{theorem}\label{thmPP1}
  Assume that $u_0\in\cV$ and one of the following conditions holds:
  \begin{itemize}
  \item[(i)] $u_0(-\infty)\ne u_0(\infty)$,
  \item[(ii)] there is $t>0$ such that for all $\la\in\R,$ one has
    $z(V_\la u(\cdot,t))<\infty.$
  \end{itemize}
  Then, $u$ is quasiconvergent.
\end{theorem}
If condition (i) is assumed, this is the content of the main theorem
in \cite{Pauthier_Polacik1}. In the proof of the theorem, we first
proved that condition (i) and Lemma \ref{lemzero} imply that condition
(ii) holds (this is actually the only place where condition (i) is
used in the proof).  As noted in \cite[Remark 3.3]{Pauthier_Polacik1},
the quasiconvergence result holds if condition (i) is replaced by (ii)
from the start.

The following result concerning various invariant sets for \eqref{eq:1}
is a variant of the squeezing lemma from
\cite{P:entire}. This is an indispensable tool in our proofs.
\begin{lemma}\label{squeezinglemma}
  Let $U$ be a bounded entire solution of \eqref{eq:1} such that if
  $\beta\in f^{-1}\{0\}$ is an unstable equilibrium of
  \eqref{eq:ODE}, then
  \begin{equation}\label{eq:2.15}
    z\left( U(\cdot,t)-\beta\right)\leq N\quad (t\in\R)
  \end{equation}
  for some $N<\infty.$ Let $K$ be any one of the following subsets of
  $\R^2:$
 $$
 \underset{t\in\R}{{\textstyle \bigcup}}\tau\left( U(\cdot,t)\right),\quad
 \tau\left(\omega(U)\right), \quad \tau\left(\Omega(U)\right),\quad
 \tau\left(\alpha(U)\right), \quad \tau\left( A(U)\right).
 $$
 Assume that $\mathcal{O}$ is a  non-stationary periodic orbit of
 \eqref{eq:sys} such that one of the following inclusions holds:
 \begin{equation*}
   \text{(i) \quad $K\subset   \mathcal{I}(\mathcal{O})$,\qquad\qquad
   (ii)\quad $K\subset\R^2\setminus{\overline{\mathcal{I}}(\mathcal{O})}$.}
 \end{equation*}
 Let
 $\Pi$ be the connected component of $\cP_0$ containing $\cO$.
 If (i) holds, then 
 $K\subset{\overline{\mathcal{I}}\left(\Sigma_{in}(\Pi)\right)}$;
 and if (ii) holds, then
 $K\subset\R^2\setminus\mathcal{I}\left(\Lambda_{out}(\Pi)\right)$
 (in particular, $\Pi$ is necessarily bounded in this case).
\end{lemma}

\begin{proof}
  We prove the result in the case (i) only, the proof in the case (ii)
  is analogous.  To simplify the notation, let
  $\Sigma_{in}:=\Sigma_{in}(\Pi)$.

  Take first
  $\displaystyle K={\textstyle\bigcup}_{t\in\R}\tau\left(
    U(\cdot,t)\right)$.  We go by contradiction: assume that
  $K\not\subset{\overline{\mathcal{I}}\left(\Sigma_{in}\right)}.$ Then
  there exists a periodic orbit $\mathcal{O}_1\subset \Pi$ such that
  $\displaystyle \tau\left( U(\cdot,t_1)\right)\cap
  \mathcal{O}_1\neq\emptyset,$ for some $t_1\in\R.$ By the hypotheses,
  $\overline{K}\subset \overline{\mathcal{I}}(\mathcal{O})$ and
  $\overline{K}$ is a compact set.  Using the compactness and the
  ordering of periodic orbits contained in $\Pi,$ as given in Lemma
  \ref{le:p0}(v), we find the minimal periodic orbit
  $\mathcal{O}_{min}\subset\mathcal{P}_0$ with
  $\overline{K}\subset{\overline{\mathcal{I}}\left(\mathcal{O}_{min}\right)}$. Clearly,
  \begin{equation}\label{eq:2.16}
    \overline{K}\subset{\overline{\mathcal{I}}\left(\mathcal{O}_{min}\right)}, \qquad
    \overline{K}\cap\mathcal{O}_{min}\neq\emptyset.
  \end{equation}
  Hence, there exist sequences $x_n,t_n$ such that
$$
\left(
  U(x_n,t_n),U_x(x_n,t_n)\right)\underset{n\to\infty}{\longrightarrow}(a,b)\in\mathcal{O}_{min}.
$$
Let $\psi_{min}$ be a periodic solution of \eqref{eq:steady} with
$\psi_{min}(0)=a,$ $\psi_{min}'(0)=b,$ so that
$\tau(\psi_{min})=\mathcal{O}_{min}.$ Consider the sequence of
functions $U_n:=U(\cdot+x_n,\cdot+t_n).$ By parabolic estimates, upon
extracting a subsequence, $U_n$ converges in $C^1_{loc}(\R^2)$ to an
entire solution $U_\infty$ of \eqref{eq:1}. Obviously,
$U_\infty(\cdot,0)-\psi_{min}$ has a multiple zero at $x=0.$
 
We claim that $U_\infty\not\equiv\psi_{min}.$ Indeed, by Lemma
\ref{le:p0}(iv), there exists an unstable equilibrium $\be$ of
\eqref{eq:ODE} such that $\displaystyle z(\psi_{min}-\beta)=\infty.$
Hence, there exists $M>0$ such that
$\displaystyle z_{(-M,M)}(\psi_{min}-\beta)>N+1,$ where $N$ is as in
\eqref{eq:2.15}. Obviously, all zeros of $\psi_{min}-\beta$ are
simple. 
Considering that $U(\cdot+x_n,t_n)$ converges
uniformly on $(-M,M)$ to $U_\infty(\cdot,0)$ and
$\displaystyle z_{(-M,M)}\left( U(\cdot+x_n,t_n)-\beta\right)\leq N,$
we see that $U_\infty$ cannot be identical to $\psi_{min}.$
 
Now, using Lemma \ref{le:p0}(v), we find a sequence $\mathcal{O}_n$ of
periodic orbits such that
$\displaystyle\mathcal{O}_{n+1}\subset\mathcal{I}\left(\mathcal{O}_n\right)$,
$\displaystyle\mathcal{O}\subset\mathcal{I}\left(\mathcal{O}_n\right)$,
for $n=1,2,\dots$, 
and
$\displaystyle \dist\left(\mathcal{O}_n,\mathcal{O}_{min}\right)\to0.$\footnote{
Here and below, for $A,B\subset\R^2$,
$\dist(A,B)=\inf_{a\in A,b\in B}|a-b|$.}  There is a sequence $\psi_n$
of periodic solutions of \eqref{eq:steady} such that
$\tau(\psi_n)=\cO_n$ and $\displaystyle \psi_n\to\psi_{min}$ in
$C^1_{loc}(\R).$ Then, the sequence of functions $w_n:=U_n-\psi_n$
converges in $C^1_{loc}(\R^2)$ to
$w(x,t):=U_\infty(x,t)-\psi_{min}(x)$, which is an entire solution of a linear
parabolic equation \eqref{eqlin}. Since $w(\cdot,0)$ has a multiple
zero at $x=0$ and $w(\cdot,0)\not\equiv0,$ Lemma
\ref{robustnesszero} implies that there exist $n_0,x_{0},\delta_{0}$ such that
the function $w_{n_0}(\cdot,\delta_{0})$ has a multiple zero at
$x=x_0.$ Consequently,
\begin{equation}\label{eq:2.17}
  \tau\left( U(\cdot,t_{n_0}+\delta_0)\right)\cap\mathcal{O}_{n_0}\neq\emptyset.
\end{equation}
However, since $\mathcal{O}_{min}\subset\mathcal{I}({O}_{n_0}),$
\eqref{eq:2.17} contradicts \eqref{eq:2.16}.
This contradiction concludes the proof of
Lemma \ref{squeezinglemma} in the case
$\displaystyle K=\underset{t\in\R}{\cup}\tau\left( U(\cdot,t)\right)$.
 
If $K$ is any of the sets $\tau\left(\omega(U)\right)$,
$ \tau\left(\Omega(U)\right)$, $\tau\left(\alpha(U)\right)$,
$\tau\left( A(U)\right)$, then the conclusion follows from the
previous case and the invariance properties of these sets. Indeed,
consider $K=\tau\left(\omega(U)\right)$ for instance, the other cases
being similar.  For any $\vp\in\omega(U)$ there is an entire solution
$\tilde{U}$ with $\tilde{U}(\cdot,t)\in\omega(U)$ for all $t$ and
$\tilde{U}(\cdot,0)=\vp.$ Then
$\displaystyle
\tilde{K}:=\underset{t\in\R}{\cup}\tau\left(\tilde{U}(\cdot,t)\right)$
satisfies the hypotheses of the first case, and so $\tilde{K}$ is
included in ${\overline{\mathcal{I}}\left(\Sigma_{in}\right)}.$ This
is true for all $\vp\in\omega(U),$ hence
$\displaystyle
\tau\left(\omega(U)\right)\subset{\overline{\mathcal{I}}\left(\Sigma_{in}\right)}.$
\end{proof}

Finally, we recall the following well known result concerning the
solutions in $\cV$ (the proof can be found in \cite[Theorem
5.5.2]{Volpert3}, for example).
\begin{lemma}\label{le:3.1}
  Assume that $u_0\in \cV$. Then the limits
  \begin{equation}\label{limits-t}
    \theta_-(t):=\lim_{x\to-\infty}u(x,t),\qquad \theta_+(t):=\lim_{x\to\infty}u(x,t)
  \end{equation}
  exist for all $t> 0$ and are solutions of the following
  initial-value problems:
  \begin{equation}\label{eq:3.2}
    \dot{\theta}_\pm=f(\theta_\pm),\qquad \theta_\pm(0)=u_0(\pm\infty).
  \end{equation}
\end{lemma}

\section{Spatial trajectories of entire solutions in
  $\omega(u)$}\label{sec:3}

Throughout this section, we assume, in addition to the standing
hypotheses (ND), (MF) on $f$, that $u_0\in \cV$, $u_0(\pm\infty)=0$,
and the solution of \eqref{eq:1}, \eqref{ic1} is bounded. We
reserve the symbol $u(x,t)$ for this fixed solution.

Due to \eqref{eq:3.2}, the limits \eqref{limits-t} are equal:
$\theta^+\equiv\theta^-=:\hat \theta$, where $\hat\theta$ is the
solution of \eqref{eq:ODE} with $\hat\theta(0)\equiv 0$. This gives
the first two statements of the following corollary; the last
statement follows from Lemma \ref{lemzero}.
\begin{cor}\label{corollary:3.2}
  If $f(0)=0$, then $\hat \theta\equiv 0$; we set $\theta:=0$ in this
  case. If $f(0)\ne 0$, then
  $\hat\theta(t)\to\theta\in\R$ as ${t\to\infty}$, where
  $\theta\in f^{-1}\{0\}$ is a stable equilibrium of
  \eqref{eq:ODE}. In either case, if $\psi$ is any periodic steady
  state of \eqref{eq:steady} such that $\theta$ is not in the range of
  $\psi,$ then there exists $T>0$ such that
  $\displaystyle z(u(\cdot,t)-\psi)<\infty$ for all $t>T.$
\end{cor}

Following the outline given in Section \ref{outline}, we examine the
elements $\varphi$ of $\om(u)$ whose spatial trajectories are not
contained in any chain. At the end, we want to show that no such
elements of $\om(u)$ exist (an application of Lemma \ref{le:2.7} then
yields the desired quasiconvergence results). For that aim, we first
examine the entire solutions through such elements
$\varphi\in \om(u)$. In Proposition \ref{prop:3.3} below, we
expose a certain structure these entire solutions would necessarily
have to have. Then, in Section \ref{morse-dec}, we show that that
structure is incompatible with other properties of the $\om$-limit
set.

Up to a point, we treat the cases (S) and (U) simultaneously.  When
(U) holds, we sometimes have to assume one or both of the extra
conditions (NC), (R); we indicate when this is needed.  The following
notation will be used in the case (U): $\Pi_0$ is the connected
component of $\mathcal{P}_0$ whose closure contains $(0,0)$. Note that
$\Pi_0$ is well defined, for $f'(0)>0$ implies that $(0,0)$ is a
center for \eqref{eq:sys}.

\begin{prop}\label{prop:3.3}
  Under the above hypotheses, assume that $\vp\in\omega(u)$ and let
  $U$ be the entire solution of \eqref{eq:1} with $U(\cdot,0)=\vp.$
  Assume that $\tau(\vp)\cap\mathcal{P}_0\neq\emptyset$, so that there
  exists a connected component $\Pi$ of $\mathcal{P}_0$ with
  \begin{equation}\label{eq:3.3}
    \tau(\vp)\cap\Pi\neq\emptyset.
  \end{equation}
  If {\rm (S)} holds, or if {\rm (U)} holds and $\Pi\ne\Pi_0$, then the following
  statements are true.
  \begin{itemize}
  \item[(i)] The connected component $\Pi$ satisfying \eqref{eq:3.3}
    is unique, it is bounded, and
    \begin{equation}
      \label{eq:5}
      \underset{t\in\R}{{\textstyle \bigcup}}\tau\left( U(\cdot,t)\right)\subset\Pi.
    \end{equation}
  \item[(ii)] Let $\Sigma_{in}=\Sigma_{in}(\Pi)$ be the inner chain
    and $\Lambda_{out}=\Lambda_{out}(\Pi)$ the outer loop associated
    with $\Pi,$ as in Lemma \ref{le:p0}.  Then
    \begin{equation}
      \label{eq:6}
      \tau\left(\alpha(U)\right)\subset\Sigma_{in},\qquad \tau\left(\omega(U)\right)\subset\Lambda_{out}.
    \end{equation}
  \end{itemize}
  If {\rm (U)} holds and $\Pi=\Pi_0$, then statement (i) is true if
  condition {\rm (NC)} is satisfied, and statement (ii) is true if
  conditions {\rm (NC)} and {\rm (R)} are both satisfied.
\end{prop}

Statement (i) is proved in the next subsection. Subsection
\ref{existlim} is devoted to the   
behavior at $x\approx\pm\infty$ of $U(\cdot,t),$ and subsection
\ref{sub33} to additional properties  
when {\rm (U)} holds. Statement (ii) is then proved in sections
\ref{sec:entire} and \ref{completionprop}.

For the remainder of this section, we fix $\vp\in\omega(u)$ and denote
by $U$ be the entire solution of \eqref{eq:1} such that
$U(\cdot,0)=\vp$ and $U(\cdot,t)\in\omega(u)$ for all $t.$ Recall from
Section \ref{sec:inv} that there exists a sequence
$t_n\to\infty$ such that
\begin{equation}
  \label{eq:4}
  u(\cdot,\cdot+t_n)\underset{n\to\infty}{\longrightarrow}U \textrm{
    in } C_{loc}^1(\R^2). 
\end{equation}

\subsection{No intersection with chains}
\label{part1}
In this subsection, we prove statement (i) of Proposition
\ref{prop:3.3}.  The following result is a first step toward that
goal. 
\begin{lemma}\label{le:3.4}
  Let $\Sigma\subset \R^2$ be a nontrivial chain. If
  $\tau(\vp)\cap \mathcal{I}(\Sigma) \neq\emptyset$, then
  \begin{equation}
    \label{eq:44}
    \tau(U(\cdot,t))\subset \mathcal{I}(\Sigma)\quad(t\in\R)
  \end{equation}
  (in particular, $\tau(U(\cdot,t))\cap\Sigma=\emptyset$ for all
  $t\in\R$).  The result remains valid if one considers a loop
  $\Lambda$ in place of the chain $\Sigma.$
\end{lemma}
\begin{proof} Note that the second statement is a consequence of the
  first one; just consider the chain $\Sigma$ containing the loop
  $\La$ and use the connectedness of the set
  $\cup_{t\in\R}\tau(U(\cdot,t))$.

  Let $\Sigma$ be a nontrivial chain. We first show that the
  assumption $\tau(\vp)\cap \mathcal{I}(\Sigma) \neq\emptyset$ implies
  \begin{equation}
    \label{eq:46}
    \tau(\vp)\cap\Sigma=\emptyset.
  \end{equation}
  Assume for a contradiction that the intersections are both
  nonempty. The relation
  $\tau(\vp)\cap \Sigma\neq\emptyset$ means that there is
  a steady state $\phi$ of \eqref{eq:1} such that
  $\tau(\phi)\subset\Sigma$ and $U(\cdot,0)-\phi=\vp-\phi$ has a
  multiple zero at some point $x_0$. Using both assumptions
  $\tau(\vp)\cap\Sigma\neq\emptyset$ and
  $\tau(\vp)\cap \mathcal{I}(\Sigma) \neq\emptyset$, together with the
  connectedness of $\tau(\vp)$ and the fact that distinct chains are
  disjoint with positive distance,  we find a periodic orbit
  $\mathcal{O}\subset\mathcal{P}_0\cap\mathcal{I}(\Sigma)$ such that
  $\tau(U(\cdot,0))\cap\mathcal{O}\neq\emptyset.$ Hence there is a
  steady state $\psi$ of \eqref{eq:1} such that
  $\tau(\psi)=\mathcal{O}$ (so $\psi$ is nonconstant and periodic) and
  $U(\cdot,0)-\psi=\vp-\psi$ has a multiple zero at some $x_1.$
  Obviously, $U(\cdot,0)-\phi\not\equiv 0\not\equiv U(\cdot,0)-\psi$.

  From $\tau(\phi)\subset\Sigma$, we infer that $\phi$ is either a
  ground state at some level $a\in f^{-1}\{0\}$, or a standing wave
  with some limits $a, b\in f^{-1}\{0\}$, or a constant steady state
  $a$. We just consider the first possibility, the other two being
  similar.  Assuming that $\phi$ is a ground state at level $a,$ we
  first show that necessarily $a=0$ (and consequently $0$ is a stable
  equilibrium of \eqref{eq:ODE}, cp. Sect. \ref{sec:22}).  Indeed, the
  function $w(x,t)=U(x,t)-\phi(x)$ is a nontrivial solution of a
  linear equation \eqref{eqlin} and $w(\cdot,0)$ has a multiple zero
  at $x=x_0.$ By \eqref{eq:4}, the sequence
  $w_n:=u(\cdot,\cdot+t_n)-\phi$ converges in $C^1_{loc}(\R^2)$ to
  $w.$ Therefore, by Lemma \ref{robustnesszero}, there exist sequences
  $x_n\to x_0,$ $\delta_n\to 0$ such that for all sufficiently large
  $n$ the function $w_n(\cdot,\delta_n)$ has a multiple zero at
  $x=x_n.$ In other words, $u(\cdot,t_n+\delta_n)-\phi$ has a multiple
  zero at $x=x_n.$ Since $ t_n+\delta_n\to\infty$ and
  $u(\cdot,t)\not\equiv\phi$ (due to the assumption
  $\tau(\vp)\cap \mathcal{I}(\Sigma) \neq\emptyset$), Lemma
  \ref{lemzero} implies that $z(u(\cdot,t)-\phi)=\infty$ for all
  $ t>0.$ Now, $\phi(\pm\infty)=a$ and
  $u(\pm\infty,t)=\hat \theta(t)$, where $f(a)=0$ (and $a$ is a stable
  equilibrium of \eqref{eq:ODE}) and $\hat\theta$ is a solution of
  \eqref{eq:ODE}.  If $\hat\theta(t)\ne a$ for some (hence any) $t$,
  then, by Lemma \ref{lemzero}, $z(u(\cdot,t)-\phi)<\infty$. Thus,
  necessarily, $\hat\theta \equiv a$, which shows that $a=0$, as
  desired.

  To conclude, we use the fact that $(0,0)=(a,0)$ belongs to
  $\Sigma$, as does $\tau(\phi)$.  Since $\Sigma$ is connected and
  $\tau(\psi)\subset\mathcal{I}(\Sigma),$ we have either $\psi>0$ or
  $\psi<0.$ Therefore, by Corollary \ref{corollary:3.2},
  $z(u(\cdot,t)-\psi)<\infty$ for all large enough $t$.  On the other
  hand, using Lemma \ref{robustnesszero} in a similar way as above,
  since $U(\cdot,t_1)-\psi$ has a multiple zero and
  $U(\cdot,t_1)\not\equiv\psi,$ we obtain that
  $z(u(\cdot,t)-\psi)=\infty$ for all $t>0$. This contradiction
  completes the proof of \eqref{eq:46}.

  Using \eqref{eq:46}, the connectedness of $\tau(\varphi)$, and the
  assumption $\tau(\vp)\cap \mathcal{I}(\Sigma) \neq\emptyset$ we
  obtain that $\tau(\vp)\subset \mathcal{I}(\Sigma)$.  The stronger
  statement \eqref{eq:44} follows from this. Indeed, if \eqref{eq:44}
  is not valid, then for some $t_1$, we have
  $\tau(U(\cdot,t_1))\subset \ol{\mathcal{I}}(\Sigma)$ and
  $\tau(U(\cdot,t_1))\cap\Sigma\ne\emptyset$. At the same time,
  $\tau(U(\cdot,t_1))\not\subset \Sigma$ (otherwise, $U\equiv\varphi$
  is a steady state, by Lemma \ref{le:2.7}, and then
  $\tau(\varphi)\subset \Sigma$ would contradict the
  assumption). Thus,
  $\tau(U(\cdot,t_1))\cap {\mathcal{I}}(\Sigma)\ne
  \emptyset$. Applying what we have already proved to
  $U(\cdot,t_1)\in \om(u)$ in place of $\varphi$, we obtain
  $\tau(U(\cdot,t_1))\cap\Sigma=\emptyset$, a contradiction.  The
  proof is now complete.
\end{proof}

The next result is analogous to the previous one, but the proof
requires different arguments. 
\begin{lemma}
  \label{le:anbut}
  If $\Sigma$ is a nontrivial chain and
  $\tau(\vp)\cap (\R^2\setminus
  \overline{\mathcal{I}}\left(\Sigma\right)) \neq\emptyset$, then
  \begin{equation}
    \label{eq:44b}
    \tau(U(\cdot,t))\subset\R^2\setminus
    \overline{\mathcal{I}}\left(\Sigma\right) 
    \quad(t\in\R)
  \end{equation}
  (in particular, $\tau(U(\cdot,t))\cap\Sigma=\emptyset$ for all
  $t\in\R$).
\end{lemma}
\begin{proof}
  It is sufficient to prove that $\tau(\varphi)\cap\Sigma=\emptyset$.
  The stronger conclusion \eqref{eq:44b} follows from this by a
  similar argument as in the last paragraph of the previous proof.

  We go by contradiction. Assume that
  $\tau(\vp)\cap (\R^2\setminus
  \overline{\mathcal{I}}\left(\Sigma\right)) \neq\emptyset$ and at the
  same time 
  $\tau(\varphi)\cap\Sigma\ne\emptyset$.  Then, there is a solution
  $\phi$ of \eqref{eq:steady} with $\tau(\phi)\subset\Sigma$ such that
  $U(\cdot,0)-\phi=\vp-\phi$ has a multiple zero at some $x=x_0.$
  Clearly, $\phi$ is either a ground state, or a standing wave, or a
  zero of $f$ which is the limit of some ground state or standing
  wave. In either case, $\tau(\phi)$ is contained in a loop
  $\Lambda\subset\Sigma$.  To derive a contradiction, we choose a
  sequence $\psi_n$ of periodic solutions of \eqref{eq:steady} such
  that $\cO_n:=\tau(\psi_n)\subset \cI(\La)$ and $\psi_n\to\phi$ in
  $C^1_{loc}(\R)$ (the existence of such a sequence of periodic orbits
  $\cO_n$ is guaranteed by Lemma \ref{le:p0} and the fact that
  distinct chains are disjoint with positive distance).  Then the
  sequence $w_n(x,t):=U(x,t)-\psi_n(x)$ converges in $C^1_{loc}(\R^2)$
  to $w(x,t):=U(x,t)-\phi,$ a solution of a linear equation
  \eqref{eqlin}.  Since $U(\cdot,0)=\varphi\not\equiv \phi$, we have that
  $w\not\equiv 0$. Moreover, $w(\cdot,0)$ has a multiple zero at
  $x=x_0.$ Hence, by Lemma \ref{robustnesszero}, there exist $n_1$,
  $t_1$ such that $w_{n_1}(\cdot,t_1)$ has a multiple zero.  This
  means that $\tau(U(\cdot,t_1))\cap\mathcal{O}_{n_1}\neq\emptyset$,
  hence $\tau(U(\cdot,t_1))\cap \cI(\La)\ne\emptyset$. Applying Lemma
  \ref{le:3.4} to $U(\cdot,t_1)$ in place of $\varphi$, and taking
  $t=0$ in \eqref{eq:44}, we obtain a contradiction to the assumption
  that
  $\tau(\vp)\cap (\R^2\setminus
  \overline{\mathcal{I}}\left(\Sigma_{in}\right)) \neq\emptyset$.
\end{proof}

In the next lemma, we deal with a trivial chain $\Sigma$, that is,
$\Sigma=\{(\be,0)\}$, where $\beta$ is an unstable equilibrium of
\eqref{eq:ODE}.

\begin{Lemma}
  \label{le:trivchain}
  Assume that $\Sigma=\{(\be,0)\}$ is a trivial chain.  If $\be=0$ (so
  {\rm (U)} holds), assume also that condition {\rm (NC)} is satisfied.  Then
  $(\beta,0)\not\in\tau(\varphi)$, which is the same as
  $\tau(\vp)\cap\Sigma=\emptyset$, unless $U\equiv \varphi\equiv \be=0$.
\end{Lemma}
\begin{proof}
  Assume that $(\beta,0)\in\tau(\varphi)$.  If
  $\varphi\not \equiv \beta$, then $U-\be$ is a nontrivial solution
  of a linear equation. Hence, by Lemma \ref{robustnesszero},
  there is a sequence $t_n\to\infty$ such that $u(\cdot,t_n)-\beta$
  has a multiple zero for $n=1,2,\dots$. Then by Lemma \ref{lemzero},
  \begin{equation}
    \label{eq:45}
    z(u(\cdot,t)-\beta)=\infty\quad(t>0),
  \end{equation}
  which is possible only if $u(\pm\infty,t)=\be$ for all $t>0$. Since
  $\beta$ is an unstable equilibrium of \eqref{eq:ODE}, this relation
  means that (U) holds and $\be=0$.  However, in this situation
  assumption (NC) is in effect, which clearly contradicts
  \eqref{eq:45}. This contradiction shows that for
  $(\beta,0)\in\tau(\varphi)$ it is necessary that
  $U\equiv \varphi\equiv \be=0$.
\end{proof}

We are ready to complete the proof of Proposition
\ref{prop:3.3}(i).

\begin{proof}[Proof of Proposition \ref{prop:3.3}(i)]
  Assume that \eqref{eq:3.3} holds for a connected component $\Pi$ of
  $\cP_0$. We first claim that $\Pi$ is bounded. Suppose not.  Since
  $u(\cdot,t)$ and $u_x(\cdot,t)$ are uniformly bounded as
  $t\to\infty$, using (MF) we find a periodic orbit
  $\mathcal{O}\subset\Pi$ such that
  $\displaystyle \tau(U(\cdot,t))\subset\mathcal{I}(\mathcal{O}),$ for
  all $t\in\R.$ Then Lemma \ref{squeezinglemma} implies that
  $\displaystyle
  \tau(\vp)=\tau(U(\cdot,0))\subset{\overline{\mathcal{I}}\left(\Sigma_{in}(\Pi)\right)},$
  in contradiction to \eqref{eq:3.3}.

  Thus $\Pi$ is indeed bounded. Let $\Sii$ and $\Lao$ be the inner
  chain and outer loop associated with $\Pi$ (as in Proposition
  \ref{prop:3.3}(ii)). If $\Pi\ne\Pi_0$, then Lemmas
  \ref{le:3.4}--\ref{le:trivchain} show that
  $\tau(U(\cdot,t))\subset\Pi$ for all $t\in\R$.  The same applies if
  $\Pi=\Pi_0$---in which case $\Sii=\{(0,0)\}$---under the extra
  assumption (NC).  This in particular shows the uniqueness of $\Pi$
  satisfying \eqref{eq:3.3}.
\end{proof}

We finish the subsection with a result ruling out some functions,
including all nonconstant periodic steady states, from
$\om(u)$. 

\begin{lemma}\label{le:3.5}
  Let $\psi$ be a  nonconstant periodic solution of
  \eqref{eq:steady} and $\cO:=\tau(\psi)$. The
   following statements are valid.
  \begin{itemize}
  \item[(i)] If $(0,0)\not \in \cI(\cO)$, then $\om(u)$  contains
    no function $\phi$ satisfying $\tau(\phi) \cap
    \ol\cI(\cO)\ne\emptyset$. In particular, $\om(u)$ does not contain
    $\psi$ itself and neither it contains any
    nonzero  $\be\in f^{-1}\{0\}$ which is an unstable equilibrium of
  \eqref{eq:ODE}. 
\item[(ii)] If $(0,0) \in \cI(\cO)$ and either {\rm (S)} holds, or
  {\rm (U)} holds together with {\rm (NC)}, 
    then $\psi\not \in\om(u)$.
  \end{itemize}
\end{lemma}
\begin{proof} First we prove that $\psi\not \in\om(u)$. 
  By Lemma \ref{le:p0}(iv), there is
  $\beta\in f^{-1}\{0\}$ such that $\be$ is an unstable equilibrium of
  \eqref{eq:ODE} and $z(\psi-\beta)=+\infty$.  Obviously, all zeros of
  $\psi-\beta$ are simple. Hence, if $\psi\in \omega(u)$,
  then $z(u(\cdot,t_n)-\beta)\to \infty$ for some sequence
  $t_n\to\infty$.  This is not possible, by Lemma \ref{lemzero}, if
  $\be\ne 0$.  Neither is it possible if $\be= 0$---which, due to the
  instability of $\be$, would mean that (U) holds---if (NC)
  holds, for  Lemma \ref{lemzero} implies that
  $z(u_x(x,t))$ is finite and bounded uniformly in $t>0$.
  We have thus proved that statement (ii) holds, and also that
  $\psi\not\in \omega(u)$ if $(0,0)\not \in \cI(\cO)$.

  To complete the proof of statement (i), assume that
  $(0,0)\not \in \cI(\cO)$. Suppose for a contradiction that
  $\om(u)$  contains a function $\phi$ satisfying $\tau(\phi) \cap
  \ol\cI(\cO)\ne\emptyset$.

  First we find a contradiction if $\phi$
  is a steady state of \eqref{eq:1}.
  Note that in this case, the assumption $\tau(\phi) \cap
  \ol\cI(\cO)\ne\emptyset$ implies that either $\tau(\phi) \subset 
  \cI(\cO)$ or $\phi$ is a shift of $\psi$. 
  If $\phi$ is a
    nonconstant periodic solution, we just use the result proved above
    with $\psi$ replaced by $\phi$ to obtain that $\phi\not\in
    \omega(u)$. If  $\phi$ is not a nonconstant periodic solution,
    then it is a constant, or a ground state, or a standing wave. In
    any case, one has $(\phi(x),\phi'(x))\to (\vartheta,0)$ as
    $x\to\infty$, where $(\vartheta,0)$ is an equilibrium of
    \eqref{eq:sys}. Obviously, $(\vartheta,0)\in \cI(\cO)$. This
    implies that $z(\psi- \phi)=+\infty$. Clearly, all zeros of
    $\psi-\phi$ are simple. Since
    $\phi\in\om(u)$, there is a sequence
    $t_n\to\infty$ such that
    $z(u(\cdot,t_n)-\psi)\to \infty$.  However, due to
    the assumption that $(0,0)\not \in \cI(\cO)$,
    $0$ is not in the range of $\psi$. So, by Lemma \ref{le:zerob},
    $z(u(\cdot,t)-\psi)$ is finite and uniformly bounded as
    $t\to\infty$, and we have a contradiction.

    Next we derive a contradiction if $\phi$ is not a steady state and
    $\tau(\phi)\cap\tau(\psi)\ne \emptyset$. Take the entire solution
    $\tilde U$ of \eqref{eq:1} with $\tilde U(\cdot,0)=\phi$.
    The assumption on $\phi$ implies that
    replacing $\psi$ by a suitable shift if necessary,
    $\phi-\psi=U(\codt,0)-\psi$ has a multiple
    zero. We have $\phi\not\equiv \psi$, as $\psi\not\in \om(u)\ni \phi$.
    Also we know, cp. \eqref{eq:4}, that there is a  sequence $\tilde
    t_n\to\infty$ such that   
    $u(\cdot,\cdot+\tilde t_n)\to \tilde U$ in $ C_{loc}^1(\R^2)$.
    Therefore, by Lemma \ref{robustnesszero},
    there is  a sequence $\tau_n\to0$ such that
    $u(\cdot,\cdot+\tilde t_n+\tau_n)-\psi$ has a multiple
    zero. Consequently, since $\tilde t_n+\tau_n\to\infty$,
    $z(u(\cdot,t)-\psi)=\infty$ for all $t>0$ and this is a
    contradiction as in the previous case.

    It remains to find a
    contradiction when $\tau(\phi)\subset \cI(\cO)$ and $\phi$ is not a
    steady state. In this case,  there is a periodic orbit $\tilde
    \cO\subset \cI(\cO)$ such that $\tilde \cO\cap
    \tau(\phi)\ne\emptyset$. Using the previous argument, with $\cO$
    replaced by $\tilde \cO$, we obtain a contradiction in this case as
    well.

    Finally, if $\be\ne 0$ is an unstable equilibrium of
    \eqref{eq:ODE}, then,  by (ND), $f'(\be)>0$.
    This implies that $(\be,0)$ is a center
  for \eqref{eq:sys}, hence any neighborhood of $(\be,0)$ contains a
  periodic orbit $\tilde\cO=\tau(\tilde\psi)$ of \eqref{eq:sys} satisfying
  $(\be,0)\in \cI(\tilde\cO)$. We can choose such a periodic orbit
  so that $(0,0)\not \in \cI(\tilde\cO)$. Then, taking
  $\phi\equiv \be$ and using what we have already proved
  of statement (i) (with $\psi$
  replaced by $\tilde \psi$), we conclude that $\be\not\in\om(u)$. 
\end{proof}

\subsection{Existence of the limits at spatial infinity}
\label{existlim}
In this subsection, we assume that $\Pi$ is as in \eqref{eq:3.3}; and
$\Sigma_{in}=\Sigma_{in}(\Pi)$, $\Lambda_{out}=\Lambda_{out}(\Pi)$, as
in Proposition \ref{prop:3.3}(ii).  If (U) holds and
$\Pi=\Pi_0$, also assume that (NC) holds.

Recall that we have fixed $\vp\in\omega(u)$ and denoted
by $U$ be the entire solution of \eqref{eq:1} with
$U(\cdot,0)=\vp$.
By Proposition \ref{prop:3.3}(i),  $U$
satisfies \eqref{eq:5}.

As a first step toward the proof of statement (ii) of Proposition
\ref{prop:3.3}, we show that the limits $U(\pm \infty,t)$ exist
and, at least when $\Pi\ne\Pi_0$ are independent of $t$.

Recall from Section \ref{sec:22} that $\Sigma_{in}$, as any other
chain, has the following structure:
\begin{equation}\label{Sigmain}
  \Sigma_{in} = \left\{(u,v)\in\R^2:u\in J,\ v=\pm\sqrt{2\left( F(p)-F(u)\right)}\right\},\ J= [p,q],
\end{equation}
for some $p\le q$.  If $p=q$, then $\Sigma_{in}$ is trivial: it
reduces to a single equilibrium $(p,0)$.  Necessarily, $p$ is an
unstable equilibrium of \eqref{eq:ODE} in this case.  If $p<q,$ then
$(p,0)$ and $(q,0)$ lie on (distinct) homoclinic orbits and
\begin{equation}\label{eq:3.6}
  f(p)=F'(p)<0,\quad f(q)= F'(q)>0; \qquad F(u)\le F(p)\quad(u\in [p,q]).
\end{equation}
We define
\begin{equation}\label{betaplusmoins}
  \beta_- := \min\{\beta\in[p,q],\ f(\beta)=0\}, \qquad
  \beta_+ := \max\{\beta\in[p,q],\ f(\beta)=0\}.
\end{equation}
These are well-defined finite quantities, as every chain contains
equilibria (finitely many of them, by (ND)).  Of course, if $p=q$,
then $\be_-=\be_+=p$. Otherwise, $(\beta_-,0)$, $(\beta_+,0)$ are
contained in the interiors of distinct homoclinic loops. This implies
that $p<\beta_-<\beta_+<q$ and
\begin{equation}\label{A1beta}(\be_\pm,0)\in \cI(\Sigma_{in}); \text{
    in particular, }
  (\beta_-,0)\not\in\Sigma_{in},\ (\beta_+,0)\not\in\Sigma_{in}.
\end{equation}
Also, the definition of $\beta_-$, $\beta_+$ and \eqref{eq:3.6} imply
that $\beta_-$, $\beta_+$ are unstable equilibria of \eqref{eq:ODE}.

In the following lemma, we relate $\be_\pm$ and the limit
$\theta=\lim_{t\to\infty}u(\pm\infty,t)$ (cp. Corollary
\ref{corollary:3.2}).
\begin{lemma}\label{le:3.7} The following statements hold:
  \begin{itemize}
  \item[(i)] If $p=q$ (that is, $\Sigma_{in}=\{(p,0)\}$), then
    necessarily $\be_\pm =p=0$ ($=u_0(\pm\infty)$) and so {\rm (U)} holds
    and $\Pi=\Pi_0$.
  \item[(ii)] If $0$ is a stable equilibrium of \eqref{eq:ODE} or if
    $f(0)\ne 0$ (so that $\theta\ne 0$), then
    \begin{equation}
      \label{eq:7a}
      \beta_-<\theta <\beta_+.
    \end{equation}
  \end{itemize}
\end{lemma}

\begin{proof}
  Pick a periodic orbit $\mathcal{O}\subset\Pi$ such that
  $\tau(\vp)\cap\mathcal{O}\neq\emptyset$, and let $\psi$ be a
  periodic solution of \eqref{eq:steady} with $\cO=\tau(\psi)$. Then,
  possibly after replacing $\psi$ by a suitable shift,
  $U(\cdot,0)-\psi=\vp-\psi$ has
  a multiple zero. By Lemma \ref{le:3.5}, $\vp\not\equiv
  \psi$. Applying Lemma \ref{robustnesszero}, we find a sequence
  $t_n\to\infty$ such that $u(\cdot,t_n)-\psi$ has a multiple zero,
  and then it follows from Lemma \ref{lemzero} that
  $z(u(\cdot,t)-\psi)=\infty $ for all $t>0$.  Corollary
  \ref{corollary:3.2} now tells us that $\theta$ must be in the range
  of $\psi$. Hence, by the definition of $\Sigma_{in},$
  \begin{equation*}
    (\theta,0)\in{\overline{\mathcal{I}}(\Sigma_{in})}.
  \end{equation*}
  This and the definition of $\be_\pm$ give
  \begin{equation}
    \label{eq:7}
    \beta_-\le\theta\le \beta_+.
  \end{equation}
  If $p=q$, then $\Sigma_{in}=\{(p,0)\}$, so $\theta =p$, and $\theta$
  is an unstable equilibrium of \eqref{eq:ODE}. By Corollary
  \ref{corollary:3.2}, $\theta=0$ and statement (i) is proved.

  Assume now that $0$ is a stable equilibrium of \eqref{eq:ODE} or
  $f(0)\ne 0$. In both cases, $\theta$ is a stable equilibrium of
  \eqref{eq:ODE}. Also, the case $p=q$ is ruled out. The stability of
  $\theta$ and the instability of $\be_\pm$ imply that \eqref{eq:7}
  holds with the strict inequalities, completing the proof of statement
  (ii).
\end{proof}

\begin{remark}\label{rm:emp}{\rm
    Note that $\hat \theta(t)=u(\pm \infty,t)$, being a solution of
    \eqref{eq:ODE}, 
    cannot go across an equilibrium
    of \eqref{eq:ODE}. Thus  \eqref{eq:7a} implies
    that
    \begin{equation}
      \label{eq:15}
      \beta_-< u(\pm \infty,t)<\beta_+\quad (t\ge 0), \quad\text{in
        particular,}\ \beta_-<0<\beta_+.
    \end{equation}
  }\end{remark}
We can now prove the existence of the limits 
\begin{equation}\label{eq:29}
  \Theta_-(t):=\underset{x\to-\infty}{\lim}U(x,t),\qquad \Theta_+(t):=\underset{x\to\infty}{\lim}U(x,t).
\end{equation}
\begin{Lemma}\label{le:limits} The limits \eqref{eq:29}
  exist for all $t\in\R$. Moreover, the following statements hold:
  \begin{itemize}
  \item[(i)] If $p<q$ (that is, $\Sigma_{in}$ is a nontrivial chain),
    then $\Theta_-(t)$, $\Theta_+(t)$ are independent of $t$, and
    their constant values, denoted by $\Theta_-$, $\Theta_+$, satisfy
    $(\Theta_\pm,0)\in \Sigma_{in}\cup\La_{out}$. Also, $\Theta_\pm$
    are stable equilibria of $\eqref{eq:ODE}$.
  \item[(ii)] If $p=q$ (that is, $\Sigma_{in}=\{(0,0)\}$, condition
    {\rm (U)} holds, and $\Pi=\Pi_0$), then $\Theta_-(t)$ is either independent of
    $t$ and its constant value $\Theta_-$ satisfies
    $(\Theta_-,0)\in \{(0,0)\}\cup\La_{out}$, or it is a strictly
    monotone solution of \eqref{eq:ODE} with $\Theta_-(-\infty)=0$ and
    $(\Theta_-(\infty),0)\in\La_{out}$.  The same is true for
    $\Theta_+(t)$.
  \end{itemize}
\end{Lemma}
\begin{proof}
  From \eqref{eq:5} we in particular obtain that $\tau(U(\cdot,t))$ cannot
  intersect the $u-$axis between $p$ and $q$, or, in other words,
  \begin{equation}\label{eq:27}
    U_x(x,t)\neq0\text{ whenever }U(x,t)\in [p,q].
  \end{equation}

  Assume now that for some $t=t_0$ one of the limits in \eqref{eq:29},
  say the one at $\infty$, does not exist:
  \begin{equation*}
    \ol \ell := \limsup_{x\to\infty} U(x,t_0)>\ul \ell :=
    \liminf_{x\to\infty}  U(x,t_0). 
  \end{equation*}
  Then there is a sequence $\ol x_n$ of local-maximum points of
  $U(\cdot,t_0)$ and a sequence $\ul x_n$ of local-minimum points of
  $U(\cdot,t_0)$, such that $\ol x_n\to\infty$, $\ul x_n\to\infty$,
  and
  \begin{equation}
    \label{eq:13}
    U (\ol x_n,t_0)\to \ol\ell,\qquad
    U (\ul x_n,t_0)\to \ul\ell.
  \end{equation}
  In view of \eqref{eq:27}, we may also assume, passing to a
  subsequence if necessary, that either $p> U (\ul x_n,t_0)$ for all
  $n$ or $U (\ol x_n,t_0)>q$ for all $n$. We assume the former, the
  latter can be treated similarly. Obviously, we also have
  \begin{equation}
    \label{eq:14}
    p\ge  \ul \ell \ge \hat p:=\inf \{(u: (u,0)\in \Pi\}.
  \end{equation}
  Observe that there is no zero of $f$ in $(\hat p,\be_-)$, and the
  instability of $\be_-$ implies $f<0$ in $(\hat p,\be_-)$.

  Pick $\xi_0>\ul\ell$ so close to $\ul\ell$ that also
  $\xi_0<\min\{\ol \ell,\be_-\}$.  Clearly, each of the functions
  $U (\cdot,t_0)-\xi_0$ and $U_x (\cdot,t_0)$ has infinitely many sign
  changes, which implies, by \eqref{eq:4}, that
  $z(u(\cdot,t_0+t_n)-\xi_0)\to\infty$ and
  $z(u_x(\cdot,t_0+t_n))\to\infty$ as $n\to\infty$. The latter
  immediately gives contradiction if $p=q=0$. Indeed, in this case
  $\Pi=\Pi_0$, so condition (NC) is in effect, which implies, by Lemma
  \ref{lemzero}, that $z(u_x(\cdot,t))$ is finite and bounded as
  $t\to\infty$.  If $p<q$, we employ the former. Take the solution
  $\xi(t)$ of \eqref{eq:ODE} with $\xi(t_0)=\xi_0$. Since $f<0$ in
  $(\hat p,\be_-)$, we have $\xi(t)\upto\be_-$ as $t\to-\infty$.  The
  monotonicity of the zero number gives
  $z(u(\cdot,s)-\xi(s-t_n))\to\infty$ as $n\to\infty$ for any
  $s>0$. On the other hand,  by \eqref{eq:15}, the function
  $u(\cdot,s)-\be_-$ has only 
  finitely many zeros, and by Lemma \ref{lemzero} we may fix $s>0$ such
  that all these zeros are simple. Then, since $\xi(s-t_n)\to \be_-$
  as $n\to\infty$ and \eqref{eq:15} holds, for all sufficiently large
  $n$ we have $z(u(\cdot,s)-\xi(s-t_n))=z(u(\cdot,s)-\be_-)$, which yields
  a contradiction.

  Thus, \eqref{eq:29} is proved, and parabolic estimates imply that
  also
  \begin{equation}\label{eq:29b}
    \underset{x\to-\infty}{\lim}U_x(x,t)=0,\qquad
    \underset{x\to\infty}{\lim}U_x(x,t)=0.
  \end{equation}
  It follows that for any $t$ the points $(\Theta_\pm(t),0))$ are
  contained in $\bar \Pi$. If for some $t$ the point
  $(\Theta_-(t),0))$ is equal to an equilibrium $(\eta,0)$ of
  \eqref{eq:sys} in $\Sigma_{in}\cup\La_{out}$, then $\Theta_-(t)$ is
  independent of $t$ (as it is a solution of \eqref{eq:ODE} and
  $f(\eta)=0$). Otherwise, $\Theta_-(t)\le p$ or $\Theta_-(t)\ge q$
  and one shows easily (as for the solution $\xi(t)$ above) that
  $\Theta_-(t)$ converges as $t\to-\infty$ to $\be_-$ or $\be_+$,
  respectively. In this case, we also obtain that either
  $(\be_-,0)\in \bar\Pi$ or $(\be_+,0)\in \bar\Pi$, which can hold
  only if $p=q$.

  We conclude that if $p<q$, then $\Theta_-(t)$ takes a constant value
  $\Theta_-$ for all $t$, and $(\Theta_-,0)$ is an equilibrium of
  \eqref{eq:sys} in $\Sigma_{in}\cup\La_{out}$. The fact that
  $(\Theta_-,0)$ is contained in a nontrivial chain implies that
  $\Theta_-$ is a stable equilibrium of \eqref{eq:ODE}
  (cp. Sect. \ref{sec:22}). This proves statement (i) for
  $\Theta_-(t)$; the proof for $\Theta_+(t)$ is analogous.

  If $p=q$ (and $\Sigma_{in}=(0,0)$), we have proved that
  $\Theta_-(t)$ is either independent of $t$ and $(\Theta_-,0)$ is an
  equilibrium of \eqref{eq:sys} contained in $\{(0,0)\}\cup\La_{out}$, or
  it is a strictly monotone solution of \eqref{eq:ODE} with
  $\Theta_-(-\infty)=0$ ($=\be_\pm$). In the latter case,
  $\Theta_-(t)$ converges as $t\to\infty$ to a zero $\eta$ of $f$ such
  that $(\eta,0)\in \bar \Pi\setminus\{(0,0)\}$. Thus, necessarily,
  $(\eta,0)\in\La_{out}$. The arguments for $\Theta_+(t)$ are
  similar. The proof is now complete.
\end{proof}

\begin{remark}
  \label{rm:limst}
  {\rm Note that we have used the inclusion $U(\cdot,t)\in\om(u)$ to
    prove the existence of the limits \eqref{eq:29} only. Once the
    existence of the limits has been proved, the inclusion was no
    longer used and statements (i), (ii) were derived from
    \eqref{eq:29}, \eqref{eq:5} alone.  }
\end{remark}

\subsection{Additional properties when (U) and (NC) hold}\label{sub33}
As in the previous subsection, we assume that $\Pi$
is as in \eqref{eq:3.3} and
$\Sigma_{in}=\Sigma_{in}(\Pi)$,
$\Lambda_{out}=\Lambda_{out}(\Pi)$, but here
we specifically  assume that  (U) holds and
$\Pi=\Pi_0$. So $\Sigma_{in}$ is the trivial chain $\{(0,0)\}$.
We also assume that (NC) holds.

By Proposition \ref{prop:3.3}(i), the entire solution $U$
satisfies
\begin{equation}
      \label{eq:5c}
      \underset{t\in\R}{{\textstyle \bigcup}}\tau\left( U(\cdot,t)\right)\subset\Pi_0.
\end{equation}

\begin{lemma}
  \label{le:pmnm}
  The following statements are valid.
  \begin{itemize}
  \item[(i)] There is a positive integer $m$ such that for all $t\in
    \R$ one has
    $z(U_x(\cdot,t))\le m$  and all zeros of
    $U_x(\cdot,t)$ are simple. 
  \item[(ii)] For any $t\in \R$, the function $U(\cdot,t)$ has no
    positive local minima and no negative local maxima. 
  \end{itemize}
\end{lemma}
\begin{proof}
 First we prove that all zeros of
    $U_x(\cdot,t)$ are simple. Suppose for a contradiction that $x_0$
    is multiple zero of  $U_x(\cdot,t_0)$  for some $t_0$. By
    parabolic regularity, since $f$ is Lipschitz, the function
    $u_x$ is bounded in $C^{1+\al}(\R^2)$ for some
    $\al\in(0,1)$.
    Therefore, by \eqref{eq:4},
    \begin{equation}
      \label{eq:54}
      u_x(\cdot,\cdot+t_n)\underset{n\to\infty}{\longrightarrow}U_x
    \end{equation}
    in $C_{loc}^1(\R^2)$. It now follows from Lemma
    \ref{robustnesszero}, that there is a sequence $\tau_n\to0$ such that
    $u_x(\cdot,\cdot+ t_n+\tau_n)$ has a multiple
    zero. Consequently, since $ t_n+\tau_n\to\infty$,
    $z(u_x(\cdot,t))=\infty$ for all $t>0$, in contradiction to (NC).

    The simplicity of the zeros of  $U_x(\cdot,t)$ for any $t\in\R$
    is thus proved, and from \eqref{eq:54} and (NC) it follows that the
    other statement in  (i) is valid as well.

    Take now $t_0$ such that the (finite) zero number
    $k:=z(u_x(\cdot,t))$
    is independent of $t$ for $t\ge t_0$ and
    all zeros of  $u_x(\cdot,t)$ simple. Such $t_0$ exists due to (NC)
    and Lemma \ref{lemzero}. Then, for $t>t_0$,
    the zeros of $u_x(\cdot,t)$ are
    given by a $k$-tuple $\xi_1(t)<\dots<\xi_k(t)$, where
    $\xi_1,\dots,\xi_k$ are $C^1$ functions of $t$.

    Observe also that $z(u(\cdot,t))$ is finite for all
    $t>t_0$. Since $f(0)=0$ due to (U),
    $u$ itself is a solution of a linear equation
    \eqref{eqlin}. Therefore,   
    making $t_0$ larger if necessary,
    we may assume that all zeros of $u(\cdot,t)$ are simple for
    $t>t_0$. In particular, $u(\xi_i(t),t)\ne0$ for $t>t_0$,
    $i=1,\dots,k$.

    Let $\xi(t)$ be
    any of the functions  $\xi_1(t),\dots,\xi_k(t)$.
    Since $\xi(t)$ is a simple zero of
    $u_x(\cdot,t)$, it is a local minimum point of $u(\cdot,t)$ for all
    $t>t_0$ or a local maximum  point of $u(\cdot,t)$ for all
    $t>t_0$. Moreover, $u(\xi(t),t)$ does not change sign on
    $(t_0,\infty)$. 

    Assume now that $u(\xi(t),t)$ is a positive local minimum of
    $u(\cdot,t)$ for some---hence any---$t>t_0$. Since
    \begin{equation*}
      \frac{d}{dt}u(\xi(t),t)=u_t(\xi(t),t)=u_{xx}(\xi(t),t)+f(u(\xi(t),t))\ge
      f(u(\xi(t),t)),
    \end{equation*}
    the positivity and boundedness of $u(\xi(t),t)$ imply that
    $\liminf_{t\to\infty}u(\xi(t),t)\ge \ga^+$, where $\ga^+$ is the 
    smallest positive zero of $f$. For any  function
    $\tilde \varphi\in\om(u)$
    this clearly means that if $\tilde \varphi$
    has a positive local minimum
    $m$, then $m\ge \ga^+$. Applying this to
    $\tilde\varphi:=U(\cdot,t)$, for any $t\in\R$, we obtain, since  
    $\tau(U(\codt,t))\subset\Pi_0$, that
    $U(\codt,t)$ can have no positive local
    minimum. Similarly one shows that
    $U(\codt,t)$ does not have any negative local
    maximum. 
\end{proof}

Under condition (R), the critical points of $U(\codt,t)$ stay in a
bounded interval:
\begin{lemma}
  \label{le:bddcp}
  Assume that, in addition to {\rm (U)} and {\rm (NC)}, condition {\rm
    (R)} holds.  
  Then
  there is a constant $d>0$ such that for every $t\in\R$ the critical points of
  $U(\codt,t)$ are all contained in $(-d,d)$.  Moreover,
  the number of the critical points of $U(\codt,t)$ and the number of
  its zeros are both (finite and)   independent of $t$.
\end{lemma}
\begin{proof}
  As in the previous proof, there is $t_0>0$ such that for all $t>t_0$
  the zeros of $u_x(\cdot,t)$ are
    given by a $k$-tuple $\xi_1(t)<\dots<\xi_k(t)$, where
    $\xi_1,\dots,\xi_k$ are $C^1$ functions of $t$.
    Let $\xi(t)$ be
    any of the functions  $\xi_1(t),\dots,\xi_k(t)$.

    Take sequences $a_n\to-\infty$,
    $b_n\to\infty$ as in (R) and let 
    $\la\in \{a_1,a_2,\dots\}\cup\{b_1,b_2,\dots\}$, so 
  $V_\la u(\cdot,t):=u(2\la-\codt,t)-u(\cdot,t)$ has only finitely
  many zeros if $t$ is sufficiently large. Since $x=\la$ is one of
  these zeros, Lemma \ref{lemzero} implies that for all sufficiently
  large $t$ one has  
  $-2\partial_xu(\la,t)=\partial_xV_\la u(\la,t)\neq0$.
  In particular, $\xi(t)\ne \la$ if $t$ is large enough.
  Since this holds for arbitrary $\la\in
  \{a_1,a_2,\dots\}\cup\{b_1,b_2,\dots\}$, 
  it follows 
  that as $t\to\infty$ one has
  either $\xi(t)\to\infty$,  or $\xi(t)\to-\infty$, or else
  $\xi(t)$ stays in a bounded interval.
  Using this, \eqref{eq:54}, and the fact that the zeros of
  $U_x(\codt,t)$ are all simple, we obtain that these zeros are
  contained in a bounded interval $(-d,d)$ independent of $t$.
  It follows from
  the simplicity and boundedness  of the zeros of $U_x(\codt,t)$ that
  their number is independent of $t$.

  As noted in the proof of Lemma \ref{le:pmnm}, the function
  $u(\cdot,t)$ has only simple zeros, a finite number of them, for all
  sufficiently large $t$. Using Lemma \ref{robustnesszero}, similarly
  as in that proof, one shows that for any $t$
  the zeros of $U(\cdot,t)$ are all
  simple. Also, their number is finite,
  as  $z(U_x(\cdot,t))<\infty$, and nonincreasing in $t$.
  The only way the zero number  $z(U(\cdot,t))$ 
  can drop at some $t_0$ is 
  that some of the zeros escape to $-\infty$ or $\infty$ as
  $t\to t_0-$.   This clearly does not happen if $U(-\infty,t_0)\ne
  0$ or $U(\infty,t_0)\ne
  0$, respectively. On the other hand if  $U(-\infty,t_0)= 0$, then
  $U(-\infty,t)= 0$ for all $t$, and in this case the  zeros
  of $U(\codt,t)$ are all greater than the minimal  critical
  point of $U(\codt,t)$. A analogous remark applies in the case
  $U(\infty,t_0)= 0$. Since the set of the critical points is always
  contained in $(-d,d)$, we obtain that $z(U(\codt,t))$ is independent
  of $t$.  
  \end{proof}

\section{A classification of entire solutions with spatial
  trajectories between two chains}
\label{sec:entire}
In the previous section, we considered entire solutions $U$ satisfying
$U(\cdot,t)\in \om(u)$ for all $t\in\R$. We derived certain conditions
any such solution $U$ would have to satisfy, see Proposition
\ref{prop:3.3}(i) and Lemma \ref{le:limits}.  In this section,
we examine the entire solutions with the indicated
properties and classify them in a certain way. Our classification in
particular proves Proposition \ref{prop:3.3}(ii) under the
extra assumption that $\Sigma_{in}$ is a nontrivial chain. We
stress, however, that no reference is made in this section to the
solution $u$ or its limit set $\om(u)$. Thus the results here are
completely independent from the previous and forthcoming sections and
can be viewed as contributions to the general understanding of entire
solutions of \eqref{eq:1}.

Our assumptions throughout this section are as follows.  We assume
that the standing hypotheses (ND), (MF) on $f$ hold, $\Pi$ is a
bounded connected component of $\cP_0$, and
$\Sigma_{in}:=\Sigma_{in}(\Pi)$, $\Lambda_{out}:=\Lambda_{out}(\Pi)$.
The next standing hypotheses delineates the class of entire solutions
we consider:
\begin{description}
\item[\bf (HU)] \ $U$ is a bounded entire solution of \eqref{eq:1} such that
  \begin{equation}
    \label{eq:5b}
    \tau\left( U(\cdot,t)\right)\subset \bar\Pi\quad(t\in\R)
  \end{equation}
  and the limits
  \begin{equation}\label{eq:29a}
    \underset{x\to-\infty}{\lim}U(x,t)=\Theta_-(t),\quad
    \underset{x\to\infty}{\lim}U(x,t)=\Theta_+(t)
  \end{equation}
  exist for all $t\in\R$.
\end{description}

Our main result in this subsection is following proposition concerning
the case when $\Sii$ is a nontrivial chain.
\begin{proposition}
  \label{prop:entire}
  Under the above hypotheses, assuming also that $\Sii$ is a
  nontrivial chain, the following alternative holds.  Either $U$ is
  identical to a steady state $\phi$ with
  $\tau(\phi) \subset\Sigma_{in}\cup\La_{out}$ or else
  \begin{equation}
    \label{eq:5a}
    \underset{t\in\R}{{\textstyle \bigcup}}\tau\left( U(\cdot,t)\right)\subset \Pi
  \end{equation}
  and
  \begin{equation}
    \label{eq:6a}
    \tau\left(\alpha(U)\right)\subset\Sigma_{in},\qquad
    \tau\left(\omega(U)\right)\subset\Lambda_{out}. 
  \end{equation}
\end{proposition}

An interpretation of this result is that any entire solution of
\eqref{eq:1} satisfying \eqref{eq:5b}, \eqref{eq:29a} is either a
steady state or a connection, in $L^\infty_{loc}(\R)$, between the
following two sets of steady states:
\begin{align*}
  E_{in}&:=\{\varphi: \varphi\text{ is solution of \eqref{eq:steady}
          with }\tau(\varphi)\subset \Sigma_{in}\},\\
  E_{out}&:=\{\varphi: \varphi\text{ is solution of \eqref{eq:steady}
           with }\tau(\varphi)\subset \Lambda_{out}\}.
\end{align*}
Moreover, the connection always goes from $E_{in}$ to $E_{out}$ as
time increases from $-\infty$ to $\infty$.  Note that this result, in
conjunction with Proposition \ref{prop:3.3}(i) and Lemma
\ref{le:limits}, implies that statement (ii) of Proposition
\ref{prop:3.3} holds under the extra assumption that
$\Sigma_{in}$ is a nontrivial chain.

In the case when $\Sii$ is a trivial chain, we do not have such a
complete characterization of entire solutions satisfying (HU). We
only prove some partial results in this case, which will be used in Section
\ref{completionprop}.  For that, we will need the following additional
assumption:
\begin{description}
\item[\bf (TC)] (Additional assumption in the case $\Sii=\{(\be,0)\}$
  is a trivial chain). If $U$ is not a steady state, then for all
  $t\in\R$ the
  function $U(\cdot,t)-\be$ has only simple zeros and and the number
  of its critical points is finite and bounded uniformly in $t$.
\end{description}

In the next subsection, we prove several results valid in general,
whether $\Sii$ is trivial or nontrivial, assuming (TC) in the former
case. Then, in Subsection \ref{sec:nontriv}, we examine in more
detail the case when $\Sii$ is nontrivial and prove Proposition
\ref{prop:entire}.

The following notation will be used throughout this section.

Recall that $\Sigma_{in}$ (as any other chain) has the structure as in
\eqref{Sigmain} for some $p\le q$. We define the values $\be_\pm$ as
in \eqref{betaplusmoins}. They are unstable equilibria of
\eqref{eq:ODE}. If $\Sii=\{(\be,0)\}$ is a trivial chain, then
$\be_\pm=p=q=\be$. If $\Sii$ is nontrivial, then $p< q$ and
\eqref{eq:3.6}, \eqref{A1beta} hold.

As for $\Lambda_{out}$, there are two possibilities:
\begin{description}
\item[\bf(A1)] $\Lambda_{out}$ is a \emph{homoclinic loop,} that is,
  it is the union of a homoclinic orbit of \eqref{eq:sys} and its limit
  equilibrium, or, in other words,
  \begin{equation}
    \label{eq:21}
    \Lambda_{out}=\{(\ga,0)\}\  {\textstyle \bigcup }\  \tau(\Phi),
  \end{equation}
  where $f(\ga)=0$ and $\Phi$ is a ground state of \eqref{eq:steady} at
  level $\ga$. We choose $\Phi$ so that $\Phi'(0)=0$, that is, the
  only critical point $\Phi$ is  $x=0$.
\item[\bf(A2)] $\Lambda_{out}$ is a \emph{heteroclinic loop,} that is,
  it is the union of two heteroclinic orbits of \eqref{eq:sys} and their
  limit equilibria $(\gamma_\pm,0)$. In other words,
  \begin{equation}
    \label{eq:22}
    \Lambda_{out}=\{(\ga_-,0), (\ga_+,0)\}\ {\textstyle \bigcup}\ \tau(\Phi^+)\ {\textstyle \bigcup}\ \tau(\Phi^-),
  \end{equation}
  with $\gamma_-<\gamma_+$, $f(\ga_\pm)=0$, and $\Phi^\pm$ are
  standing waves of \eqref{eq:steady} connecting $\gamma_-$ and
  $\gamma_+$, one increasing the other one decreasing. We adopt the
  convention that $\Phi^+_x>0$ and $\Phi^-_x<0$.
\end{description}
To have a unified notation, we set
\begin{equation}
  \label{eq:19}
  \begin{aligned}
    \hat p&:=\inf\{a\in\R: (a,0)\in \Pi\}=\inf\{a\in\R: (a,0)\in \La_{out}\},\\
    \hat q&:=\sup\{a\in\R: (a,0)\in \Pi\}=\sup\{a\in\R: (a,0)\in
    \La_{out}\}.
  \end{aligned}
\end{equation}
Thus, $\{\hat p,\hat q\}=\{\ga, \Phi(0)\}$ if (A1) holds; and
$\hat p=\ga_-$, $\hat q =\ga_+$ if (A2) holds.

Also remember that if $(\bar \ga,0)$ is any equilibrium of \eqref{eq:sys}
contained in $\Lambda_{out}$ or in $\Sigma_{in}$ when $\Sii$ is a
nontrivial chain, then $f'(\bar \ga)<0$ (cp. Section \ref{sec:22}).
This in particular applies to $\ga$, $\ga_\pm$ in (A1), (A2).

\subsection{Some general results}
\label{sec:gen}
We assume the standing hypothesis for this section, as spelled out in
the paragraph containing (HU). In case $\Sii=\{(\be,0)\}$, we also
assume the extra hypothesis (TC).

We start by recalling the following consequence of hypothesis (HU)
(cp. Remark \ref{rm:limst}).
\begin{corollary}
  \label{co:rment} The following statements hold:
  \begin{itemize}
  \item[(i)] If $\Sii$ is a nontrivial chain,
    $\Theta_\pm(t)=:\Theta_\pm$ are independent of $t$ and
    $(\Theta_\pm,0)\in \Sigma_{in}\cup\La_{out}$.
  
  \item[(ii)] If $\Sigma_{in}=\{(\be,0)\}$ is a trivial chain and
    $\Theta(t)$ stands for $\Theta_+(t)$ or $\Theta_+(t)$, then
    either $\Theta(t)=:\Theta$ is independent of $t$ and
    $(\Theta,0)\in \{(\be,0)\}\cup\La_{out}$, or $\Theta(t)$ is a
    strictly monotone solution of \eqref{eq:ODE} with
    $\Theta(-\infty)=\be$ and $(\Theta(\infty),0)\in\La_{out}$.
  \end{itemize}
\end{corollary}

Next we prove  the following basic dichotomy.
\begin{lemma}
  \label{le:dichotomy}
  Either $U$ is identical to a steady state $\phi$ with
  $\tau(\phi) \subset\Sigma_{in}\cup\La_{out}$, or else $U$ is not a
  steady state and \eqref{eq:5a} holds.
\end{lemma}
\begin{proof}
  The existence of the limits \eqref{eq:29a} implies that $U$ cannot
  be a nonconstant periodic steady state. Thus if \eqref{eq:5a} holds,
  $U$ cannot be any steady state.
  
  Assume now that \eqref{eq:5a} does not hold. Then there exist
  $x_0,t_0\in \R$ and a steady state $\phi$ with
  $\tau(\phi) \subset\Sigma_{in}\cup\La_{out}$ such that
  $U(\cdot,t_0)-\phi$ has a multiple zero at $x_0$. By connectedness
  of $\tau(\phi)$, $\tau(\phi)\subset \Sigma_{in}$ or
  $\tau(\phi)\subset \La_{out}$.  For definiteness, we assume the
  former; the arguments in the latter case are analogous (and one
  does not need to deal with trivial chain in that case).

  We want to show that $U\equiv \phi$.  If $\Sii=\{(\be,0)\}$ is a
  trivial chain (hence $\phi\equiv\be$), this follows immediately from
  (TC), specifically from the assumption that $U(\cdot,t)-\be$ has
  only simple zeros. Assume now that $\Sii$ is a nontrivial chain.  If
  $U\not\equiv \phi$, then $U-\phi$ is a nontrivial solution of a
  linear equation \eqref{eqlin}. Using Lemma \ref{le:p0} (and the fact
  that there are only finitely many chains), we find a sequence
  $\psi_n$ of periodic solutions of \eqref{eq:steady} such that
  $\tau(\psi_n)\subset \cI(\Sigma_{in})$ and $\psi_n\to\phi$ in
  $C^1_{loc}(\R)$. Applying Lemma \ref{robustnesszero}, we find a
  sequence $t_n\to t_0$ such that $U(\cdot,t_n)-\psi_n$ has a multiple
  zero. Consequently,
  $\tau(U(\cdot,t_n))\cap\tau(\psi_n)\ne \emptyset$, in contradiction
  to \eqref{eq:5b}. This contradiction shows that, indeed,
  $U\equiv \phi$.
\end{proof}

Clearly, the inclusion \eqref{eq:5a} implies that
\begin{equation}\label{eq:27a}
  U_x(x,t)\neq0\text{ whenever }U(x,t)\in [p,q].
\end{equation}

The following lemma shows in particular that if $U$ is not a steady
state and $\Sii$ is a nontrivial chain, then the number of critical
points of $U(\cdot,t)$ is bounded uniformly in $t$. If $\Sii$ is a
trivial chain, we have this by assumption, see (TC).
\begin{lemma}
  \label{le:zerob} Assume that $\Sii$ is a nontrivial chain.  If $U$
  is not a steady state, then the following statements are valid:
  \begin{itemize}
  \item[(i)] There are $N^+,N^-<\infty$ such that
    \begin{equation}\label{eq:28}
      z\left( U(\cdot,t)-\beta_\pm\right) = N^\pm\quad(t\in\R).
    \end{equation} 
  \item[(ii)] Let $\be=\be_-$ or $\be=\be_+$, and $t_0\in \R$.  Let
    $I:=(\zeta_1,\zeta_2)$, with $-\infty\le \zeta_1<\zeta_2\le\infty$,
    be any nodal interval of $U(\cdot,t_0)-\be$ (that is,
    $U(\cdot,t_0)-\be\ne 0$ in $I$ and $U(\cdot,t_0)-\be= 0$ on
    $\partial I$). Then $U(\cdot,t_0)$ has at most one critical point
    in $I$ and if such a critical point exists, it is nondegenerate.
  \end{itemize}
\end{lemma}
\begin{remark}\label{rm:critp}
  {\rm With $\beta$ and $I=(\zeta_1,\zeta_2)$ as in statement (ii),
    the number of critical points of $U(\cdot,t_0)$ in $I$ can be
    specified by elementary considerations.  For example, $U(\cdot,t_0)$ has
    exactly one critical point in $I$ if $\zeta_1$, $\zeta_2$ are
    both finite (and hence are two successive zeros of
    $U(\cdot,t_0)-\beta $). If $\zeta_1\in \R$, $\zeta_2=\infty$,
    $U_x(\zeta_1,t_0)>0$, then either $U(\cdot,t_0)$ has exactly one
    critical point in $I$ and in this case
    $\Theta_+=U(\infty,t_0)< \beta_+$ or else  $U_x(\cdot,t_0)>0$ in
    $I$. The discussion in the
    other cases is similar.
  }
\end{remark}

\begin{proof}[Proof of Lemma \ref{le:zerob}]
  By Lemma \ref{le:dichotomy}, \eqref{eq:5a} holds, and by Corollary
  \ref{co:rment},
  $(\Theta_\pm,0)$ are independent of $t$ and contained in
  $ \Sigma_{in}\cup\La_{out}$. These inclusions and \eqref{A1beta}
  imply that
  $\{\Theta_-,\Theta_+\}\cap
  \{\beta_-,\beta_+\}=\emptyset$. Therefore, the zero numbers
  $z\left( U(\cdot,t)-\beta_\pm\right)$ are finite for all $t$, and
  are nonincreasing in $t$. We show that
  $z\left( U(\cdot,t)-\beta_+\right)$ does not drop at any $t_0\in \R$
  (the proof for $\be_-$ is similar). By \eqref{eq:27a}, all zeros of
  $U(\cdot,t)-\beta_\pm$ are simple, hence locally they are given by
  $C^1$ functions of $t$.  The only way
  $z\left( U(\cdot,t)-\beta_+\right)$ can drop at $t_0$ is that one of
  these $C^1$ functions, say $\xi(t)$, is unbounded as $t\to 0$. To
  rule this possibility out, we show that $|\xi'(t)|$ is uniformly
  bounded. Indeed, from \eqref{A1beta} and the fact that
  $\tau(U(\cdot,t))\subset\Pi$ we infer that $|U_x(x,t)|$ is bounded
  from below by a fixed positive constant (independent of $x$ and $t$)
  whenever $U(x,t)=\be_+$. Since, by parabolic estimates, $|U_t|$ is
  uniformly bounded, a bound on $\xi'(t)$ is found immediately upon
  differentiating the identity $U(\xi(t),t)=\be_+$. This completes the
  proof of statement (i).

  In the proof of statement (ii), we only consider the case of a
  bounded nodal interval $I=(\zeta_1,\zeta_2)$, the other cases can be
  treated similarly.  Also, we  assume for definiteness that
  $U(\cdot,t_0)-\beta> 0$ in $I$, the case $U(\cdot,t_0)-\be < 0$ in
  $I$ being analogous. Suppose for a contradiction that $U(\cdot,t_0)$
  has more than one critical point in $(\zeta_1,\zeta_2)$ or has a
  degenerate critical point there.  From statement (i) and Remark
  \ref{convzero} we infer that the function $U(\cdot,t)-\be$ has a
  finite number (independent of $t$) of zeros, all of them
  simple. Using this and the implicit function theorem, we obtain the
  following.  There are $C^1$ functions $\bar \zeta_i(t)$ defined for
  all $t\in\R$ such that $\bar \zeta_i(t_0)=\zeta_i$, $i=1,2$, and,
  for any $t$, $(\bar \zeta_1(t),\bar \zeta_2(t))$ is a nodal interval
  of $U(\cdot,t)-\be$: $U(\cdot,t)>\be$ in
  $(\bar \zeta_1(t),\bar \zeta_2(t))$, $U(\bar\zeta_i(t),t)=\be$,
  $i=1,2$.  Considering the zero number of $U_x(\cdot,t)$ in
  $(\bar \zeta_1(t),\bar \zeta_2(t))$ (remembering that
  $U_x(\bar\zeta_i(t),t)\ne 0$, due to the simplicity of the zeros),
  we infer from Lemma \ref{le:zerot} that for all $t<t_0$ the function
  $U(\cdot,t)$ has at least two critical points in
  $(\bar \zeta_1(t),\bar \zeta_2(t))$. Moreover, for $t<t_0$,
  $t\approx t_0$ the critical points are all nondegenerate. Pick one
  of such $t$, say $t_1$.  Due to \eqref{eq:27a}, the value of
  $U(\cdot,t_1)$ at the critical points is greater than $q$, which is
  greater than $\be_+$. Therefore, there is $\xi_1>q$ such that the
  function $U(\cdot,t_1)-\xi_1$ has at least three zeros.  Let
  $\xi(t)$ denote the solution of $ \dot{\xi}(t)=f(\xi(t))$ with
  $\xi(t_1)=\xi_1$.  Then $\xi(t)>\beta_+$ for all $t$ and
  $\xi(-\infty)=\beta_+.$ Consider the function
  $V(x,t)=U(x,t)-\xi(t).$ It solves a linear equation \eqref{eqlin}
  and satisfies $V(\bar \zeta_i(t),t)<0$ for all $t<t_0$. Therefore, by
  Lemma \ref{le:zerot}, $V(\cdot,t)$ admits at least 3 zeros in
  $(\bar \zeta_1(t),\bar \zeta_2(t))$. Take now a large enough
  negative $t$ so that $\xi(t)\in (\be_-,q)$. Using the fact that
  $U(\cdot,t)-\xi(t)$ has at least 3 zeros in
  $(\bar \zeta_1(t),\bar \zeta_2(t))$ while $U(\cdot,t)>\be_+$ in
  $(\bar \zeta_1(t),\bar \zeta_2(t))$, we find a critical point at
  which $U(\cdot,t)$ takes a value in $(\be_-,q)$, which clearly
  contradict \eqref{eq:27a}. This contradiction proves the conclusion
  of statement (ii).
\end{proof}
\begin{corollary}\label{co:tilde}
  If $\varphi\in A(U)\cup \Om(U)$ and $\tilde U$ is the entire
  solution of \eqref{eq:1} with $\tilde U(\cdot,0)=\varphi$ (and
  $\tilde U(\cdot,t)\in A(u)\cup \Om(U)$ for all $t$), then condition
  {\rm (HU)} holds with $U$ replaced by $\tilde U$. In particular, $\varphi$
  is not identical to any nonconstant periodic steady state.
\end{corollary}
\begin{proof}
  The inclusion $\tau(\tilde U(\cdot,t))\subset \bar\Pi$ for all $t$
  follows from \eqref{eq:5b} and the fact that in the definition of
  $A(U)$, $\Om(U)$ one can take the convergence in $C^1_{loc}(\R)$.  We
  next show that the limits $\tilde U(\pm\infty,t)$ exist for every
  $t\in \R$. A sufficient condition for this is that the zero number
  of $\tilde U_x(\cdot,t)$ is finite for all $t$. This is verified
  easily using the fact---assumed in (TC) or proved in Lemma
  \ref{le:zerob}, depending on whether $\Sii$ is trivial or not---that
  $z(U_x(\cdot,t))$ is finite and bounded from above by some constant
  $k$ independent of $t$. Indeed, if $z(\tilde U_x(\cdot,t_0))=\infty$
  for some $t_0$, then we can find $t<t_0$ such that
  $\tilde U_x(\cdot,t)$ has at least $k+1$ \emph{simple} zeros. Since
  $\tilde U(\cdot,t)\in A(u)\cup \Om(U)$, we obtain by approximation
  that $U_x(\cdot,t_1))$ has $k+1$ zeros for some $t_1$, which is
  impossible.
\end{proof}

In the following lemma we establish a basic relation of $U$ to
$\Sigma_{in}$, $\Lambda_{out}$.

\begin{lemma}
  \label{le:basicrel}
  Assume $U$ is not a steady state and let $K$ be any one of the sets
  $\Sigma_{in}$, $\Lambda_{out}$.  Then the following statements are
  valid.
  \begin{itemize}
  \item[(i)] If $(x_n,t_n)$, $n=1,2,\dots,$ is a sequence in $\R^2$
    such that
    \begin{equation}
      \label{eq:58}
      \dist((U(x_n,t_n),U_x(x_n,t_n)),K)\to 0,
    \end{equation}
    then,
    possibly after passing to a subsequence, one has
    $U(\codt+x_n,\cdot+t_n)\to\varphi$ in $C^1_{loc}(\R^2)$, where
    $\varphi$ is a steady state of \eqref{eq:1} with
    $\tau(\varphi)\subset K$.
  \item[(ii)] There exists a sequence $(x_n,t_n)$, $n=1,2,\dots$ as in
    (i) with the additional property that $|t_n|\to\infty.$
    Consequently, there exists a steady state of \eqref{eq:1} with
    $\tau(\varphi)\subset K$ and
    \begin{equation}
      \label{eq:12}
      \varphi\in A(U)\cup \Om(U).
    \end{equation} 
  \end{itemize}
\end{lemma}
(Recall that $A(U)$ and $\Om(U)$ are the generalized limit sets of
$U$, as defined in Section \ref{sec:inv}.)

\begin{proof}[Proof of Lemma \ref{le:basicrel}]
  With the sequence $(x_n,t_n)$ as in (i), we may assume, passing to a
  subsequence if necessary, that
  \begin{equation*}
    \left(
      U(x_n,t_n),U_x(x_n,t_n)\right)\underset{n\to\infty}{\longrightarrow}(a,b)\in
    K.
  \end{equation*}
  Let $\varphi$ be the solution of \eqref{eq:steady} with
  $(\varphi(0),\varphi'(0))=(a,b)$, so $\tau(\varphi)\subset K$.
  Consider the sequence of functions $U_n:=U(\cdot+x_n,\cdot+t_n).$ Up
  to a subsequence, it converges in $C^1_{loc}(\R^2)$ to $\tilde{U}$,
  an entire solution of \eqref{eq:1}.  Clearly,
  $\left(\tilde{U}(0,0),\tilde{U}_x(0,0)\right)=(a,b)$, so
  $\tilde U(\cdot,0)-\varphi$ has a multiple zero at $x=0$. Now,
  unless $\tilde U\equiv \varphi$, Lemma \ref{robustnesszero} implies
  that if $n$ is large enough, the function $U(\cdot+x_n,t)-\varphi$
  has a multiple zero for some $t\approx t_n$. This would mean that
  $\tau\left( U(\cdot,t)\right)\cap\tau(\varphi)\ne\emptyset$, which
  is impossible by \eqref{eq:5a}. Thus, necessarily,
  $\tilde U\equiv \varphi$ which yields the conclusion of statement
  (i).

  We now prove the existence of a sequence $(x_n,t_n)$ with the above
  property and with $|t_n|\to \infty$. This is is trivial if
  $(\Theta_-,0)$ is independent of $t$ and contained in $K$, for in
  this case we have $(U(x,t),U_x(x,t))\to (\Theta_-,0)$ as
  $x\to-\infty$ for every $t$. Similarly, the statement is trivial if
  $(\Theta_+,0)\in K$. If $\Theta_-(t)$ is not constant (which may
  happen only if $\Sii$ is a trivial chain, cp. Lemma \ref{co:rment}),
  then again the statement is trivial and follows from the facts
  that $(U(x,t),U_x(x,t))\to (\Theta_-(t),0)$ as $x\to-\infty$ and
  either $(\Theta_-(\infty),0)\in K$ or $(\Theta_-(-\infty),0)\in K$
  (cp.  Lemma \ref{co:rment}).  A similar argument applies
  if $\Theta_+(t)$ is not
  constant.  It remains to consider the case when $(\Theta_\pm,0)$ are
  both independent of $t$ and contained in $ K^*$, where where
  $K^*\in \{\Sigma_{in}, \Lambda_{out}\}$, $K^*\ne K$. First we show
  the existence of a sequence satisfying \eqref{eq:58}.
  Suppose that no such
  sequence  exists.  Then there is $\e>0$
  such that
  \begin{equation}
    \label{eq:16}
    \dist\left( \tau\left( U(\cdot,t)\right),K\right)>\e\quad (t\in\R).
  \end{equation}
  This implies that there is a periodic orbit $\cO$, taken
  sufficiently close to $K$ (cp. Lemma \ref{le:p0}) such that
  \begin{equation*}
    \underset{t\in\R}{{\textstyle \bigcup}}\tau\left( U(\cdot,t)\right)\subset
    \mathcal{I}(\mathcal{O})\ \text{ or }\ 
    \underset{t\in\R}{{\textstyle \bigcup}}\tau\left( U(\cdot,t)\right) \subset\R^2\setminus{\overline{\mathcal{I}}(\mathcal{O})}).
  \end{equation*}
  In either case, Lemma \ref{squeezinglemma} shows that
  \begin{equation}
    \label{eq:18}
    \underset{t\in\R}{{\textstyle \bigcup}}\tau\left( U(\cdot,t)\right)\cap \Pi=\emptyset,
  \end{equation}
  in contradiction to \eqref{eq:5a} (cp. Lemma \ref{le:dichotomy}).
  Thus there is a sequence 
  satisfying \eqref{eq:58}. We claim that $|t_n|\to\infty$. Indeed, if
  not, then for a subsequence we have $t_n\to t_0\in\R$. Since $U_t$
  is bounded, we have
  $U(\cdot,t_n)\to U(\cdot,t_0)$ uniformly on $\R$. Consequently, by
  parabolic regularity, also  $U_x(\cdot,t_n)\to U_x(\cdot,t_0)$
  uniformly on $\R$. Therefore,
  $(U(x,t_n),U_x(x,t_n))\approx (\Theta_\pm,0)\in K^*$ if $n$ and $\pm
  x$ are sufficiently large. This implies, in view of \eqref{eq:58},
  that the sequence
  $(x_n)$ is bounded and so, passing to a subsequence, we have $x_n\to
  x_0$.  Using \eqref{eq:58} and the convergence $(x_n,t_n)\to
  (x_0,t_0)$, we obtain $(U(x_0,t_0),U_x(x_0,t_0))\in K$, which is a
  contradiction to  \eqref{eq:5a}.  This contradiction proves
  our claim and completes the proof of  the first part of statement
  (ii). The last conclusion in (ii)
  follows immediately from statement (i) and the
  definition of the limit sets $A(U)$, $\Om(U)$.
\end{proof}

We next show that \eqref{eq:6a} holds if one of the
zero numbers $z(U(\cdot,t)-\be_\pm)$ vanishes, that is, $U<\be_+$ or
$U>\be_-$.  The following lemma is a stronger result, which partly also
applies when $\Sii=\{(\be,0)\}$ is a trivial chain (in which case
$\be_\pm=\be$). This lemma will be used at several other occasions
below.
\begin{lemma} The following statements are valid (recall that $\hat
  p$, $\hat q$ are defined in \eqref{eq:19}).
  \label{le:ep}
  \begin{itemize}
  \item[(i)] If $U\le \hat q-\vartheta$ for some $\vartheta>0$ and $U$
    is not a steady state, then $\om(U)=\{\hat p\}$ (so, necessarily,
    $f(\hat p)=0$) and $\tau(\al(U))\subset\Sigma_{in}$. Similarly, if
    $U\ge\hat p+\vartheta$ for some $\vartheta>0$ and $U$ is not a
    steady state, then $\om(U)=\{\hat q\}$ (so $f(\hat q)=0$) and
    $\tau(\al(U))\subset\Sigma_{in}$.
  \item[(ii)] If for some $t_0\in \R$ and $\vartheta>0$ one has
    $U(\cdot,t)\le \hat q-\vartheta$ for all $t<t_0$, then either
    $U\equiv \hat p$ or else $\tau(\al(U))\subset\Sigma_{in}$.  If for
    some $t_0\in \R$ and $\vartheta>0$ one has
    $U(\cdot,t)\ge \hat p+\vartheta$ for all $t<t_0$, then either
    $U\equiv \hat q$ or else $\tau(\al(U))\subset\Sigma_{in}$.
  \item[(iii)] Assume that $\Sii$ is a nontrivial chain. If
    $U\le\be_-$, then $U\equiv \hat p$; and if $U\ge \be_+$, then
    $U\equiv \hat q$.
  \end{itemize}
\end{lemma}
\begin{proof}
  We only prove the first statements in (i) and (ii); the proofs of
  the other statements in (i) and (ii) are analogous and are omitted.

  The statements in (iii) follow from (i) and the fact that if
  $\Sigma_{in}$ is a nontrivial chain, then there are no functions
  $\varphi$ with $\tau(\varphi)\subset \Sigma_{in}$ satisfying
  $\varphi\le \be_-$ or $\varphi\ge \be_+$.
 
  To prove (i), assume that $U$ is not a steady state and
  $U\le\hat q-\vartheta$ for some $\vartheta>0$. By Lemma \ref{le:dichotomy},
  \eqref{eq:5a}
  holds and, in particular, $U>\hat p$.  By Lemma \ref{le:basicrel},
  the set $A(U)\cup\Om(U)$ contains a steady state $\varphi$ with
  $\tau(\varphi)\subset \La_{out}$. Obviously,
  $\varphi\le \hat q-\vartheta$, hence, necessarily, $\varphi=\hat p$
  (and $f(\hat p)=0$).  Let $\psi$ be any periodic solution of
  \eqref{eq:steady} with $\psi'(0)=0$,
  $\psi(0)\in ( \hat q-\vartheta,\hat q)$, and let $\rho>0$ be the
  minimal period of $\psi$. Clearly, $\tau(\psi)\subset \Pi$, in
  particular, $\min\psi=\psi(\rho/2)>\hat p$.  From
  $\hat p\in A(U)\cup\Om(U)$ we infer that there exist $\xi,t_1\in\R$
  such that
  \begin{equation}
    \label{eq:24}
    \text{$U(\cdot,t_1)<\psi$ in
      $[-\rho+\xi,\rho+\xi]$.}
  \end{equation}
  Consequently, $U(\cdot,t_1)<\psi$ in $[k\rho,(k+1)\rho]$, where
  $k:=[\xi/\rho]$ is the integer part of $\xi/\rho$. This and the
  relations
  $$\psi(k\rho)=\psi((k+1)\rho)=\psi(0)>\hat q-\vartheta\ge U$$
  yield, upon an application of the comparison principle, that
  $U(\cdot,t)<\psi$ in $[k\rho,(k+1)\rho]$ for all $t>t_1$. Hence, for
  each $\varphi\in\om(U)$ we have $\varphi\le \psi$ in
  $[k\rho,(k+1)\rho]$. We claim that, in fact,
  \begin{equation}
    \label{eq:23}
    \varphi\le \min \psi \quad (\varphi\in\om(U)).
  \end{equation}

  Indeed, consider the set $M$ of all $\bar\eta\in\R$ such that
  \begin{equation*}
    \varphi\le \psi(\cdot-\eta) \text{ in  $[k\rho+\eta,(k+1)\rho+\eta]$
      for  all $\varphi\in\om(U)$ and all $\eta$ between $0$ and $\bar\eta$.}
  \end{equation*}
  We have shown above that $0\in M$.  Suppose for a contradiction that
  $\eta_-:=\inf M>-\infty$. Then, clearly, $\eta_-\in M$, and using
  the compactness of $\om(U)$ in $L_{loc}^\infty(\R)$ one shows easily
  that for some $\varphi\in \om(U)$ the inequality
  \begin{equation}
    \label{eq:20}
    \text{$\varphi\le \psi(\cdot-\eta_-)$ in
      $[k\rho+\eta_-,(k+1)\rho+\eta_-]$}
  \end{equation}
  is not strict. Let $\tilde U$ be the entire solution of \eqref{eq:1}
  with $\tilde U(\codt,0)=\varphi$ and $\tilde U(\codt,t)\in \om(U)$
  for all $t\in\R$. The inclusion $\eta_-\in M$ implies that
  $\tilde U\le \psi(\cdot-\eta_-)$ in
  $[k\rho+\eta_-,(k+1)\rho+\eta_-]$ for all $t$ and from the
  assumption on $U$ it follows that
  $$\tilde U\le  \hat q-\vartheta<\psi(0)= \psi(k\rho)=\psi((k+1)\rho).$$
  Therefore, by the strong comparison principle, the inequality in
  \eqref{eq:20} is in fact strict and we have a desired contradiction.
  We have thus proved that $\inf M=-\infty$.  Similar arguments show
  that $\sup M=\infty$, hence $M=\R$.  Now, given any
  $\varphi\in \om(U)$ and $x_0\in \R$, take
  $\eta:=x_0-k\rho-\rho/2$. Then $x_0\in [k\rho+\eta,(k+1)\rho+\eta]$
  and the fact that $\eta\in M$ yields
  \begin{equation*}
    \varphi(x_0)\le
    \psi(x_0-\eta)=\psi(k\rho+\rho/2)=\psi(\rho/2)=\min \psi. 
  \end{equation*}
  This proves \eqref{eq:23}.

  Clearly, taking the periodic solution $\psi$ with the maximum
  $\psi(0)$ sufficiently close to $\hat q$, we can make
  $\min \psi-\hat p$ as small as we like. Therefore, \eqref{eq:23}
  implies that $\om(U)=\{\hat p\}$, as stated in Lemma \ref{le:ep}.

  Next we show that $\hat p\not\in \al(U)$. We actually prove that
  $\hat p\in \al(U)$ implies that $U\equiv \hat p$ (which, of course,
  is a contradiction with the fact that $U$ is not a steady state). Note that in this
  argument we only use that the estimate
  $U(\cdot,t)\le \hat q-\vartheta$ holds for all sufficiently large
  negative $t$, say for all $t<t_0$, so the argument can be repeated
  in the proof of statement (ii) below. Assume that
  $\hat p\in \al(U)$.  As in the previous paragraphs, taking any
  periodic solution $\psi$ with $\psi'(0)=0$ and
  $\psi(0)\in ( \hat q-\vartheta,\hat q)$, we again obtain
  \eqref{eq:24}, but this time we can take $\xi=0$ and we can choose
  $t_1<0$ arbitrarily large. The comparison principle then implies in
  particular that $U(\cdot,t_0)<\psi$ in $[-\rho,\rho]$. Taking a
  sequence of periodic solutions $\psi$ with $\psi(0)\upto \hat q$, we
  obtain $U(\cdot,t_0)\equiv \hp$. Consequently, $\hp$ is a steady
  state and $U\equiv \hp$, as claimed.

  Take now an arbitrary $\varphi\in \al(U)$ and let $\tilde U$ be the
  entire solution of \eqref{eq:1} with $\tilde U(\codt,0)=\varphi$ and
  $\tilde U(\codt,t)\in \al(U)$ for all $t\in\R$. Obviously,
  $\tilde U$ inherits the relation $\tilde U\le \hat q-\vartheta$ from $U$.
  If $\tilde U$ is not a steady state, then, by Corollary
  \ref{co:tilde}, what we have already proved above in this proof
  applies equally well to $\tilde U$:
  $\hat p\in \om(\tilde U)\subset \al(U)$ (the latter relation is by
  compactness of $\al(U)$ in $L^\infty_{loc}(\R)$). This is impossible
  as we have just proved, so $\varphi$ has to be a steady state
  different from $\hat p$. Moreover, $\varphi$ cannot be periodic
  (cp. Corollary \ref{co:tilde}), and therefore the relation
  $\varphi \le \hat q-\vartheta$ implies $\tau(\varphi)\subset
  \Sigma_{in}$. This shows that $\tau(\al(U))\subset \Sigma_{in}$,
  completing the proof (i).

  We now prove statement (ii). As already noted above, under the assumption of
  (ii), $\hat p\in \al(U)$ implies that $U\equiv \hat p$. If
  $\hat p\not\in \al(U)$, we can repeat the previous paragraph to show
  that $\tau(\al(U))\subset \Sigma_{in}$.
\end{proof}

\subsection{Nontrivial  inner chain}
\label{sec:nontriv}
In this subsection, we assume that $\Sii$ is a nontrivial chain (and
continue to assume the standing hypotheses formulated in the paragraph
containing (HU)).  By Corollary \ref{co:rment},
the limits
$\Theta_\pm =U(\pm \infty,t)$ are independent of $t$ and contained in
  $\Sigma_{in}\cup\La_{out}$.

We distinguish the following cases of how $(\Theta_\pm,0)$ can be
included in $\Sigma_{in}\cup\Lambda_{out}$:
\begin{itemize}
\item[\bf(C1)] $(\Theta_\pm,0)\in \Lambda_{out}$
\item[\bf(C2)] $(\Theta_\pm,0)\in \Sigma_{in}$
\item[\bf(C3)] $(\Theta_-,0)\in \Sigma_{in}$ and
  $(\Theta_+,0)\in \Lambda_{out}$; or $(\Theta_+,0)\in \Sigma_{in}$
  and $(\Theta_-,0)\in \Lambda_{out}$.
\end{itemize}
We tackle these cases separately in the following subsections; in each
of them, we prove that the conclusion of Proposition
\ref{prop:entire} holds. Some of the forthcoming results actually give a
more specific description the $\al$ and $\om$-limit sets than the
general description given in \eqref{eq:6a}.

\subsubsection{Case (C1): $(\Theta_\pm,0)\in \Lambda_{out}$}\label{subsc1}
We first show that under condition {\rm (C1)}, $U$ converges in
$L^\infty(\R)$---not just in $L_{loc}^\infty(\R)$---to a steady state
$\phi$ with $\tau(\phi)\subset \Lambda_{out}$. In particular,
$\tau(\om(U))\subset \Lambda_{out}$.
\begin{Lemma}
  \label{le:conv} Assume {\rm (C1)}.  Then the limit
  $\phi:=\lim_{t\to\infty}U(\cdot,t)$ in $L^\infty(\R)$ exists and is
  more specifically described as follows.
  \begin{itemize}
  \item[(i)] If $\Theta_-\ne \Theta_+$ (so $\Lambda_{out}$ is a
    heteroclinic loop as in \emph{(A2)}), then $\phi$ is a standing wave -- a
    shift of $\Phi^+$ or  $\Phi^-$.
  \item[(ii)] If $\Theta_-=\Theta_+$ and $\Lambda_{out}$ is a
    heteroclinic loop as in \emph{(A2)}, then $\phi$ is identical to one of
    the constants $\ga_-$, $\ga_+$.
  \item[(iii)] If $\Theta_-=\Theta_+=\ga$ and $\Lambda_{out}$ is a
    homoclinic loop as in \emph{(A1)}, then $\phi$ is identical to the
    constant $\ga$ or to a shift of the ground state $\Phi$.
  \end{itemize}
  In all these cases, $\tau(\phi)\subset \Lambda_{out}$.
\end{Lemma}
\begin{proof}
  If $\Theta_-\ne \Theta_+$, then
  $\{\Theta_-, \Theta_+\}= \{\gamma_-, \gamma_+\}$ (cp. (A2)).
  Clearly, $U$ is a front-like solution in the sense that $U$ takes
  values between its limits $\ga_-$, $\ga_+$, at $x=-\infty$,
  $x=\infty$. Since $f'(\ga_\pm)<0$, statement (i) becomes a special
  case of a well-known convergence result \cite[Theorem 3.1]{FMcL}.

  Under the assumptions of statement (ii), $\Theta_-= \Theta_+$ is
  equal to one of the constants $\ga_-$, $\ga_+$ and
  $\ga_-\le U\le \ga_+$. In this situation, the convergence stated in
  (ii) is also well-known and can be easily derived from \cite[Theorem
  3.1]{FMcL}, see for example \cite[Proof of Lemma 3.4]{P:examples}.

  Assume now that $\Theta_-=\Theta_+=\ga$ and $\Lambda_{out}$ is a
  homoclinic loop as in (A1). Clearly, $\ga\le U\le \tilde q=\Phi(0)$
  and, since $(\ga,\tilde q]$ is the range of the ground state $\Phi$,
  $F<F(\ga)$ in $(\ga,\tilde q]$. Since also $f'(\ga)<0$, we are in
  the setup of \cite[Theorem 2.5]{Matano_Polacik_CPDE16} whose
  conclusion, translated to the present notation, is the same as the
  conclusion in (iii).
\end{proof}
The following lemma completes the proof of Proposition
\ref{prop:entire} in the case (C1).
\begin{lemma}
  \label{le:c1alpha}
  Assume {\rm (C1)}. If $U$ is not a steady state, then
  $\tau(\al(U))\subset \Sigma_{in}$.
\end{lemma}
\begin{proof}
  Let $\phi$ be as in Lemma \ref{le:conv}: $U(\codt,t)$ converges to
  $\phi$ uniformly as $t\to\infty$, hence, by parabolic estimates,
  \begin{equation}
    \label{eq:26}
    U(\cdot,t)\underset{t\to\infty}{\longrightarrow}\phi
    \quad\text{in $C^1_b(\R)$.} 
  \end{equation}

  First of all we note that if $\phi$ is identical to one of the
  constants $\hat p$ or $\hat q$ (cp. statements (ii), (iii) in Lemma
  \ref{le:conv}), then $U$ itself is identical to that
  constant. Indeed, we have either $U(\cdot,t)<\be_-$ or
  $U(\cdot,t)>\be_+$ for all large enough $t$ and consequently, by
  Lemma \ref{le:zerob}, for all $t\in \R$. Our statement now follows
  directly from Lemma \ref{le:ep}(ii). Thus, assuming that $U$ is not
  a steady state, we only need to consider the cases (i), (iii) in
  Lemma \ref{le:conv}, and in  the case (iii) we may assume that
  $\phi=\Phi(\cdot-\xi)$ for some $\xi\in\R$.

  \vspace{5pt} \emph{Case (iii) of Lemma \ref{le:conv} with
    $\phi=\Phi(\cdot-\xi)$.} For definiteness, we also assume that
  $\hat q=\Phi(0)$ (and this is the maximum of $\Phi$, cp. (A1)),
  the case $\hat p=\Phi(0)$ being analogous.  It follows from
  \eqref{eq:26} and Lemma \ref{le:zerob}(i) that
  $z\left( U(\cdot,t)-\beta_\pm\right) = 2$ for all $t\in \R$.
  Furthermore, by Lemma \ref{le:zerob}(ii) and Remark \ref{rm:critp},
  $U(\cdot,t)$ has a unique critical point, the global maximum point.

  By Lemma \ref{le:basicrel}, the set $A(U)\cup \Om(U)$ contains a
  steady state $\varphi$ with $\tau(\varphi)\subset \Sii$. The
  possibility $\varphi\in \Om(U)$ is ruled out by uniform convergence
  \eqref{eq:26} to $\phi=\Phi(\cdot-\xi)$, hence $\varphi\in A(U)$.
  Thus, there are sequences $x_n$ and $t_n\to-\infty$ such
  that
  \begin{equation}
    \label{eq:31}
    \text{$U(\cdot+x_n,t_n)\to \varphi$ in $C^1_{loc}(\R)$.}
  \end{equation}
  We use this in the following conclusion. Fixing any
  periodic solution $\psi$ of \eqref{eq:steady} with
  $\tau(\psi)\subset\Pi,$ the inclusion $\tau(\varphi)\subset \Sii$
  implies that $\varphi-\psi$ has infinitely many simple
  zeros. Therefore, \eqref{eq:31} and the monotonicity of the zero
  number imply that
  \begin{equation}
    \label{eq:32}
    \text{$z( U(\cdot,t)-\psi)\to\infty$ as $t\to-\infty$}.
  \end{equation}
  
  We now use \eqref{eq:32} to show that $\Phi\not\in A(U)$ (hence no
  shift of $\Phi$ is contained in $A(U)$, by the shift-invariance of
  $A(U)$).  We go by contradiction. Assume $\Phi\in A(U)$: for some
  sequences $\tilde x_n\in\R$, $\tilde t_n\to-\infty$ we have
  \begin{equation}
    \label{eq:25}
    U(\cdot+\tilde x_n,\tilde t_n)\to \Phi
  \end{equation}
  in $L^\infty_{loc}(\R)$.  Observe that the monotonicity of
  $U(\cdot,\tilde t_n)$ in intervals not containing its unique critical point
  and the relations $U(\pm\infty)=\Theta_\pm=\hp=\Phi(\pm\infty)$
  imply that the convergence in \eqref{eq:25} is actually uniform. It
  then follows from parabolic estimates that the convergence takes
  place in $C^1_b(\R)$.  Consequently,
  \begin{equation}
    \label{eq:30}
    z( U(\cdot,\tilde t_n)-\psi)\to z(\Phi-\psi)=:k,
  \end{equation}
  where $k$ is obviously finite (it is actually equal to 2, as one can
  easily verify).  This contradiction to \eqref{eq:32} proves our
  claim that no shift of $\Phi$ is contained in $A(U)$. This means, by
  Lemma \ref{le:basicrel}, that for some $\vartheta>0$ the maximum of
  $U(\cdot,t)$ stays below $\Phi(0)-\vartheta=\hq-\vartheta$ as
  $t\to-\infty$. An application of Lemma \ref{le:ep}(ii) gives
  $\tau(\al(U))\subset\Sigma_{in}$, which is the desired conclusion.

  \vspace{5pt} \emph{Case (i) of Lemma \ref{le:conv}.}  In this case
  $\phi$ is a standing wave. For definiteness, we assume that
  $\phi=\Phi^+(\cdot-\xi)$ for some $\xi\in\R$, where $\Phi^+$ is the
  increasing standing wave connecting $\ga_-=\hp$ and $\ga_+=\hq$
  (cp. (A2)); the case when $\phi$ is a shift of $\Phi^-$ is
  analogous.  From \eqref{eq:26} and Lemma \ref{le:zerob} we infer
  that $z(U(\codt,t)-\be_\pm)=1$ and $U_x(\cdot,t)>0$ for all $t$.

  We first proceed similarly as in the previous case.  By Lemma
  \ref{le:basicrel}, the set $A(U)\cup \Om(U)$ contains a steady state
  $\varphi$ with $\tau(\varphi)\subset \Sii$, and $\varphi\in \Om(U)$
  is ruled out by uniform convergence \eqref{eq:26} to
  $\Phi^+(\cdot-\xi)$.  Hence $\varphi\in A(U)$. Repeating almost
  verbatim the arguments involving \eqref{eq:32} and \eqref{eq:30}
  (just replace $\Phi$ by $\Phi^+$ and the relations
  $\Theta_\pm=\hp=\Phi(\pm\infty)$ by $\Theta_-=\hp=\Phi^+(-\infty)$,
  $\Theta_+=\hq=\Phi^+(\infty)$), one shows that
  no shift of $\Phi^+$ is contained in $A(U)$.  Obviously, by the
  monotonicity, no shift of $\Phi^-$ can be contained in $A(U)$
  either.

We claim that none of the constants $\ga^+$, $\ga^-$ is contained
in $\al(U)$.  (We remark that
both these constants are contained in $A(U)$,
simply because $U(\pm\infty,t)=\ga^\pm$.)  Suppose, for example, that
$\gamma_+\in \al(U)$ (the possibility $\gamma_-\in \al(U)$ is ruled
out similarly). So there is a sequence $t_n\to-\infty$ such that
$U(\cdot,t_n)\to\gamma_+$ locally uniformly. Pick a small $\e>0$ so
that $\ga^+-\e>\be_+$. Define
\begin{equation}
  \label{eq:57}
  \underline{u}_0(x):=
  \left\{ \begin{array}{cc}
\gamma_-, & \ \textrm{ if } x< 0 \\ \gamma_+-\e, & \ \textrm{ if
}x\geq0
          \end{array} \right.
\end{equation}
and let $\underline{u}(x,t)$
be the solution of \eqref{eq:1} emanating from $\underline{u}_0$ at
$t=0.$ By \cite{FMcL}, there exists $K\in\R$ such that
$\underline{u}(\cdot,t)$ converges uniformly to $\Phi^+(\cdot-K)$ as
$t\to\infty.$ On the other hand, by the assumption on $U$ and due to
$U_x>0$, for every $M\in\R$ there exists $n_M$ such that
$U(\cdot,t_n)>\underline{u}_0(x+M)$ whenever $n>n_M$.
For any such $n$, 
the comparison principle gives $U(x,t_n+t)>\underline{u}(x+M,t)$ for all
$t>0$ and $x\in\R.$ Choosing $t=t'-t_n$ and
taking $n\to\infty$ (so
$t_n\to-\infty$), we obtain that $U(x,t')\geq\Phi^+(x-K+M),$ for all
$x,t'\in\R$.  Taking $M\to\infty$, we obtain $U(\cdot,t')\geq\gamma_+,$
which is a contradiction proving our claim. 

                        Note that a similar comparison argument gives
the following. If there exist $x_0\in \R$ and a sequence
$t_n\to-\infty$ such that for some
$\ga^+-\e>\be_+$ one has $U(x_0,t_n)>\ga^+-\e>\be_+$ for all $n$,
then there is $K\in\R$ such that $U(x,t)\geq\Phi^+(x-K)$ for all
$x,t\in\R$.  We show that this is impossible, thereby showing that
\begin{equation}
  \label{eq:48}
  \limsup_{t\to-\infty}U(x_0,t)\le \be_+\quad(x_0\in
\R).
\end{equation}
Indeed, by Theorem
\ref{thm:GS}, $\alpha(U)$ contains a steady state $\vp_0$ of
\eqref{eq:1}.  It cannot be nonconstant and periodic due to the
monotonicity of $U$.  As shown above, $\vp_0$ cannot be identical to
any of the constants $\ga_\pm$ or any shift of $\Phi^\pm$.  Therefore,
$\tau(\vp_0)\subset\Sigma_{in}.$ This is not compatible with the
relation $\varphi_0\ge \Phi^+(\cdot-K)$, which would obviously follow
from $U\geq\Phi^+(\cdot-K)$.  Thus \eqref{eq:48} is proved and it
implies that  $\varphi\le \be_+$ for all $\varphi\in \al(U)$.  Similarly one
shows that $\varphi\ge \be_-$ for all $\varphi\in \al(U)$.

                        We can now conclude. Given any $\varphi\in
\al(U)$, let $\tilde U$ be the entire solution of \eqref{eq:1} with
$\tilde U(\codt,0)=\varphi$ and $\tilde U(\codt,t)\in \al(U)$ for all
$t\in\R$. Then $\be_-\le \tilde U\le\be_+$. A direct application of
Lemma \ref{le:ep}(i) shows that $\varphi\equiv \tilde U$ is a steady
state. The relations $\be_-\le \varphi\le\be_+$ imply that
$\tau(\varphi)\subset \Sii$. This shows that $\tau(\al(U))\subset
\Sii$, as desired.
                          \end{proof}

\subsubsection{Case (C2): $(\Theta_\pm,0)\in \Sigma_{in}$ }
Throughout this subsection, we assume that (C2) holds.  We will also
assume that the zero numbers $N^\pm$ are both positive:
\begin{equation}
  \label{eq:33}
  N^\pm=z(U(\cdot,t)-\be_\pm)>0.
\end{equation}
This is the only case we still need to worry about, for Lemma
\ref{le:ep}(iii) shows that the conclusion of Proposition
\ref{prop:entire} holds if one of these zero numbers vanishes.

By (C2), $N^\pm$ are even numbers, in fact, if they are nonzero, they are
both equal to 2:
\begin{lemma}
  \label{le:z2}
  Under conditions \eqref{eq:33}, we have
  \begin{equation}\label{eq:z2}
    z\left( U(\cdot,t)-\beta_-\right)=2,\  z\left( U(\cdot,t)-\beta_+\right)=2,\ t\in\R.
  \end{equation}
\end{lemma}

\begin{proof}
  Assume for a contradiction that \eqref{eq:z2} is false. Then, since
  $N^\pm$ are nonzero even numbers, we have (cp. Figure \ref{ruled-out})
  \begin{equation}\label{eq:331-17}
    z(U(\cdot,t)-\be_-)\geq4\ \,(t\in\R)\quad \textrm{ or }\ 
    \,  z(U(\cdot,t)-\beta_+)\geq4\quad (t\in\R).
  \end{equation}

  \begin{figure}[h]
  \vspace{-1.0cm}
 
  \addtolength{\belowcaptionskip}{20pt}
  \addtolength{\abovecaptionskip}{-7.2cm}
  \centering
  \includegraphics[scale=.65]{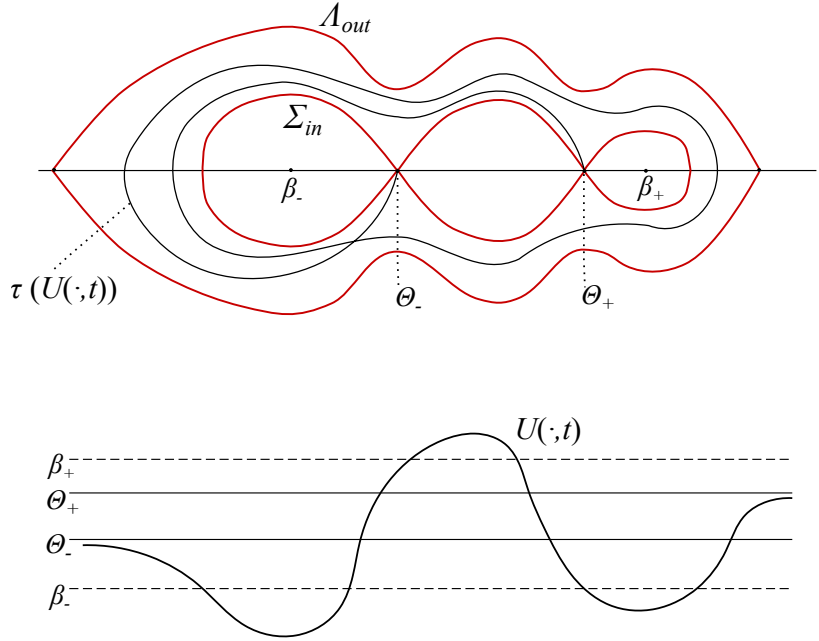}
  \caption[ruled-out case]{An illustration of
    relations \eqref{eq:331-17} that are ruled out
    in the proof of Lemma \ref{le:z2}, in this case
    $z\left( U(\cdot,t)-\beta_-\right)=4$. The top figure depicts the
    spatial trajectory and the bottom figure the graph of
    $U(\cdot,t)$. The relation $\Theta_-<\Theta_+$ chosen for  the
    figure   
    is of no significance; the limits may actually be related the
    other way or may be equal.\label{ruled-out}}  
\end{figure}

  Recall that under condition (C2) the limits $U(\pm\infty,t)=\Theta_\pm$
  are in $(\beta_-,\beta_+)$.  By \eqref{eq:27a}, $U_x(x,t)\ne 0$
  whenever  $U(x,t)\in[\beta_-,\beta_+].$ It follows that,
  assuming \eqref{eq:331-17}, the zero numbers
  $z( U(\cdot,t)-\Theta_\pm)$ are finite and greater than or equal to 2,
  and the zeros being counted are all simple. Moreover, between any two
  successive zeros of $U(\cdot,t)-\Theta_-$ there are (two) zeros of
  either $U(\cdot,t)-\be_-$ or $U(\cdot,t)-\be_+$.  This implies that
  $z( U(\cdot,t)-\Theta_-)$ is bounded uniformly in $t$. The same goes
  for $z( U(\cdot,t)-\Theta_+)$.  Thus, by the monotonicity of the
  zero number, $z( U(\cdot,t)-\Theta_\pm)$
  are constant in $t$ for all sufficiently large negative $t$, say for
  all $t<t_0$. We define
  \begin{align}
    \xi_-(t) & :=\min\left\{ x:U(x,t)=\Theta_-\right\}\text{ \ \ (the first zero of $U-\Theta_-$)}, \nonumber \\
    \xi_+(t) & :=\max\left\{ x:U(x,t)=\Theta_+\right\}\text{ \ (the last zero of $U-\Theta_+$)}. \nonumber 
  \end{align}
  These are well defined and continuous functions of $t$ for $t<t_0$.
  As one checks easily, \eqref{eq:331-17} implies that $\xi_-(t)<\xi_+(t)$.
  Since $U(\pm\infty,t)=\Theta_\pm$, the function $U(\cdot,t)$ is not
  monotone on $(-\infty,\xi_-(t))$, nor it is such on
  $(\xi_+(t),\infty).$ Therefore, by Lemma \ref{lemmaBPQ}, there
  exists $K>0$ such that
  \begin{equation}\label{eq:331-18}
    -K<\xi_-(t)<\xi_+(t)<K\quad (t<t_0).
  \end{equation}
  By \eqref{eq:27a} and \eqref{eq:331-17}, we have
  \begin{equation}\label{eq:331-19}
    z_{(\xi_-(t),\xi_+(t))}\left( U(\cdot,t)-\beta_-\right)\geq2 \
    \, (t<t_0)\quad\textrm{ or }\quad z_{(\xi_-(t),\xi_+(t))}\left(
      U(\cdot,t)-\beta_+\right)\geq2\ \,(t<t_0).
  \end{equation}
  We consider the latter, the former is analogous. By Theorem 
  \ref{thm:GS}, there is a steady state $\phi$ of \eqref{eq:1} with
  $\phi\in\alpha(U)$.  Using \eqref{eq:331-18}, \eqref{eq:331-19} and
  taking into account that between any two successive zeros of
  $U(\cdot,t)-\beta_+$ the function $U(\cdot,t)$ achieves a value
  greater than $q$ or smaller than $p$, we infer that
  $\displaystyle z_{(-K,K)}\left( \phi-\beta_+\right)\geq2.$
  Obviously, $\tau(\phi)\subset \bar\Pi$ and $\phi$ is not a
  nonconstant periodic solution (see Corollary \ref{co:tilde}).
  Moreover, because of \eqref{eq:331-18}, there exist $x_1,x_2$ with
  $-K\leq x_1<x_2\leq K$ such that
  $\phi(x_1),\phi(x_2)\leq\max(\Theta_-,\Theta_+).$ These
  conditions on $\phi$ leave only one possibility for the steady
  state $\phi$: $\phi=\Phi(\cdot-x_0)$ for some $x_0\in(-K,K),$ where
  $\Phi$ is the ground state at level $\ga=\hat p$ as in (A1) (and,
  necessarily, $\La_{out}$ is a homoclinic loop).  We have thus shown
  that for some sequence $t_n\to-\infty,$
  \begin{equation*}
    U(\cdot,t_n)\underset{n\to\infty}{\longrightarrow}\Phi(\cdot-x_0)
  \end{equation*}
  in $L^\infty_{loc}(\R)$.  Notice that from \eqref{eq:331-18} it
  follows that $\Phi(\pm K-x_0)\leq\max(\Theta_+,\Theta_-)<q.$ We show
  that this leads to a contradiction, which will complete the proof.


  Let $\psi$ be any periodic solution of \eqref{eq:steady} with
  $\tau(\psi)\subset\Pi$ and let $\rho>0$ be the minimal period of
  $\psi.$ Shifting $\psi$, we may assume that
  $\psi(K)=\max\psi>q>\Phi(K-x_0).$ Then $\Phi(\cdot-x_0)<\psi$ on
  $(K,K+\rho)$ (otherwise, a shift of the graph of $\psi$ would be
  touching the graph of $\Phi(\cdot-x_0)$, which is impossible for two
  distinct solutions of \eqref{eq:steady}).  Consequently, if $n_0$ is
  large enough, we have $t_{n_0}<t_0-1$ and
  \begin{equation*}
    U(x,t_{n_0})<\psi(x)\quad(x\in(K,K+\rho)).
  \end{equation*}
  Moreover, since $\xi_+(t)<K$ for all $t<t_0,$ we have
  \begin{equation*}
    U(K,t)<\psi(K)\ \textrm{ and } \ U(K+\rho,t)<\psi(K+\rho)\quad
 (t<t_0).  
  \end{equation*}
  Therefore, applying the comparison principle on 
  $(K,K+\rho)\times(t_{n_0},t_0)$, we obtain
  \begin{equation*}
    U(x,t_0-1)<\psi(x)\quad (x\in(K,K+\rho)).
  \end{equation*}
  This is true for all periodic solutions $\psi$ with the indicated
  properties. Taking a sequences of such periodic solution with
  $\psi_j(K)\to\Phi(0)$---which entails $\rho\to\infty$ and
  $\psi_j\to\Phi(\cdot-K)$
  locally uniformly---we obtain that $U(x,0)<\Phi(x-K),$ $x>K.$ So
  $U(x,t_0-1)\to\gamma$ as $x\to\infty,$ in contradiction to (C2).
\end{proof}

\begin{figure}[h]
  \vspace{-1cm}

  \addtolength{\belowcaptionskip}{20pt}
  \addtolength{\abovecaptionskip}{-14cm}
  \centering
  
  \hspace{-1.35cm} 
 \includegraphics[scale=.8]{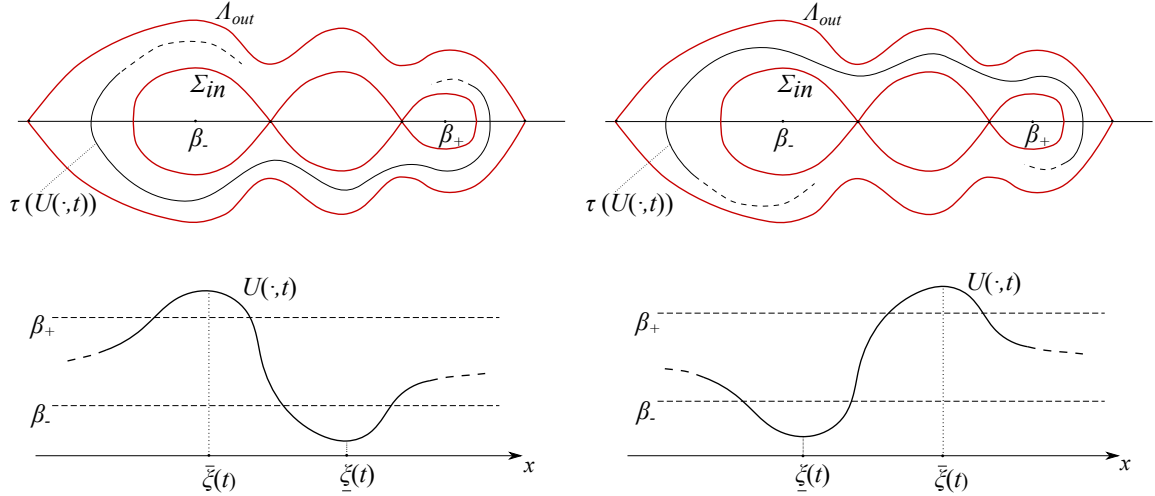}
  %


  
  \caption[Graphs and spatial trajectories]{The spatial trajectories
    and graphs of $U(\cdot,t)$ in the cases $\ol\xi(t)<\ul\xi(t)$
    (the figures on the left) and $\ol\xi(t)>\ul\xi(t)$ (the figures
    on the right). \label{graphs-traj}}  
\end{figure}

Assuming \eqref{eq:33}, the previous lemma shows that \eqref{eq:z2}
holds for all $t\in\R$.  Lemma \ref{le:zerob} now tells us that for
every $t$ the function $U(\cdot,t)$ has exactly two critical points
$\ol\xi(t)$, $\ul\xi(t)$, both nondegenerate, which are the global
maximum and minimum points of $U(\cdot,t)$, respectively. Since
$U(\pm\infty,t)=\Theta_\pm\in (\be_-,\be_+)$, we have
\begin{equation}
  \label{eq:36}
  U(x,t)>\beta_-\ \,(x\in (-\infty,\ol \xi(t)),\ t\in\R) \quad 
  \textrm{ or }\quad \ U(x,t)>\beta_-\ \,(x\in(\ol\xi(t),\infty),\ t\in\R), 
\end{equation}
depending on whether $\ol\xi(t)< \ul\xi(t)$ or $\ol\xi(t)> \ul\xi(t)$
(cp. Figure \ref{graphs-traj}). 
 Similarly, 
\begin{equation}
  \label{eq:36a}
 U(x,t)<\beta_+\ \,(x\in(\ul\xi(t),\infty),\ t\in\R)\quad 
  \textrm{ or }\quad  U(x,t)<\beta_+\ \,(x\in (-\infty,\ul \xi(t)),\ t\in\R). 
\end{equation}
We next prove the conclusion of Proposition \ref{prop:entire} in the
case (C2) when $\Lao$ is a homoclinic loop.
\begin{lemma}
  \label{le:finalc2a1}
  Assume that {\rm (C2)} and \eqref{eq:33} hold, $\Lao$ is a homoclinic loop
  as in {\rm (A1)}, and $U$ is not a steady state. Then
  $\tau(\al(U))\subset\Sigma_{in}$ and $\om(U)=\{\ga\}$, where
  $\ga\in\{\hat p,\hq\}$ is as in {\rm (A1)}.  In particular, \eqref{eq:6a}
  holds.
\end{lemma}
\begin{proof}
  For definiteness, we assume that $\ga=\hp$---so $\Phi$ is a ground
  state at level $\hp$ and $\hq=\Phi(0)$ is its maximum---the case
  $\ga=\hq$ is similar. We prove that for some $\vartheta>0$
  \begin{equation}
    \label{eq:34}
    \max_{x\in\R} U(x,t)<\hq-\vartheta\quad(t\in\R).
  \end{equation}
  Once this is done, the desired conclusion follows immediately from
  Lemma \ref{le:ep}(i).

  Assume that \eqref{eq:34} is not true for any
  $\vartheta>0$. Then there is a sequence $t_n\in\R$ such that
  $U(\ol\xi(t_n),t_n)\upto \hq$ (and $U_x(\ol\xi(t_n),t_n)=0$). As in
  the proof of Lemma \ref{le:basicrel}(i), passing to a subsequence if
  necessary, we have $U(\cdot+\ol\xi(t_n),t_n)\to\Phi$ in
  $C^2_{loc}(\R)$. This and the relations
  $\Phi(\pm\infty)=\ga=\hp<\be_-$ clearly contradict
  \eqref{eq:36}. Thus \eqref{eq:34} indeed holds and the proof is
  complete.
\end{proof}

\begin{figure}[h]
  \vspace{-1cm}
 
  \addtolength{\belowcaptionskip}{20pt}
  \addtolength{\abovecaptionskip}{-15cm}
  \centering
  \hspace{1cm}
  \includegraphics[scale=.76]{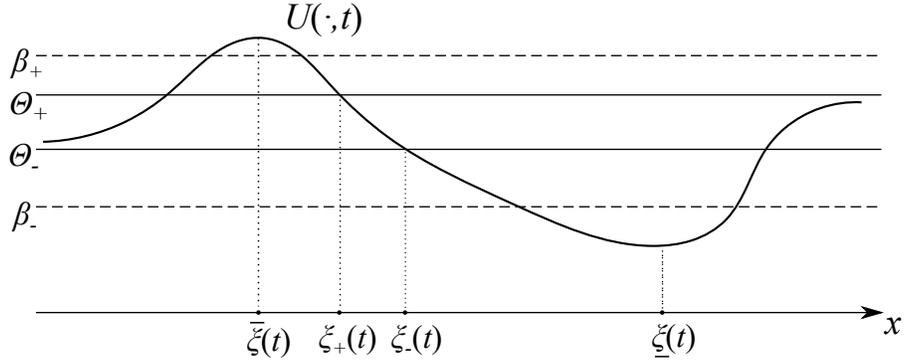}
  \caption[Graph of $U(\cdot,t)$]{The graph of $U(\codt,t)$ when
    $\ol\xi(t)<\ul\xi(t)$. The relation $\xi_+(t)<\xi_-(t)$ holds when
    $\Theta_+>\Theta_-$ (as in the figure), it is reversed when
    $\Theta_+<\Theta_-$, and  $\xi_+(t)=\xi_-(t)$ when
    $\Theta_+=\Theta_-$. 
\label{graph-xi}
  }  
\end{figure}
We now treat the case when $\Lao$ is a heteroclinic loop.
\begin{lemma}
  \label{le:finalc2a2}
  Assume that {\rm (C2)} and \eqref{eq:33} hold, $\Lao$ is a heteroclinic
  loop as in {\rm (A2)}, and $U$ is not a steady state. Then \eqref{eq:6a}
  holds: $\tau(\al(U))\subset\Sigma_{in}$, $\tau(\om(U))\subset\Lao$.
\end{lemma}

\begin{proof}
  With $\ol\xi(t)$, $\ul\xi(t)$ as above, we only consider the case
  $\ol\xi(t)< \ul\xi(t)$; the arguments in the case
  $\ol\xi(t)> \ul\xi(t)$ are similar. Thus (cp. Figure \ref{graph-xi})
  \begin{equation}
    \label{eq:37}
    \begin{aligned}
      &U_x(x,t)>0\quad(x\in (-\infty, \ol\xi(t))\cup
      (\ul\xi(t),\infty),\
      t\in\R),\\
      &U_x(x,t)<0\quad(x\in (\ol\xi(t),\ul\xi(t)),\ t\in\R).
    \end{aligned}
  \end{equation}

  As in the proof of Lemma \ref{le:z2}, we define
  \begin{align}
    \xi_-(t) & :=\min\left\{ x:U(x,t)=\Theta_-\right\}\text{ \ \ (the first zero of $U-\Theta_-$)}, \nonumber \\
    \xi_+(t) & :=\max\left\{ x:U(x,t)=\Theta_+\right\}\text{ \ (the last zero of $U-\Theta_+$)}. \nonumber 
  \end{align}
  Clearly, for all $t\in \R$, $\xi_\pm (t)$ are defined and
  \begin{equation}
    \xi_\pm
    (t)\in (\ol\xi(t),\ul\xi(t)) \label{eq:38}
  \end{equation}
  ($\xi_-(t)$, $\xi_+(t)$ may be equal, or ordered either way,
  depending on the relation between $\Theta_-$ and $\Theta_+$).
 Since $\xi_\pm (t)$ is a simple zero of $U(\cdot,t)-\Theta_\pm$,
  it is a $C^1$ function of $t$.

  We split the rest of the proof into several steps.

  \textit{Step 1.} We show that
  \begin{equation}
    \label{eq:35}
    \tau\left( A(U)\right)\subset\Sigma_{in},
  \end{equation}
  which in particular gives the first inclusion in Lemma \ref{le:finalc2a2}:
  $\tau(\al(U)\subset\Sii$.

  It is sufficient to prove that the constants $\ga_\pm$ are not
  contained in $A(U)$. Indeed, if this holds, then $A(U)$ does not
  contain any shifts of the standing waves $\Phi_\pm$ either (by
  compactness and translation invariance of $A(U)$).  Consequently, by
  Lemma \ref{le:basicrel},
  $\displaystyle dist\left( \tau\left( A(U)\right),\Lambda_{out}
  \right)>0,$ and \eqref{eq:35} follows upon an application of Lemma
  \ref{squeezinglemma}.

  Assume, for a contradiction that $\ga_+\in A(U)$ (arguments to rule
  out the possibility $\ga_-\in A(U)$ are similar and are omitted).
  Clearly, since $\Theta_\pm=U(\pm\infty,t)$, the function
  $U(\cdot,t)$ is monotone neither on $(-\infty,\xi_-(t))$ nor on
  $(\xi_+(t),\infty).$ Therefore, by Lemma \ref{lemmaBPQ}, $\xi_-(t)$
  is bounded from below and $\xi_+(t)$ from above as $t\to-\infty:$
  there is a constant $K>0$ such that
  \begin{equation}\label{eq:331-9}
    \xi_-(t)>-K,\ \xi_+(t)<K,\quad( t<0).
  \end{equation}
  Since $\ga^+\in A(U)$, there is a sequence $t_n\to-\infty$ such
  that, denoting $x_n:=\ol\xi(t_n)$, we have $U(x_n,t_n)\to\ga_+$ (and
  $U_x(x_n,t_n)=0$). As in Lemma \ref{le:basicrel}(i), passing
  to a subsequence if necessary, we obtain that the sequence of
  functions $U_n:=U(\cdot+x_n,\cdot+t_n)$ converges in
  $C^1_{loc}(\R^2)$ to  $\gamma^+$.  Moreover,
  because of \eqref{eq:38}, \eqref{eq:331-9}, we have
  $U(K,t)<\Theta_+$ for all $t<0,$ thus
  $x_n\to-\infty$ as $n\to\infty.$

  Let $\psi$ be any periodic solution of \eqref{eq:steady} with
  $\tau(\psi)\subset\Pi$ and $\psi(0)>\be_+$, $\psi'(0)=0$.
  Let $2\rho>0$ be the minimal period of $\psi$, so
  $\psi(0)$ is the maximum of $\psi$, and
  $\psi(-\rho)=\psi(\rho)<\be_-$ is the minimum of $\psi$.  Obviously,
  for all large enough $n$, say for all $n>n_0$, we have
  \begin{equation*}
    U(\cdot+x_n,t_n)>\psi\text{ on $[-\rho,\rho]$}.
  \end{equation*}
  Also, due to \eqref{eq:331-9} and the convergence $x_n\to-\infty$,  we
  have, making $n_0$ larger if necessary,
  \begin{equation*}
    U(\pm\rho +x_n,t)>\Theta_->\psi(-\rho)=\psi(\rho)\quad(n>n_0,\
    t\in (t_n,0]). 
  \end{equation*}
  Therefore, by the comparison principle, for $n>n_0$,
  \begin{equation*}
    U(x+x_n,t)>\psi(x)\quad (x\in [-\rho,\rho], \quad t>t_n).
  \end{equation*}
  In particular, at $t=0$, we obtain
  \begin{equation*}
    \max_{x\in [-\rho,\rho]} U(x+x_n,0)\ge \max\psi>\be_+ \quad
    (n>n_0).
  \end{equation*}
  Since $x_n\to -\infty$, we obtain a contradiction to the fact that
  $U(-\infty,0)=\Theta_-<\be_+$.  This contradiction completes the proof of
  \eqref{eq:35}.

  \textit{Step 2.}  We show that
  \begin{equation}\label{eq:341-3}
    \tau\left(\omega(U)\right)\subset\Sigma_{in} \textrm{ or } \tau\left(\omega(U)\right)\subset\Lambda_{out}.
  \end{equation}
  Note that due to \eqref{eq:37}, in both cases $\Theta_-=\Theta_+$
  and $\Theta_-\neq\Theta_+,$ the solution $U(\cdot,t)$ satisfies the
  hypotheses of Theorem \ref{thmPP1}. So  $\omega(U)$
  consists of steady states, and it does not contain non-constant
  periodic functions (cp. Corollary \ref{co:tilde}).  So
  $\tau(\omega(U))\subset \bar\Pi\setminus\mathcal{P}_0$ and it is
  connected.  This gives \eqref{eq:341-3}.

  \textit{Step 3.} In this step we complete the proof of Lemma
  \ref{le:finalc2a2} by showing that
  $\tau(\omega(U))\subset\Lao$. 
  In view of \eqref{eq:341-3}, we just need to rule out the possibility
  \begin{equation}
    \label{eq:40}
    \tau(\omega(U))\subset\Sigma_{in}.
  \end{equation}
  Assume it holds. We derive a contradiction. Pick a sufficiently
  small $\e>0$ such that
  $$\ga_-=\hat p<p-\e,\quad \ga^+= \hat q>q+\e.$$
  Relation \eqref{eq:40} in particular implies that
  for any  $M>0$ there exists $T=T(M)$ such that
  \begin{equation}\label{eq:341-6}
    p-\e<U(x,t)<q+\e\quad(x\in(-M,M),\ t>T(M)).
  \end{equation}

  By Step 1 and Lemma \ref{le:basicrel}(ii), $\Om(U)$
  contains one of the constants $\ga_\pm$ (or a shift of one of the
  standing waves $\Phi_\pm$, and, consequently, also both constants
  $\ga_\pm$).  We only consider the case $\gamma_+\in\Omega(U)$, the
  case $\gamma_-\in\Omega(U)$ being similar.  Hence, there is a
  sequence $t_n\to\infty$ such that, denoting $x_n:=\ol\xi(t_n)$, we
  have
  \begin{equation}\label{eq:341-5}
    U(\cdot+x_n,t_n)\underset{n\to\infty}{\longrightarrow}\gamma_+,
  \end{equation}
  with the convergence in $L^\infty_{loc}(\R).$
 Clearly, \eqref{eq:341-5}, \eqref{eq:341-6} imply that
  $|x_n|\to\infty$.  We claim that necessarily $x_n\to-\infty.$
  Observe that, by
  \eqref{eq:37} and \eqref{eq:38},
  \begin{align}\label{star}
    &U(x,t)>\min(\Theta_-,\Theta_+) \quad
      (-\infty<x<\max\{\xi_-(t),\xi_+(t\}),\\ 
    &U(x,t)<\max(\Theta_-,\Theta_+)\quad
      (\min\{\xi_-(t),\xi_+(t)\}<x<\infty);\label{star1}
  \end{align}
  and $U_x(\xi_\pm(t),t)<0.$ Using the last relation and
  Lemma \ref{lemmaBPQ}, we obtain the
  following monotonicity relations for all $t>0$:
  \begin{align}
    \textrm{if }\xi_-(t)>\xi_-(0), &\textrm{ then }U_x(\cdot,t)<0 \textrm{ on }\left(\xi_-(0),\xi_-(t)\right), \label{eq:341-42} \\
    \textrm{if }\xi_+(t)<\xi_+(0), &\textrm{ then }U_x(\cdot,t)<0 \textrm{ on }\left(\xi_+(t),\xi_-(0)\right). \label{eq:341-43} 
  \end{align}
  From \eqref{star1} and \eqref{eq:341-5}, it follows that there is
  $n_1$ such that 
  $\xi_\pm(t_n)>x_n$ for all  $n>n_1$. If for some $n>n_1$
   it is also true that $x_n>\xi_-(0)$, then the relations
  $\xi_-(0)<x_n<\xi_-(t_n)$ and 
  \eqref{eq:341-42} give $U(\xi_-(0),t_n)>U(x_n,t_n)$. This inequality
  can hold only for finitely many $n$, due to
  \eqref{eq:341-6}, \eqref{eq:341-5}. Thus for all large enough $n$ we
  have $x_n\le \xi_-(0)$, hence, since $|x_n|\to\infty$, it must be
  true that $x_n\to-\infty$, as claimed. 
  
  Pick now a periodic solution $\psi$ of \eqref{eq:steady} with
  $\tau(\psi)\subset\Pi$ such that $\min\psi<p-\e$ and
  $\max\psi>q+\e$. We shift $\psi$ such that 
  $\max\psi=\psi(0)$. Let $2\rho>0$ be the
  minimal period of $\psi$. Thus we have
  \begin{equation*}
    \min\psi=\psi(\pm\rho)<p-\e,\qquad \max\psi=\psi(0)>q+\e.
  \end{equation*}
By \eqref{eq:341-5}, for $n$ large enough,
  \begin{equation}\label{dot}
    U(x_n+x,t_n)>\psi(0)=q+\e\quad(x\in(-\rho,\rho)).
  \end{equation}
By \eqref{star1}, necessarily $\xi_\pm(t_n)>x_n+\rho$. We 
now show that for some large enough $n_0$, the following must
  hold in addition to \eqref{dot}:
  \begin{equation}\label{eq:341-7}
    U(x_{n_0}\pm\rho,t)>\psi(\pm\rho)\quad( t>t_{n_0}).
  \end{equation}
  If this does not hold, then there exist arbitrarily large $n$ such that 
  for some  $\tilde{t}_n>t_n$ one has
  $\displaystyle U\left( x_n+\bar\rho,\tilde{t}_n\right)=\psi(\bar\rho)=p-\e$,
  where $\bar\rho$ is either $-\rho$ or $\rho$.
  Since $U(\cdot,t)>\min(\Theta_-,\Theta_+)$ on
  $(-\infty,\xi_+(t))$ (cp. \eqref{star}), it follows that
  $\displaystyle\xi_+\left(\tilde{t}_n\right)<x_n+\bar\rho$.
  But, due to
  $x_n\to-\infty$, we also have $x_n+\bar\rho<\xi_+(0)$ if $n$ is large
  enough; so, by \eqref{eq:341-43},
  $\displaystyle U\left(\xi_+(0),\tilde{t}_n\right)<U\left(
    x_n+\bar\rho,\tilde{t}_n\right)=p-\e$. This cannot be true for
  arbitrarily large $n$, due to
  \eqref{eq:341-6}, 
  so \eqref{eq:341-7} must indeed
  hold for some, arbitrarily large, $n_0$.

  Using \eqref{dot}, \eqref{eq:341-7}, and  the comparison principle,
  we obtain $U(x_{n_0},t)>\psi(0)=q+\e,$
  for all $t>t_{n_0}.$  This is a contradiction to \eqref{eq:341-6}.

  We have shown that the assumption \eqref{eq:40} leads to a
  contradiction, which
  concludes the proof of Lemma \ref{le:finalc2a2}.
\end{proof}

\subsubsection{Case (C3):  $(\Theta_-,0)\in \Sigma_{in}$  and
  $(\Theta_+,0)\in \Lambda_{out}$}
Our assumption in this subsection is that 
$(\Theta_-,0)\in \Sigma_{in}$ and $(\Theta_+,0)\in \Lambda_{out}$.
The case
$(\Theta_+,0)\in \Sigma_{in}$ and $(\Theta_-,0)\in \Lambda_{out}$ is
 analogous and will be skipped. 

For definiteness, we also assume that $\Theta_+=\hp$ (hence
$f(\hp)=0$); the other possibility, $\Theta_+=\hq$, can be treated in
an analogous way.
  
\begin{lemma}
  \label{le:c3om}
  Under condition {\rm (C3)},
  $ \tau\left(\omega(U)\right)\subset\Lambda_{out}$.
\end{lemma}
\begin{proof}
  Theorem \ref{thmPP1} implies that $U$ is quasiconvergent, hence and
  $\omega(U)$ contains only non-periodic steady states or constant
  steady states.

  If $\Lao$ is a heteroclinic loop, as in (A2), we choose a decreasing
  continuous function $\tilde u_0$ such that
  $\tilde u_0(-\infty) =\ga_+>\tilde u_0 >\ga_-=\tilde u_0(\infty)$
  and $\tilde u_0\ge U(\codt,0)$. By the comparison principle,
  the corresponding solution
  $\tilde u=u(\cdot,\cdot,\tilde u_0)$ of \eqref{eq:1} satisfies
  $\tilde u(\cdot,t)>U(\cdot,t)$ for all $t>0$.  By
  \cite[Theorem 3.1]{FMcL}, the (front-like) solution
  $\tilde u(\codt,t)$ converges in $L^\infty(\R)$ to a shift of the
  decreasing standing wave $\Phi^-$, say $\Phi^-(\codt-\eta)$, as
  $t\to\infty$.  This implies that for all $\varphi\in\om(U)$, we have
  $\varphi\le \Phi^-(\codt-\eta)$. Now, every $\varphi\in\om(U)$ is a
  steady state with 
  $\tau(\varphi)\subset\bar\Pi$.
  Therefore, $\varphi\le \Phi^-(\codt-\eta)$ implies that
  $\varphi$ is identical to
  $\hq=\ga^-$ or to a shift of $\Phi^-$. In particular,
  $\tau(\varphi)\subset \Lao$.

  If $\Lao$ is a homoclinic loop, as in (A1), we have $\ga=\hp$ since
  we are assuming that $\Theta^+=\hp$ and $f(\hp)=0$.  The arguments
  here are similar as for the heteroclinic loop, but instead of a
  front-like solution, we compare $U$ to a solution which converges to
  a shift of the ground sate $\Phi$.  For that, we find a continuous
  function $\tilde u_0$ with the following properties:
  \begin{itemize}
    \item[(s1)]\quad $\tilde u_0$ is even and monotone
      nondecreasing in $(-\infty,0)$;
    \item[(s2)]\quad $\tilde u_0>\ga, \quad\tilde u_0(\pm\infty)=\ga$;
    \item[(s3)]\quad $\tilde u_0(x)\ge U(x,0)$ for all sufficiently large $x>0$;
    \item[(s4)]\quad the 
      solution   $\tilde u(\cdot,t):=u(\cdot,t,\tilde u_0)$ converges
      in $L^\infty(\R)$ to  $\Phi$  as $t\to\infty$. 
  \end{itemize}
  Such a function $\tilde u_0$ is provided by \cite[Theorem
  2.6]{Matano_Polacik_CPDE16}. More specifically, take first a
  continuous function $u_1$ satisfying (s1)--(s3)
  and such that $u_1\equiv \be_+$ on the interval $(-\ell,\ell)$ (so,
  in particular, $u_1\le \be_+$). According to Theorem~2.6 of
  \cite{Matano_Polacik_CPDE16}, if $\ell$ is sufficiently large, then
  for some $\la>1/2$ the function $\tilde u_0:=\ga+2\la (u_1-\ga)$
  satisfies (s4) (the corresponding solution
  $\tilde u(\cdot,t)$ is a threshold solution in the terminology of
  \cite{Matano_Polacik_CPDE16}); it obviously satisfies
    (s1)--(s3) as well.

  Since $U(x,0)$ is decreasing for large $x>0$, relations (s1)--(s3)
   imply that for a sufficiently large
  $\eta>0$ we have $z(\tilde u_0(\cdot-\eta)-U(\codt,0))=1$.
  Therefore, $z(\tilde u(\cdot-\eta,t)-U(\codt,t))\le 1$ for all
  $t>0$.  As a consequence, taking into account that the difference of
  any two steady states \eqref{eq:1} has only simple zeros, we have
  $z(\Phi(\codt-\eta)-\varphi)\le 1$ for all $\varphi\in \om(U)$.
  Therefore, any $\varphi\in\om(U)$ is identical to $\hq=\ga$ or to a
  shift of $\Phi$. In particular, $\tau(\varphi)\subset \Lao$.

\end{proof}

Turning our attention to $\alpha(U)$, we start with a preliminary
lemma.

\begin{lemma}\label{le:3.17}
  Assume {\rm (C3)} holds.  Then $\alpha(U)$ does not contain any function
  $\varphi$ with $\tau(\varphi)\subset \Lao$.
\end{lemma}
\begin{proof}
  We assume that
  \begin{equation}
    \label{eq:33a}
    N^\pm=z(U(\cdot,t)-\be_\pm)>0,
  \end{equation}
  the case when $N^-=0$ or $N^+=0$ having been settled by Lemma
  \ref{le:ep}(iii).


  Recall that we are also assuming, without loss of generality, that
  $f(\hp)=0 $ and $\Theta_+=\hp$.  Thus,
  $U(-\infty,t)=\Theta_-\in(\beta_-,\beta_+)$ and $U(\infty,t)=\hp$
  for all $t\in\R.$ Using Lemma \ref{le:zerob},
  we obtain that the function $U(\cdot,t)$ is
  decreasing to $\hp$ as $x\to\infty$, and monotone near
  $x=-\infty$, and the function $U(\cdot,t)-\Theta_-$ has only
  finitely many zeros, all of them simple, with
  $N:=z(U(\cdot,t)-\Theta_-)$ independent of $t$.

  We denote by $\eta(t)$ the first zero of $U(\cdot,t)-\Theta_-.$
  Since $\Theta_-=U(-\infty,t)$, the function $U(\cdot,t)$ is not
  monotone on $(-\infty,\eta(t))$. Therefore, by Lemma \ref{lemmaBPQ},
  there is $\kappa\in\R$ such that
  \begin{equation}
    \eta(t)>\kappa\quad (t<0).\label{eq:41}
  \end{equation}
  We distinguish the following two possibilities
  \begin{itemize}
  \item[(pi)] \quad $U(\cdot,t)<\Theta_-$ on
    $(-\infty,\eta(t))$ \quad ($t\in\R$) 
  \item[(pii)] \quad $U(\cdot,t)>\Theta_-$ on
    $(-\infty,\eta(t))$ \quad ($t\in\R$)
  \end{itemize}

  Assume (pi). Then \eqref{eq:41} implies that for all $t<0$
  \begin{equation}\label{eq:333-4a}
    U(x,t)\leq\Theta_-<\beta_+\quad(x\le\kappa)
  \end{equation}
  and any function in $\alpha(U)$ has to satisfy this relation.  In
  particular, the constant $\hq$ and any shift of $\Phi_-$ (if $\Lao$
  is a heteroclinic loop) are ruled out from $\al(U)$.  It remains to
  rule out the constant $\hp$, any shift of the ground state $\Phi$
  (if $\Lao$ is a homoclinic loop), and any shift of $\Phi_+$ (if
  $\Lao$ is a heteroclinic loop). Take any of these functions,
  denoting it by $\varphi$, and assume for a contradiction that
  $\varphi\in \alpha(U)$. By \eqref{eq:333-4a},
  \begin{equation}
    \label{eq:42}
    \varphi(x)\leq\Theta_-<\beta_+\quad(x\le\kappa).
  \end{equation}
  Take any periodic solution $\psi$ of \eqref{eq:steady} with
  $\tau(\psi)\subset\Pi$, $\psi(\kappa)=\beta_+,$ and
  $\psi'(\kappa)>0$.  Obviously, there is $\rho>0$ such that
  $\psi(\kappa-\rho)=\beta_+.$ We claim that
  \begin{equation}
    U(x,t)<\psi\quad(x\in[\kappa-\rho,\kappa],\ t\le0).\label{eq:43}
  \end{equation}
  Due \eqref{eq:333-4a}, this follows from the comparison principle if
  we can find a sequence $t_n\to-\infty$ such that the claim is valid
  for $t=t_n$, $n=1,2,\dots$. Note that the function $\varphi$, fixed
  as above, satisfies $\varphi<\psi$ on $(\kappa-\rho,\kappa)$.  This
  is trivial if $\varphi\equiv \hp$; if $\varphi$ is a shift of the
  ground state or the increasing standing wave, it follows from
  \eqref{eq:42} (otherwise, a shift of the graph of $\psi$ we would be
  touching the graph  $\varphi$ at some point,
  which is impossible for two distinct solutions of \eqref{eq:steady}).
  Since $\varphi\in \al(U)$, there is a sequence $t_n\to-\infty$ such that
  $U(\cdot,t_n)\to \varphi$ locally uniformly.  This sequence,
  possibly after omitting a finite number of terms, has the desired
  property.

  Thus, \eqref{eq:43} has to hold for any periodic solution $\psi$
  with the indicated properties. We now choose a sequence of such
  periodic orbits $\psi_k$ converging locally uniformly on $\R$ to a
  shift of the ground state (if $\Lao$ is a homoclinic loop) or a
  shift of $\Phi_+$ (if $\Lao$ is a heteroclinic loop).  In either
  case, the relations \eqref{eq:43} with $\psi=\psi_k$, $k=1,2,\dots$
  and $t=0$, contradict the relations $U(-\infty,0)=\Theta_->\hp$. This
  contradiction completes the proof if (pi) holds.  
 
  Now assume (pii). Then \eqref{eq:41} implies that for all $t<0$
  \begin{equation*}
    U(x,t)\geq\Theta_->\hp\quad (x\le\kappa).
  \end{equation*}
  Therefore, each function in $\alpha(U)$ has to satisfy this
  inequality, which shows that the following functions cannot be
  elements of $\alpha(U)$: the constant $\hp$, any shift of the ground
  state $\Phi$ (if $\Lao$ is a homoclinic loop), any shift of the
  increasing standing wave $\Phi_+$ (if $\Lao$ is a heteroclinic
  loop). Thus,  we only need to show that if $\Lao$ is a
  heteroclinic loop, then $\al(U)$ does not contain the constant $\hq$
  or any shift of $\Phi_-$. The arguments for this are analogous to
  the arguments used in the case (pi) to show that $\al(U)$
  does not contain the constant $\hp$ or any shift of $\Phi_+$
  and are omitted. 
\end{proof}

We conclude this section by the following lemma, which, in conjunction
with Lemma \ref{le:c3om}, shows that \eqref{eq:6a} holds in the case
(C3) as well.

\begin{lemma}\label{le:3.18}
  Assume {\rm (C3)}. Then $\tau\left( \alpha(U)\right)\subset\Sigma_{in}.$

\end{lemma}
 
\begin{proof}
  From lemma \ref{le:3.17} (combined with Lemma \ref{le:2.7}), we know
  that for any $\vp_0\in\alpha(U),$ the trajectory $\tau(\vp_0)$
 is disjoint from $\Lambda_{out}.$

  Fix an arbitrary $\vp_0\in\alpha(U)$, we need to prove that
  $\tau(\varphi_0)\subset \Sii$. Let $\tilde{U}$ be the entire
  solution with $\tilde{U}(\cdot,0)=\vp_0$. Then $\tilde U$ satisfies
  (HU) (cp. Corollary \ref{co:tilde}). Therefore, we may apply to
  $\tilde U$ the results concerning the $\om$-limit set already proved
  in this subsection and in the previous two subsections.  Thus, if
  $\vp_0$ is not a steady state, then
  $\tau(\om(\tilde U))\subset \Lao$. This would mean---since
  $\tau(\om(\tilde U))\subset\al(U) $ by the invariance and
  compactness of $\al(U)$---that $\al(U)$ contains a
  function $\vp$ with $\tau(\vp)\subset \Lao$, in contradiction to
  Lemma \ref{le:3.17}.  Therefore, $\vp_0$ has to be a steady
  state. It is not periodic (cp. Corollary \ref{co:tilde}) and
  $\tau(\vp_0)$ is not contained in $\Lao$ by Lemma \ref{le:3.17}. We
  are left with the desired option $\tau(\vp_0)\subset \Sii$, completing the
  proof.
\end{proof}



\section{Proof of Proposition \ref{prop:3.3} in the case $\Pi=\Pi_0$}
\label{completionprop}
Proposition \ref{prop:entire}, proved in the previous section,
implies 
that statement (ii) of Proposition \ref{prop:3.3} holds if
$\Pi\ne \Pi_0$. We now consider the case $\Pi= \Pi_0$ (and so
$\Sii=\{(0,0)\}$), assuming  that conditions (U), (NC), (R) hold.
 We further assume that
$\vp\in\omega(u)$, 
  $U$ is the entire solution of \eqref{eq:1} with $U(\cdot,0)=\vp$,
  and \begin{equation}
      \label{eq:5t}
     \underset{t\in\R}{\textstyle \bigcup}\tau\left( U(\cdot,t)\right)\subset\Pi_0.
    \end{equation}
    We prove that
    \begin{equation}
      \label{eq:6t}
      \alpha(U)=\{0\},\qquad \tau\left(\omega(U)\right)\subset\Lambda_{out},
    \end{equation}
    thereby completing the proof of Proposition \ref{prop:3.3} (note that
    $\alpha(U)=\{0\}$ is equivalent to $\tau(\alpha(U))=\{(0,0)\}=\Sii$).

    We use a similar  notation as in the previous section:
    \begin{equation}
  \label{eq:19t}
  \begin{aligned}
    \hat p&:=\inf\{a\in\R: (a,0)\in \Pi_0\}=\inf\{a\in\R: (a,0)\in
    \La_{out}\},\\ 
    \hat q&:=\sup\{a\in\R: (a,0)\in \Pi_0\}=\sup\{a\in\R: (a,0)\in
    \La_{out}\}.
  \end{aligned}
\end{equation}
Thus, $\{\hat p,\hat q\}=\{\ga, \Phi(0)\}$ if (A1) holds; and
$\hat p=\ga_-$, $\hat q =\ga_+$ if (A2) holds, where conditions (A1),
(A2) are as in Section \ref{sec:entire} (cp. \eqref{eq:21},
\eqref{eq:22}).

Recall from Lemmas \ref{le:pmnm}, \ref{le:bddcp} that
$k:=z(U(\cdot,t))$, $\ell:=z(U_x(\cdot,t))$ are finite and independent of $t$,
all zeros of $U(\cdot,t)$, $U_x(\cdot,t)$ are  
simple for all $t$, and the zeros of $U_x(\cdot,t)$ are
contained in an interval $(-d,d)$ independent of $t$. 
Further,  $U(\cdot,t)$ has no positive local minima and no negative
local maxima. This means that the zeros and critical points of
$U(\cdot,t)$ alternate.

Clearly, one of the following possibilities occurs:
\begin{equation}
  \label{eq:55}
  \text{$k=0$;   \quad $\ell\ge 2$; \quad $\ell=1$ and $k=1$; \quad
  $\ell=1$ and $k=2$; \quad $\ell=0$ and $k=1$.}
\end{equation}
We differentiate with respect to these four possibilities.

If $k=0,$ then \eqref{eq:6t} is a direct consequence of Lemma \ref{le:ep}(i).

Next we show that $\ell\geq2$ is impossible.
Assume it holds.
Then $U(\cdot,t)$ has at least one zero contained between
two successive critical points---hence contained in $(-d,d)$---for all
$t$. If  $U(\pm \infty,t)\le 0$, then the function
$U(\cdot,t)$ assumes its positive global maximum in $(-d,d)$, at one of the
local maxima.
We claim that $U$ has to stay below $\hat q-\vartheta$ for some
$\vartheta>0.$ Assume otherwise:  
then there exist sequences $(x_n)$ in $(-d,d)$ and $(t_n)$ in $\R$ such that 
$(U(x_n,t_n),U_x(x_n,t_n))$ converges to $(\hat q,0)$ as $n\to\infty,$
and by Lemma 
\ref{le:basicrel}(i), up to a subsequence, $U(\cdot+x_n,\cdot+t_n)$
converges in $C^1_{loc}(\R)$ to some steady state $\phi$
as $n\to\infty$ with $\tau(\phi)\subset\Lambda_{out}.$ On the other
hand, $U(\cdot+x_n,t_n)$ admits two critical  
points  in $(-d,d)$  where it takes opposite signs,  which clearly
contradicts the convergence to $\phi$. Thus our claim is proved and
Lemma \ref{le:ep}(i) now implies that 
$\om(U)=\{\hat p\}$. This, however, is also prevented by the existence of
a zero in   $(-d,d)$ and we have a contradiction. Similarly one shows
that the relations $U(\pm \infty,t)\ge 0$ lead to a contradiction.
If $U(- \infty,t)> 0>U(\infty,t)$ or $U(- \infty,t)< 0<U(\infty,t)$,
then, by 
  \cite[Theorem 3.1]{FMcL}, the (front-like) solution
  $U(\codt,t)$ converges as $t\to\infty$ to
  a standing wave of \eqref{eq:1} 
  in  $C_b^1(\R)$. But this implies that  $U(\codt,t)$ has no critical
  points in $[-d,d]$, and we have a  contradiction again.

Now consider the case  $\ell=1=k$. We denote by
$\xi(t)$, $\eta(t)$ the critical point and the zero of $U(\cdot,t)$,
respectively. 
For definiteness, we assume that $\xi(t)<\eta(t)$ and
$U(\cdot,t)>0$ in $(-\infty,\eta(t))$; the other possibilities that
can occur in the case $\ell=k=1$ can be treated similarly. It follows that $\xi(t)$
is the point of positive maximum of $U(\cdot,t).$
First, we dispose of the possibility that $\Lao$ is a homoclinic
loop. Since $\xi(t)<\eta(t)$ and $\xi(t)$ is
the unique critical point of $U(\cdot,t)$, we have
$U_x(\cdot,t)<0$ in $[\eta(t),\infty)$, in particular
$U(\infty,t)<U(\eta(t),t)=0$. Therefore, $U(\infty,t)$ converges to a
stable equilibrium of the equation $\dot \zeta=f(\zeta)$
as $t\to\infty$. In 
 view \eqref{eq:5t}, this equilibrium has to equal $\hat p$, which
 gives $f(\hat p)=0$. So if $\Lao$ is a homoclinic
loop as in (A1), the ground state $\Phi$ satisfies
$\Phi(\pm\infty)=\hat p$ and $\Phi(0)=\hat q$.
Using Lemma \ref{le:basicrel}(i) and the facts that
$\Phi$ has two zeros while $U(\cdot,t)$ has only one and that
$U(\xi(t),t)>0$ with $\xi(t)$ bounded,  
one shows easily that the global maximum of
$U(\cdot,t)$, namely $U(\xi(t),t)$,
has to stay below  $\hat q-\vartheta$ for some
$\vartheta>0$. By Lemma \ref{le:ep}(i), 
$\om(U)=\{\hat p\}$, which contradicts the  relation 
$U(\xi(t),t)>0$. We may thus proceed assuming that $\Lao$ is a heteroclinic
loop, in particular $\hp$, $\hq$ are zeros of $f$.
As above, if $U(\xi(t),t)$ 
stays below  $\hat q-\vartheta$ for some
$\vartheta>0$, then Lemma \ref{le:ep}(i) yields a contradiction. 
Thus, there is a sequence $t_n\in \R$ such that that
$U(\xi(t_n),t_n)\to \hq $ and, then, by Lemma \ref{le:basicrel}(i)
and the fact that $\xi(t_n)\in (-d,d)$, up to some subsequence,
$U(\cdot,t_n)\to \hq $
in $L^\infty_{loc}(\R)$. Obviously, the sequence $\{t_n\}$ is
unbounded.  Pick any periodic solution $\psi$ of \eqref{eq:steady} with
$\tau(\psi)\subset\Pi_0$ and $\min\psi=\psi(-d)\le 0$. Let $2\rho>0$ be the
minimal period of $\psi$. Then, for any large enough $n$
we have $U(\cdot,t_n)>\psi$ in $[-d-2\rho, -d]$. Since also
$U(\cdot,t)>0\ge \psi(-d)=\psi(-d-2\rho)$ in $(-\infty,-d]$,
the comparison principle
gives $U(\cdot,t)>\psi$ in $[-d-2\rho,-d]$ for all $t>t_n$. Consequently, 
\begin{equation}
\label{eq:56}
U(x,t)>\psi(-d-\rho)=\max \psi\quad (x\in [-d-\rho,-d], \, t>t_n),
\end{equation}
since $U_x(\cdot,t)>0$ in $(-\infty,-d)$ (the only critical point of
$U(\cdot,t)$ is in $(-d,d)$ and it is the maximum point). 
Using \eqref{eq:56} and taking admissible periodic solutions
with $\max\psi\to\hq$ (and $\rho\to\infty)$, we obtain 
two conclusions. First, $U(\cdot,t)\to \hq$
in $L^\infty_{loc}(-\infty,-d)$ as $t\to\infty$ and, consequently,
$\tau(\om(U))=\{(\hq,0)\}\subset \Lao$. Second, 
the sequence $\{t_n\}$ has to be bounded from below (otherwise
\eqref{eq:56} leads to $U(\cdot,0)\equiv \hq$, which is absurd).
This implies that there is $\vartheta>0$ such that
$U(\xi(t),t)<\hat q-\vartheta$ for all $t<0$.
Lemma \ref{le:ep}(ii) now shows that $\tau(\al(U))=\{(0,0)\}$,
completing the proof of \eqref{eq:6t} in the case $\ell=1=k$.

In the case $\ell=1$ and $k=2$, 
we denote by
$\xi(t)$ the unique critical point  of $U(\cdot,t)$ and assume
for definiteness
$U(\xi(t),t)>0$, so $U(\xi(t),t)$ is the global maximum of
$U(\cdot,t)$. If $U(\xi(t),t)<0$, the arguments are analogous.
Since $k=2$, we have  $U(\pm \infty,t)<0$.
The possibility
$U(\pm \infty,t)=\hp$ can be treated in much the same way
as the case (C1) with $\Theta_-=\Theta_+$ in Subsection \ref{subsc1}:
\eqref{eq:6t} holds in this case. 
Consider the opposite possibility,
$U(-\infty,t)>\hp$ or $U(\infty,t)>\hp$. We assume, again just for
definiteness, that the former holds. Then,
since $U(-\infty,t)$ is a solution of $\dot \zeta=f(\zeta)$, 
$U(-\infty,t)\to 0$ as $t\to-\infty$ and $U(-\infty,t)\to \hp$ as
$t\to \infty$. In particular, $f(\hp)=0$. If $\Lao$ is a
heteroclinic loop, it is easy to prove, using
\cite[Theorem  3.1]{FMcL} as in \cite[Proof of Lemma 3.4]{P:examples}
for example, that $U(\cdot,t)\to \hp$ as $t\to\infty$, uniformly on
$\R$. This, of course, is impossible when $k=2$. Thus $\Lao$ has to be
a homoclinic loop, as in (A1), and the ground state $\Phi$ satisfies
$\Phi(\pm\infty)=\hat p$ and $\Phi(0)=\hat q$.
We claim that there is $\vartheta>0$ such that
$U(\xi(t),t)<\hat q-\vartheta$ for all $t<0$.
Indeed, otherwise, by Lemma \ref{le:basicrel}(i) and the boundedness
of $\xi(t)$,
there is a sequence $t_n\to-\infty$ such that
$U(\cdot,t_n)$ approaches a shift of the ground state in
$L_{loc}^\infty(\R)$. This in
conjunction with the property that $U(-\infty,t)\to 0$
as $t\to-\infty$ would imply that $U(\cdot,t_n)$ has more than one
critical point if $n$ is large enough, a contradiction to $\ell=1$.
Thus, our claim is proved and 
Lemma \ref{le:ep}(ii) now implies
that $\tau(\al(U))=\{(0,0)\}$.
For the proof of 
\eqref{eq:6t}, we now prove that
$\dist((0,0),\tau(\om(U)))>0$. Once proved,
this will yield a nonstationary periodic orbit $\cO$ of \eqref{eq:sys}
such that
$\tau(\om(U))\subset\R^2\setminus{\overline{\mathcal{I}}(\mathcal{O})}$
from which \eqref{eq:6t} follows at once upon an application of Lemma
\ref{squeezinglemma}. Assume for a contradiction that
$\dist((0,0),\tau(\om(U)))=0$. Then
Lemma \ref{le:basicrel}(i) yields a sequence $(x_n,t_n)\in\R^2$ such
that $t_n\to\infty$ and $U(\cdot+x_n,t_n)\to 0$ in
$C^2_{loc}(\R)$. This implies that given any periodic solution
$\psi$ of \eqref{eq:steady} with $\tau(\psi)\subset\Pi_0$ one has
$z(U(\cdot,t_n)-\psi)\to\infty$ as $n\to\infty$. On the other hand,
since $U(\pm \infty,0)<0$, picking $\psi$ near 0, so that
$\psi>U(\pm \infty,0)$, we obtain that for $t\ge 0$ 
the zero number $z(U(\cdot,t)-\psi)$ is finite and therefore bounded
as $t$ increases to infinity. This contradiction completes the proof
of \eqref{eq:6t} in the case  $\ell=1$ and $k=2$.

Finally, we deal with the case $\ell=0$ and $k=1$. 
Clearly,
$U(\pm \infty,t)$ are nonzero and have opposite signs.
Assume for definiteness that $U(-\infty,t)<0<U(\infty,t)$.
The assumption $\ell=0$ then means that $U_x>0$ everywhere. 
Being solutions of $\dot \zeta=f(\zeta)$,
$U(\pm \infty,t)$ converge to stable equilibria of this equation as
$t\to\infty$. By \eqref{eq:5t}, these equilibria have to be $\hp$, $\hq$: 
$U(-\infty,t)\to \hp$, $U(\infty,t)\to\hq$. In particular,
$f(\hp)=f(\hq)=0$
and $\Lao$ is
a heteroclinic loop. Using
\cite[Theorem  3.1]{FMcL}, we obtain that the (front-like) solution
$U(\cdot,t)$ approaches a
shift of the increasing standing wave $\Phi_+$, as $t\to\infty$,
uniformly on $\R$, so $\tau(\om(U))\subset \Lao$. 
We now claim that $\al(U)=\{0\}$.  If
$U(-\infty,t)=\hp$, $U(\infty,t)=\hq$, our claim can be proved by
essentially the same arguments 
as those used in the case
(C1) with $\Theta_-<\Theta_+$ in Subsection \ref{subsc1};
see the proof of Lemma \ref{le:c1alpha}, the relevant part 
starts with ``Case (i) of Lemma \ref{le:conv}.''
If
$U(-\infty,t)>\hp$, then $U(-\infty,t)\to 0$ as $t\to-\infty$. This, 
in conjunction with the relation $U_x>0$,
implies that $\phi\ge 0$ for any $\phi\in\al(U)$. Similarly, if
$U(\infty,t)<\hq$, then  $\phi\le 0$ for
any $\phi\in\al(U)$. Thus, if $U(-\infty,t)>\hp$ and
$U(\infty,t)<\hq$, we are done:  $\al(U)=\{0\}$.
It remains to consider the possibility when exactly one of this
inequalities holds, say $U(-\infty,t)>\hp$ and
$U(\infty,t)=\hq$ (the case $U(-\infty,t)=\hp$ and
$U(\infty,t)<\hq$ is analogous).
If there is $\phi\in\al(U)$, $\phi\not\equiv 0$, then
$\al(U)$ contains the constant $\hq$. To see this take the
solution $\tilde U$ of \eqref{eq:1} with $\tilde U(\codt,0)=\phi$
and $\tilde U(\codt,t)\in \al(U)$ for all $t \in\R$. Then, since
$\phi\ge 0$,   $\phi\not\equiv 0$, we have 
$\tilde U(\codt,t)\to \hq$ in $L^\infty_{loc}(\R)$. The compactness
of $\al(U)$ gives $\hq\in\al(U)$, as claimed. This, however, can be
proven to be contradictory by the same comparison argument
involving the function \eqref{eq:57} as in the proof of
Lemma \ref{le:c1alpha}. This shows that $\al(U)=\{0\}$ and completes
the proof of \eqref{eq:6t} in the case $\ell=0$, $k=1$.

\section{Morse decompositions and the  proofs 
	of the quasiconvergence results}\label{morse-dec}

In this section, we prove Theorems \ref{thm:1}, \ref{thm:2}, and
\ref{thm:3}, giving also a stronger and more precise version of
Theorem \ref{thm:2}, see Theorem
\ref{thm:4} below. We prepare the proofs by recalling
some results concerning Morse decompositions and chain recurrence. 

Throughout this section, we assume that $u_0$ is as in
\eqref{localizedu0}---specifically, 
$u_0\in\mathcal{V}$ and  $u_0(-\infty)=u_0(+\infty)=0$---and
the solution of \eqref{eq:1}, \eqref{ic1}
is bounded. In what follows,
the $\om$-limit set of this solution,  $\om(u_0)$,
is viewed as a compact subset of $L^\infty_{loc}(\R)$
equipped with the induced topology
(in which it is a compact metric space).

We start by recalling the following result
of \cite[Lemma 6.2]{p-Fe}. Consider 
a bounded set $Y$  in $C_b(\R)$ which is
positively invariant for \eqref{eq:1}, meaning that
$\bar u_0\in Y$ implies that $ u(\cdot,t,\bar u_0)\in Y$
for all $t>0$.
\begin{lemma}
	\label{le:cont}
	Let   $Y$ be a bounded set in $C_b(\R)$
	which is positively invariant for \eqref{eq:1}.
	Equip $Y$ with a metric from $L^\infty_{loc}(\R)$.   Given any $T>0$,
	there is $L=L(T)\in (0,\infty)$ such that for each  $t\in (0,T)$
	the map $\bar u_0\mapsto u(\cdot,t,\bar u_0):Y\to
	Y$ is Lipschitz continuous with Lipschitz constant $L$. 
\end{lemma}

We now consider the solution flow on 
$\omega(u_0)$. For any $t\in \R$, let $G(t):\om(u_0)\to\om(u_0)$ be defined
by $G(t)\vp=U(\cdot,t)$,  where $U(\cdot,t)$ is the entire solution
of \eqref{eq:1} with $U(\cdot,0)=\vp$. As noted in Section
\ref{sec:22}, this entire solution is well (and uniquely) defined
and satisfies $U(\cdot,t)\in\om(u_0)$ for all $t\in \R$.
We claim that the family $G(t)$, $t\in\R$, defines a flow on
$\om(u_0)$, that is,
\begin{itemize}
	\item[(i)] $G(0)$ is the identity on $\om(u_0)$,
	\item[(ii)] $G(t+ s) = G(t)G(s)$ \quad ($s, t\in \R)$, 
	\item[(iii)] for each $t_0 \in \R$, the map $G(t_0)$ is continuous.   
\end{itemize}
Property (i) is  obvious. The group property (ii) follows
from the uniqueness of $U$  and the time-translation invariance of
\eqref{eq:1}. The continuity of $G(t_0)$ for
$t_0 > 0$ follows from Lemma \ref{le:cont} applied to
$Y=\om(u_0)$. Let now  $t_0<0$. 
Properties (i) and (ii)  imply that $G(t_0)$ is the inverse to the
continuous map $G(-t_0)$.  Since $\om(u_0)$ is compact, the inverse is
continuous, too.

Obviously, for any fixed $\varphi$, the map $t\mapsto
G(t)\varphi:\R\to\om(u_0)$ is continuous. In fact, the map 
$(\varphi,t)\mapsto G(t)\varphi:\om(u_0)\times\R\to\om(u_0)$
is (jointly) continuous. This can be proved  easily using Lemma
\ref{le:cont}, but the fact that the joint continuity follows from
the separate continuity in $t$ and
$\varphi$ is 
a general property of flows
(see \cite[Section 8A]{Marsden-M}).

Next we note that 
the flow $G(t)$, $t\in\R$, on $\om(u_0)$ is chain recurrent in the
following sense. Let $d$ be a metric on $\om(u_0)$
compatible with the topology of $L^\infty_{loc}(\R)$.
For any $\varphi \in \omega (u_0)$, $\varepsilon > 0$, $T > 0$,
there exist
an integer $k \geq 1$, real numbers 
$t_1, \cdots, t_k \geq T$,
and elements
$\varphi_0, \cdots, \varphi_k \in \omega (u_0)$ with $\varphi_0 =
\varphi=\varphi_k$ such that

\begin{equation*}
d(G(t_{i + 1}) \varphi_i,\varphi_{i + 1}) < \varepsilon \qquad (0 \leq i < k).
\end{equation*}
This chain recurrence property of the $\om$-limit set of  solutions with
compact orbits is well-known from 
\cite[Sect. II.6.3]{Conley_78}, where it is proved for flows on
compact metric spaces. For semiflows, including those generated by
parabolic equations, proofs can be found in 
\cite[Lemma~7.5]{Chen_Polacik1995}, \cite[Proposition 1.5]{Mischaikow-S-T},
\cite[Lemma~4.5]{p-Fo:conv-asympt}. Of course, the continuity result
\ref{le:cont} is needed here, as the limit set $\om(u_0)$ is taken
with respect to the topology of $L^\infty_{loc}(\R)$.

Finally, we recall that
a Morse decomposition for $G$ is a system $\cM_1,\dots,\cM_k$ of
mutually disjoint
compact subsets of $\om(u_0)$ with the following properties:
\begin{itemize}
	\item[(mi)] For $j=1,\dots,k$, the set $\cM_j$
	is invariant under $G$: $G(t)\varphi\in \cM_j$ for all
	$\varphi\in\cM_j$ and $t\in\R$.
	\item[(mii)]  If $\varphi\in
	\om(u_0)\setminus \cup_{j=1,\dots,k} \cM_j$ and 
	$U(\cdot,t)=G(t)\varphi$ is the corresponding entire solution, then
	for some $i,j\in \{1,\dots,k\}$ with $i<j$ one has
	$\al(U)\subset \cM_i$ and $\om(U)\subset \cM_j$.
\end{itemize}
(Note that in our
definition of $\al(U)$, $\om(U)$, we use
the convergence in the topology of
$L^\infty_{loc}(\R)$, and the same topology is used on 
$\om(u_0)$.) The following result of \cite[Theorem II.7.A]{Conley_78}
will be instrumental below. The chain recurrence property of the flow
$G$ implies that if  $\cM_1,\dots,\cM_k$ is a Morse decomposition for $G$,
then
\begin{equation}
\label{eq:39}
\om(u_0)\subset \bigcup_{j=1,\dots,k} \cM_j.
\end{equation}

In the proofs of our theorems, we build Morse decompositions for $G$
using chains of \eqref{eq:sys}.
Consider
the system 
\begin{equation}\label{eq:4.2}
\Sigma_j,\ j=1,\ldots k,
\end{equation}
of all chains $\Sigma$ of \eqref{eq:sys} with the property that
$\Sigma\cap\tau(\om(u_0))\ne\emptyset$ (as
noted in Section \ref{sec:22}, conditions (ND), (MF) imply that there are
only finitely many chains).

Given any two  distinct
chains $\Sigma$, $\tilde
\Sigma$, we have, according to Lemma \ref{le:insert}(ii),
that either
$\Sigma\subset \mathcal{I}(\tilde \Sigma)$, or
$\tilde \Sigma\subset \mathcal{I}( \Sigma)$, or else
there are periodic orbits $\cO_1$, $\cO_2$ of \eqref{eq:sys}
such that 
$\ol{\cI}(\cO_1)\cap\ol{\cI}( \cO_2)=\emptyset$ and 
$\Sigma\subset \mathcal{I}(\cO)$,
$\tilde \Sigma\subset \mathcal{I}(\tilde\cO)$.
For chains $\Sigma$, $\tilde \Sigma$
intersecting $\tau(\om(u_0))$, the last possibility is
ruled out by Lemma \ref{le:3.5}(i). 
Thus,  relabelling the chains in \eqref{eq:chains} if necessary,
we may assume that 
\begin{equation}\label{eq:4.3}
\Sigma_j\subset\cI(\Sigma_{j+1}),\ j=1,\ldots k-1.
\end{equation}

We will utilize Morse decompositions with Morse sets of the form
\begin{equation}
\label{eq:49}
\{\varphi\in\om(u_0):\tau(\varphi)\subset \Sigma\},
\end{equation}
or
\begin{equation}
\label{eq:49a}
\{\varphi\in\om(u_0):\tau(\varphi)\subset \ol{\cI}(\Sigma)\},
\end{equation}
where $\Sigma$ is one of the chains \eqref{eq:chains}. Let us prove
first of all that these are compact subsets of $\om(u_0)$.
\begin{lemma}
	\label{le:compact}
	If $\Sigma$ is any  of the chains \eqref{eq:chains}, then the sets
	\eqref{eq:49}, \eqref{eq:49a} are closed subsets of $\om(u_0)$. 
\end{lemma}
\begin{proof}
	We prove the result for \eqref{eq:49}; the proof for \eqref{eq:49a}
	is similar and is omitted. Assume that $\varphi_n$, $n=1,2,\dots$
	belong to the set   \eqref{eq:49} and $\varphi_n\to\varphi$ in
	$\om(u_0)$. This means, a priori, that $\varphi_n\to\varphi$ in
	$L^\infty_{loc}(\R)$, but since $\om(u_0)$ is compact in
	$C^1_{loc}(\R)$ (cp. Section \ref{sec:inv}), we  also have
	$\varphi_n\to\varphi$ in
	$C^1_{loc}(\R)$. Pick 
	any $x\in\R$. Then $(\varphi_n(x),\varphi_n'(x))\to
	(\varphi(x),\varphi'(x))$. Since the set $ \Sigma$ is
	obviously closed in $\R^2$ and $(\varphi_n(x),\varphi_n'(x))\in
	\tau(\varphi_n)\subset \Sigma$, we obtain that
	$(\varphi(x),\varphi'(x))\in\Sigma$. Since $x\in\R$ was arbitrary,
	we have proved that $\varphi$ belongs to the set \eqref{eq:49}.
\end{proof}

We are ready to complete the proofs of our main theorems.
In proving the quasiconvergence results,
Theorems \ref{thm:1} and  \ref{thm:3}, our goal is to show
that there is a chain $\Sigma$ of
\eqref{eq:sys} such that 
\begin{equation}
\label{eq:final}
\tau(\om(u_0)\subset \Sigma.
\end{equation}
This inclusion implies, 
by Lemma \ref{le:2.7}, 
that $\om(u_0)$ consists of steady states of
\eqref{eq:1}, and also gives an additional information that
the spatial trajectories of the
functions in $\om(u_0)$ are all contained in one chain.

\begin{proof}[Proof of Theorems \ref{thm:1}, \ref{thm:3}]
	We use the chains in \eqref{eq:4.2} to define the following sets
	\begin{equation}
	\label{eq:morse}
	\cM_j:=\{\varphi\in\om(u_0):\tau(\varphi)\subset
	\Sigma_j\},\quad j=1,\dots,k. 
	\end{equation}
	They are obviously mutually disjoint---as the chains \eqref{eq:4.2}
	are such, cp. \eqref{eq:4.3}---and by Lemma \ref{le:compact} they are
	compact in $\om(u_0)$. Since the sets $\cM_j$
	consist of steady states (cp. Lemma
	\ref{le:2.7}), they are invariant for $G$. Take now an arbitrary 
	$\varphi\in
	\om(u_0)\setminus \cup_{j=1,\dots,k} \cM_j$, if there is any such
	$\varphi$, and let    
	$U(\cdot,t)=G(t)\varphi$ be the corresponding entire solution.
	By the definition of the sets \eqref{eq:morse} and \eqref{eq:4.2},
	$\tau(\vp)$ is not contained in any chain. Therefore,
	Proposition \ref{prop:3.3} tells us that---under the hypotheses of Theorem \ref{thm:1} or Theorem \ref{thm:3}--- there are
	two chains $\Sii$, $\Sigma_{out}$ such that
	$\Sii\subset\cI(\Sigma_{out})$ and
	\begin{equation}
	\label{eq:6b}
	\tau\left(\alpha(U)\right)\subset\Sigma_{in},\qquad \tau\left(\omega(U)\right)\subset\Sigma_{out}
	\end{equation}
	(for $\Sigma_{out}$, we take the chain containing the loop $\Lao$, with
	$\Lao$ as in \eqref{eq:6}). 
	Since $\al(U),\om(U)\subset \om(u_0)$, the inclusions
	\eqref{eq:6b} imply that $\Sii=\Sigma_\ell$, $\Sigma_{out}=\Sigma_j$ for
	some $\ell,j\in\{1,\dots,k\}$; and $\Sii\subset\cI(\Sigma_{out})$ implies
	that $\ell<j$. We have thus proved that
	$\cM_1,\dots,\cM_k$ is a Morse decomposition for
	the flow $G$. From \eqref{eq:39} and the connectedness of $\om(u_0)$
	we now conclude that $k=1$, that is, there is only one chain in
	\eqref{eq:4.2} and \eqref{eq:final} holds for that chain, as desired. 
\end{proof}

\begin{remark}
	\label{rm:onR}{\rm
		Hypotheses  (R) of Theorem \ref{thm:3} is only needed in the
		proof of Proposition \ref{prop:3.3} in the case that (U) holds and
		$\Pi=\Pi_0$ is the connected
		component of $\mathcal{P}_0$ whose closure contains $(0,0)$
		(see Section \ref{completionprop}).
		If this part of 
		Proposition \ref{prop:3.3} could be proved under weaker or no
		conditions in place of (R), then the above proof would work without
		change. 
	}
\end{remark}

We next state and prove a  stronger
version of Theorem \ref{thm:2}.
Recall from Section \ref{sec:3} that if  (U) holds, Proposition \ref{prop:3.3} holds true for any 
connected component $\Pi$ of $\mathcal{P}_0$ distinct from $\Pi_0.$
. 

\begin{theorem}
	\label{thm:4}
	Assume that {\rm (U)} holds.   If 
	$\varphi\in \om(u_0)$ is such that $\tau(\vp)$ is not contained in
	$\Pi_0\cup\{(0,0)\}$, then  $\varphi$ is a steady
	state of \eqref{eq:1}.  
\end{theorem}
Obviously, Theorem \ref{thm:2} follows from this result.

\begin{proof}[Proof of Theorem \ref{thm:4}]
	Let $\La_0:=\Lao(\Pi_0)$ be the outer loop associated with $\Pi_0$
	and
	let $\Sigma_0$ the the chain containing the loop $\La_0$. Note that
	$\cI(\La_0)=\Pi_0\cup \{(0,0)\}$.
	
	We first consider the possibility that
	\begin{equation}
	\label{eq:50} \tau(\om(u_0))\cap\ol\cI(\Sigma_0)=\emptyset.
	\end{equation}
	This in  particular implies that
	$\Sigma_j\cap\ol\cI(\Sigma_0)=\emptyset$ for $j=1,\dots,k$. In this
	situation, one can almost verbatim repeat the previous proof to
	conclude that $k=1$ and  $\om(u_0)\subset \cM_1$ consists of steady
	states. We only remark that if $\varphi\in
	\om(u_0)\setminus \cup_{j=1,\dots,k} \cM_j$, then, due to condition
	\eqref{eq:50},  one has $\tau(\varphi)\cap\Pi_0=\emptyset$. 
	Thus  Proposition \ref{prop:3.3} 
	applies to $\varphi$.

	Now assume that 
	\begin{equation}
	\label{eq:50n} \tau(\om(u_0))\cap\ol\cI(\Sigma_0)\ne\emptyset.
	\end{equation}
	We distinguish the following two possibilities: 
	\begin{align}
	\label{eq:51}
	\tau(\om(u_0))&\not\subset\ol\cI(\Sigma_0),\\
	\label{eq:52}
	\tau(\om(u_0))&\subset\ol\cI(\Sigma_0).
	\end{align}
	
	The first one, \eqref{eq:51}, can actually be ruled out. Indeed, if 
	\eqref{eq:51} holds, Proposition \ref{prop:3.3} ensures that at least one of the chains
	\eqref{eq:4.2} is disjoint from  $\ol\cI(\Sigma_0)$. Denoting
	by $m$ the number of such chains, we list those 
	chains as
	\begin{equation}
	\label{eq:chainst}
	\tilde \Sigma_1,\dots, \tilde \Sigma_m. 
	\end{equation}
	Here,  the labels are chosen such that
	$\tilde\Sigma_i\subset\cI(\tilde\Sigma_{i+1})$ for $ i=1,\ldots m-1$
	(cp. \eqref{eq:4.3}).
	
	Consider the following subsets of $\om(u_0)$:
	\begin{equation}
	\label{eq:morse2}
	\begin{aligned}
	\cM_0&:=  \{\varphi\in\om(u_0):\tau(\varphi)\subset\ol\cI(\Sigma_0)
	\},\\
	\cM_j&:=\{\varphi\in\om(u_0):\tau(\varphi)\subset
	\tilde \Sigma_j\}\quad(j=1,\dots,m).
	\end{aligned}
	\end{equation}
	All these sets are nonempty by \eqref{eq:51} and the definition of
	the sets $\tilde \Sigma_j$.
	We claim that these sets constitute a Morse decomposition on $\om(u_0)$.
	Clearly, the sets are mutually disjoint. By Lemma \ref{le:compact}, they are
	compact in $\om(u_0)$. The sets $\cM_j$, $j=1,\dots,m$,
	consist of steady states (cp. Lemma
	\ref{le:2.7}), hence they  are invariant for $G$.
	To prove the invariance of $\cM_0$, take any
	$\varphi\in\om(u_0)$ with $\tau(\varphi)\subset\ol\cI(\Sigma_0)$
	and let $U(\cdot,t)=G(t)\varphi$ be the corresponding entire solution.
	If $\tau(\varphi)\cap\cI(\Sigma_0)\ne \emptyset$, then, by
	Lemma \ref{le:3.4},  $\tau(U(\cdot,t))\subset\cI(\Sigma_0)$---in particular
	$U(\codt,t)\in \cM_0$---for all $t\in \R$. Otherwise,
	$\tau(\varphi)\subset \Sigma_0$ and $\varphi$ is a steady state, so
	$U(\codt,t)\in \cM_0$ holds trivially. Thus, the invariance of $\cM_0$  
	is proved.

	Take now an arbitrary 
	$\varphi\in
	\om(u_0)\setminus \cup_{j=1,\dots,k} \cM_j$, if there is any such
	$\varphi$, and let    
	$U(\cdot,t)=G(t)\varphi$ be the corresponding entire solution.
	We have $\tau(\varphi)\cap\ol\cI(\Sigma_0)=\emptyset$ and 
	$\tau(\vp)$ is not contained in any chain. Applying 
	Proposition \ref{prop:3.3} (with $\Pi\ne\Pi_0$), we obtain
	similarly as in the proof of Theorems \ref{thm:1}, \ref{thm:3}, that
	there are 
	two chains $\Sii$, $\Sigma_{out}$ such that
	$\Sii\subset\cI(\Sigma_{out})$ and
	\begin{equation}
	\label{eq:6bb}
	\tau\left(\alpha(U)\right)\subset\Sigma_{in},\qquad \tau\left(\omega(U)\right)\subset\Sigma_{out}.
	\end{equation}
	Arguing similarly as in the proof of Theorems \ref{thm:1},
	\ref{thm:3}, we obtain from \eqref{eq:6bb} that
	$\cM_0,\dots,\cM_k$ is a Morse decomposition for
	the flow $G$, and so  \eqref{eq:39}  holds. This time, however, since
	there are at least two Morse sets in \eqref{eq:morse2},
	from \eqref{eq:39} we obtain a contradiction to
	the connectedness of $\om(u_0)$.
	
	We have thus ruled  \eqref{eq:51},  so \eqref{eq:52} has to hold.
	Take now any $\varphi\in \om(u_0)$ such that
	$\tau(\vp)$ is not contained in
	$\Pi_0\cup\{(0,0)\}=\cI(\La_0)$. Note that, due to Lemma
	\ref{le:3.4}, $\tau(\varphi)\cap\Pi_0=\emptyset$. 
	We claim that  $\tau(\varphi)\subset \Sigma_0$,
	in particular, by Lemma
	\ref{le:2.7},  $\varphi$ is a steady state of \eqref{eq:1}.
	Once this claim is proved,  the proof of Theorem \ref{thm:4} will be
	complete.
	
	Suppose our claim is not true.  Then
	$\tau(\varphi)\cap\cI(\La)\ne\emptyset$, where $\La$ is a loop in
	$\Sigma_0$ different from $\La_0$.
	We can therefore find a periodic orbit $\cO$ of \eqref{eq:sys}
	such that $\cO\subset \cI(\La)$ and
	$\tau(\varphi)\cap\ol\cI(\cO)\ne\emptyset$. Since $\La\ne\La_0$, we
	have $\{(0,0)\}\not\in \cI(\cO)$.  Lemma \ref{le:3.5} now implies that
	that $\vp\not\in\om(u_0)$ and we have a
	contradiction. 
\end{proof}

\bibliographystyle{amsplain} 

\begin{thebibliography}{10}
	
	\bibitem{Angenent_Crelle88}
	S.~Angenent.
	\newblock The zero set of a solution of a parabolic equation.
	\newblock {\em J. Reine Angew. Math.}, 390:79--96, 1988.
	
	\bibitem{p-B-Q}
	T.~Bartsch, P.~Pol\'a\v{c}ik, and P.~Quittner.
	\newblock Liouville-type theorems and asymptotic behavior of nodal radial
	solutions of semilinear heat equations.
	\newblock {\em J. Eur. Math. Soc.}, 13:219--247, 2011.
	
	\bibitem{Brunovsky-F:conn}
	P.~Brunovsk{\'y} and B.~Fiedler.
	\newblock Connecting orbits in scalar reaction diffusion equations. {II}. {T}he
	complete solution.
	\newblock {\em J. Differential Equations}, 81:106--135, 1989.
	
	\bibitem{Chen-G:ex}
	X.~Chen and J.-S. Guo.
	\newblock Existence and uniqueness of entire solutions for a reaction-diffusion
	equation.
	\newblock {\em J. Differential Equations}, 212:62--84, 2005.
	
	\bibitem{Chen-G-N}
	X.~Chen, J.-S. Guo, and H.~Ninomiya.
	\newblock Entire solutions of reaction-diffusion equations with balanced
	bistable nonlinearities.
	\newblock {\em Proc. Roy. Soc. Edinburgh Sect. A}, 136:1207--1237, 2006.
	
	\bibitem{Chen_MathAnn98}
	X.-Y. Chen.
	\newblock A strong unique continuation theorem for parabolic equations.
	\newblock {\em Math. Ann.}, 311(4):603--630, 1998.
	
	\bibitem{Chen_Polacik1995}
	X.-Y. Chen and P.~Pol\'a\v{c}ik.
	\newblock Gradient-like structure and Morse decompositions for time-periodic
	one-dimensional parabolic equations.
	\newblock {\em J. Dynam. Differential Equations}, 7(1):73--107, 1995.
	
	\bibitem{Chen-G-N-Y}
	Y.-Y. Chen, J.-S. Guo, H.~Ninomiya, and C.-H. Yao.
	\newblock Entire solutions originating from monotone fronts to the
	{A}llen-{C}ahn equation.
	\newblock {\em Phys. D}, 378/379:1--19, 2018.
	
	\bibitem{Conley_78}
	C.~Conley.
	\newblock {\em Isolated invariant sets and the {M}orse index}, volume~38 of
	{\em CBMS Regional Conference Series in Mathematics}.
	\newblock American Mathematical Society, Providence, R.I., 1978.
	
	\bibitem{Du_Matano}
	Y.~Du and H.~Matano.
	\newblock Convergence and sharp thresholds for propagation in nonlinear
	diffusion problems.
	\newblock {\em J. Eur. Math. Soc. (JEMS)}, 12(2):279--312, 2010.
	
	\bibitem{p-Du}
	Y.~Du and P.~Pol\'a\v{c}ik.
	\newblock Locally uniform convergence to an equilibrium for nonlinear parabolic
	equations on {$\mathbb R^N$}.
	\newblock {\em Indiana Univ. Math. J.}, 64:787--824, 2015.
	
	\bibitem{Feireisl_NoDEA97}
	E.~Feireisl.
	\newblock On the long time behaviour of solutions to nonlinear diffusion
	equations on {$\mathbb R^n$}.
	\newblock {\em NoDEA Nonlinear Differential Equations Appl.}, 4(1):43--60,
	1997.
	
	\bibitem{p-Fe}
	E.~Feireisl and P.~Pol\'a\v{c}ik.
	\newblock Structure of periodic solutions and asymptotic behavior for
	time-periodic reaction-diffusion equations on {$\R$}.
	\newblock {\em Adv. Differential Equations}, 5:583--622, 2000.
	
	\bibitem{Fiedler-B:neumann}
	B.~Fiedler and P.~Brunovsk{\'y}.
	\newblock Connections in scalar reaction diffusion equations with {N}eumann
	boundary conditions.
	\newblock In {\em Equadiff 6 ({B}rno, 1985)}, volume 1192 of {\em Lecture Notes
		in Math.}, pages 123--128. Springer, Berlin, 1986.
	
	\bibitem{Fiedler-R:conn}
	B.~Fiedler and C.~Rocha.
	\newblock Heteroclinic orbits of semilinear parabolic equations.
	\newblock {\em J. Differential Equations}, 125:239--281, 1996.
	
	\bibitem{FMcL}
	P.~C. Fife and J.~B. McLeod.
	\newblock The approach of solutions of nonlinear diffusion equations to
	travelling front solutions.
	\newblock {\em Arch. Ration. Mech. Anal.}, 65(4):335--361, 1977.
	
	\bibitem{p-Fo:conv-asympt}
	J.~F\"oldes and P.~Pol\'a\v{c}ik.
	\newblock Convergence to a steady state for asymptotically autonomous
	semilinear heat equations on {$\mathbb R^{N}$}.
	\newblock {\em J. Differential Equations}, 251:1903--–1922, 2011.
	
	\bibitem{Gallay-S}
	T.~Gallay and S.~Slijep\v{c}evi\'c.
	\newblock Energy flow in extended gradient partial differential equations.
	\newblock {\em J. Dynam. Differential Equations}, 13:757--789, 2001.
	
	\bibitem{Gallay-S2}
	T.~Gallay and S.~Slijep\v{c}evi\'c.
	\newblock Distribution of energy and convergence to equilibria in extended
	dissipative systems.
	\newblock {\em J. Dynam. Differential Equations}, 27:653--682, 2015.
	
	\bibitem{Guo-M-ent}
	J.-S. Guo and Y.~Morita.
	\newblock Entire solutions of reaction-diffusion equations and an application
	to discrete diffusive equations.
	\newblock {\em Discrete Contin. Dynam. Systems}, 12:193--212, 2005.
	
	\bibitem{Hamel-N-ent}
	F.~Hamel and N.~Nadirashvili.
	\newblock Entire solutions of the {KPP} equation.
	\newblock {\em Comm. Pure Appl. Math.}, 52(10):1255--1276, 1999.
	
	\bibitem{Marsden-M}
	J.~E. Marsden and M.~McCracken.
	\newblock {\em The {H}opf bifurcation and its applications}.
	\newblock Springer-Verlag, New York, 1976.
	\newblock With contributions by P. Chernoff, G. Childs, S. Chow, J. R. Dorroh,
	J. Guckenheimer, L. Howard, N. Kopell, O. Lanford, J. Mallet-Paret, G. Oster,
	O. Ruiz, S. Schecter, D. Schmidt and S. Smale, Applied Mathematical Sciences,
	Vol. 19.
	
	\bibitem{Matano_Polacik_CPDE16}
	H.~Matano and P.~Pol\'a\v{c}ik.
	\newblock Dynamics of nonnegative solutions of one-dimensional
	reaction-diffusion equations with localized initial data. {P}art {I}: {A}
	general quasiconvergence theorem and its consequences.
	\newblock {\em Comm. Partial Differential Equations}, 41(5):785--811, 2016.
	
	\bibitem{p-Ma:entire}
	H.~Matano and P.~Pol\'a\v{c}ik.
	\newblock An entire solution of a bistable parabolic equation
        on
        {$\mathbb{R}$}
	with two colliding pulses.
	\newblock {\em J. Funct. Anal.}, 272:1956--1979, 2017.
	
	\bibitem{p-Ma:1d-p2}
	H.~Matano and P.~Pol\'a\v{c}ik.
	\newblock Dynamics of nonnegative solutions of one-dimensional
	reaction-diffusion equations with localized initial data. {P}art {II}: The
	generic case.
	\newblock {\em Comm. Partial Differential Equations}, to appear.
	
	\bibitem{Mischaikow-S-T}
	K.~Mischaikow, H.~Smith, and H.~R. Thieme.
	\newblock Asymptotically autonomous semiflows: chain recurrence and {L}yapunov
	functions.
	\newblock {\em Trans. Amer. Math. Soc.}, 347(5):1669--1685, 1995.
	
	\bibitem{Morita-N:entire}
	Y.~Morita and H.~Ninomiya.
	\newblock Entire solutions with merging fronts to reaction-diffusion equations.
	\newblock {\em J. Dynam. Differential Equations}, 18:841--861, 2006.
	
	\bibitem{Morita-N:exposition}
	Y.~Morita and H.~Ninomiya.
	\newblock Traveling wave solutions and entire solutions to reaction-diffusion
	equations.
	\newblock {\em Sugaku Expositions}, 23:213--233, 2010.
	
	\bibitem{Pauthier_Polacik1}
	A.~Pauthier and P.~Pol\'{a}\v{c}ik.
	\newblock Large-time behavior of solutions of parabolic equations on the real
	line with convergent initial data.
	\newblock {\em Nonlinearity}, 31(9):4423--4441, 2018.
	
	\bibitem{P:examples}
	P.~Pol\'a\v{c}ik.
	\newblock Examples of bounded solutions with nonstationary limit profiles for
	semilinear heat equations on {$\mathbb R$}.
	\newblock {\em J. Evol. Equ.}, 15:281--307, 2015.
	
	\bibitem{P:ctw}
	P.~Pol\'a\v{c}ik.
	\newblock Spatial trajectories and convergence to traveling fronts for bistable
	reaction-diffusion equations.
	\newblock In A.~Carvalho, B.~Ruf, E.~Moreira~dos Santos, S.~Soares, and
	T.~Cazenave, editors, {\em Contributions to nonlinear differential equations
		and systems, A tribute to Djairo Guedes de Figueiredo on the occasion of his
		80th Birthday}, Progress in Nonlinear Differential Equations and Their
	Applications, pages 405--423. Springer, 2015.
	
	\bibitem{P:unbal}
	P.~Pol\'a\v{c}ik.
	\newblock Threshold behavior and non-quasiconvergent solutions with localized
	initial data for bistable reaction-diffusion equations.
	\newblock {\em J. Dynamics Differential Equations}, 28:605--625, 2016.
	
	\bibitem{P:quasiconv-overview}
	P.~Pol\'a\v{c}ik.
	\newblock Convergence and quasiconvergence properties of solutions of parabolic
	equations on the real line: an overview.
	\newblock In {\em Patterns of dynamics}, volume 205 of {\em Springer Proc.
		Math. Stat.}, pages 172--183. Springer, Cham, 2017.
	
	\bibitem{P:entire}
	P.~Pol\'a\v{c}ik.
	\newblock Entire solutions and a {L}iouville theorem for a class of parabolic
	equations on the real line.
	\newblock {\em Proc. Amer. Math. Soc.}, to appear.
	
	\bibitem{Polacik_terrasse}
	P.~Pol\'a\v{c}ik.
	\newblock Propagating terraces and the dynamics of front-like solutions of
	reaction-diffusion equations on {$\mathbb{R}$}.
	\newblock {\em Mem. Amer. Math. Soc.}, to appear.
	
	\bibitem{Risler_1}
	E.~Risler.
	\newblock Global relaxation of bistable solutions for gradient systems in one
	unbounded spatial dimension.
	\newblock {\em preprint}.
	
	\bibitem{Risler:terrace-1d}
	E.~Risler.
	\newblock Global behaviour of bistable solutions for gradient systems in one
	unbounded spatial dimension.
	\newblock {\em preprint}.
	
	\bibitem{Volpert3}
	A.~I. Volpert, V.~A. Volpert, and V.~A. Volpert.
	\newblock {\em Traveling wave solutions of parabolic systems}, volume 140 of
	{\em Translations of Mathematical Monographs}.
	\newblock American Mathematical Society, Providence, RI, 1994.
	\newblock Translated from the Russian manuscript by James F. Heyda.
	
	\bibitem{Wolfrum:jde}
	M.~Wolfrum.
	\newblock A sequence of order relations: encoding heteroclinic connections in
	scalar parabolic {PDE}.
	\newblock {\em J. Differential Equations}, 183:56--78, 2002.
	
\end{thebibliography}

\end{document}